\documentclass[11pt, reqno]{amsart}

\usepackage{caption, amssymb,amsfonts,amsmath, graphicx, geometry}
\usepackage[nice]{nicefrac}
\usepackage{mathrsfs}
\usepackage[colorlinks=true,bookmarksnumbered=true, pdfauthor={Jacob, M\"orters}]{hyperref}

\newcommand{\cc}{\mathfrak{c}}
\newcommand{\eps}{\varepsilon}
\newcommand{\ssup}[1] {{\scriptscriptstyle{({#1}})}}
\newcommand{\one}{{\mathsf 1}}

\newcommand{\R}{\mathbb R}

\newcommand{\E}{\mathbb E}
\renewcommand{\P}{\mathbb P}

\newcommand{\ut}{\mathfrak{t}}

\renewcommand{\phi}{\varphi}

\newcommand{\N}{\mathbb N}

\renewcommand{\P}{\mathbb P}
\newcommand{\vertiii}[1]{{\left\vert\kern-0.25ex\left\vert\kern-0.25ex\left\vert #1 
    \right\vert\kern-0.25ex\right\vert\kern-0.25ex\right\vert}}

\newcommand{\Y}{\mathcal Y}

\newcommand{\de}{\mathrm{d}} 

\newcommand{\F}{\mathfrak{F}}

\newcommand{\U}{\mathcal U}
\newcommand{\Rec}{\mathcal R}
\newcommand{\I}{\mathcal I}
\newcommand{\St}{\mathscr S}
\newcommand{\Co}{\mathscr C}
\newcommand{\Vo}{\mathscr V}

\newcommand{\heap}[2]  {\genfrac{}{}{0pt}{}{#1}{#2}}
\newcommand{\sfrac}[2] {\mbox{$\frac{#1}{#2}$}}

\newtheorem{theorem}{Theorem}
\newtheorem{definition}{Definition}

\newtheorem{lemma}{Lemma}
\newtheorem{proposition}{Proposition}
\newtheorem{remark}{Remark}

\begin{document}

\title[The contact process on dynamical scale-free networks]
{The contact process on dynamical scale-free networks} 

\author[Emmanuel Jacob, Amitai Linker and Peter M\"orters]{Emmanuel Jacob, Amitai Linker and Peter M\"orters}

\maketitle

\vspace{-0.7cm}

\vspace{0.1cm}

\begin{quote}
{\small {\bf Abstract:}} We investigate the contact process on four different types of scale-free inhomogeneous random graphs evolving according to a stationary dynamics, where each potential edge is updated with a rate depending on the strength of the adjacent vertices. Depending on the type of graph, the tail exponent of the degree distribution and the 
updating rate, we find parameter regimes of fast and slow extinction and in the latter case identify 
metastable exponents that undergo first order phase transitions. 
\end{quote}
\smallskip

\begin{quote}
{\small {\bf R\'esum\'e:}}
Nous \'etudions le processus de contact sur quatre types diff\'erents de graphes al\'eatoires inhomog\`{e}nes invariants d'\'echelle \'evoluant selon une dyna\-mique stationnaire, o\`{u} chaque ar\^{e}te potentielle est rafra\^ichie \`{a} un taux d\'ependant de la force des sommets adjacents. En fonction du type de graphe, de l'exposant de la queue de distribution des degr\'es et du taux de rafra\^ichissement, nous trouvons des r\'egimes d'extinction rapide ou lente et, dans ce dernier cas, nous identifions des exposants m\'eta\-stables qui subissent des transitions de  phase de premier ordre.
\end{quote}

\vspace{0.3cm}

\noindent\emph{MSc Classification:} Primary 05C82; Secondary 82C22.


\pagebreak[3]

\section{Introduction}

Diffusion processes modelling the spread of information or disease are often sensitive to 
spatial inhomogeneities or temporal variation  in the surrounding medium. A paradigmatic example is the contact process, in which every vertex
of a finite graph can either be infected (occupied) or healthy (empty). The Markovian dynamics of this process
evolves in continuous time, every infected neighbour infects each of its healthy neighbours with rate $\lambda>0$ and recovers to the healthy state with rate one. Because recovered vertices are again susceptible to the infection this process is also known as the SIS infection in the epidemics literature. Note that the contact process has exactly one absorbing state, when every vertex is healthy, and that this state will be reached in finite time. This random time is called the \emph{extinction time} of the infection.
\medskip
\pagebreak[3]

On reasonably regular graphs we expect the contact process to show a phase transition in the infection rate~$\lambda$. There is $\lambda_c>0$ such that for $0<\lambda<\lambda_c$ the expected extinction time is bounded by the logarithm of the number of vertices of the graph. For $\lambda>\lambda_c$ however if there is an outbreak of the infection, with high probability the extinction time is exponential in the graph size. See \cite{DL88, DS88, MV16, CD18} and \cite[I.3]{L99} for results in this direction.\medskip

On scale-free graphs, which are highly inhomogeneous, the extinction time is exponential in the graph 
size for \emph{any $\lambda>0$}, see~\cite{BB+, GMT05, CD09}. This happens because the behaviour is dominated by a small number of vertices with extremely high degree. Indeed, if an infected vertex has degree $k\gg \lambda^{-2}$ then it typically has of order $\lambda k$ infected neighbours. Once it recovers, the probability that none of its neighbours reinfects the vertex within one time unit is roughly~\smash{$e^{-k \lambda^2}$}, which is very small. Hence the vertex can hold (and spread) the infection
for a long time, effectively exponential in $k\lambda^2$. Starting from sufficiently many vertices infected, the contact process therefore settles for a very long time in a state where a small number of vertices with very high degree and  a proportion of their direct neighbours remain infected for most of the time, a \emph{metastable state}, and only after a time exponential in the graph size the system collapses to the absorbing state. This metastable behaviour is characterised by a positive density of infected states at subexponential times, which for $\lambda\downarrow0$ decays like~$\lambda^{\xi + o(1)}$. The  exponent $\xi>0$ is called the \emph{metastable exponent}.
\medskip

The aim of this project is to investigate how  temporal variability in the surrounding medium can change the qualitative behaviour of diffusion processes in that medium using the contact process as an example. Similar problems have been studied recently for regular graphs, for example  in the work of da Silva et al.~\cite{Silva} and Hil\'ario et al.~\cite{Hilario}, but our focus is on scale-free and hence irregular graphs, which feature very different behaviour.  We interpolate between two extreme scenarios, on the one hand the infection on the static, and hence infinitely slowly evolving, scale-free network, on the other hand the mean-field model, which effectively corresponds to an infinitely fast network evolution.
In the mean-field model
infections pass between vertices with a rate given as  $\lambda$ times the {average time} that the edge connecting the vertices exists in the graph. This means, loosely speaking, that whenever the
infection wants to use a potential edge, the existence of this edge is freshly sampled using the stationary probability. The mechanism that on the scale-free graph kept the infection alive at high degree vertices does not work here, as these vertices do not have a fixed neighbourhood with an increased infection density. Hence the infection can only survive for small infection rates if the connectivity among 
high degree vertices is very high, which normally happens when the power-law exponent satisfies $\tau<3$, as was first observed by Pastor-Satorras and Vespignani~\cite{PV01}.
\medskip

To interpolate between the static and mean-field models we run the Markovian dynamics of an evolving graph and of the contact process simultaneously. A natural graph evolution is the updating of all potential edges. This simple evolution is also used in dynamical percolation models~\cite{Steif}, which motivates the name \emph{dynamical scale-free network} used in the title. If vertices in the graph are ranked and the existence of edges are independent events given the ranks, updating the edges with any rate and a fixed connection probability is a stationary dynamics. More precisely, we denote the vertex set by
$\{1,\ldots,N\}$ and initially, as well as at each updating, connect a pair $\{i,j\}$ of distinct vertices independently with a probability $p_{i,j}\wedge 1$.
The idea is that the index of a vertex indicates its rank, so that vertices with small index are strongest, i.e.\ have the largest expected degree.
Interesting choices of connection probability that lead to scale-free networks are, for given $\beta>0$ and $0<\gamma<1$,
\begin{itemize}
	\item the \emph{factor probability} given by
$p_{i,j}= \beta\, N^{2\gamma-1}  i^{-\gamma}j^{-\gamma},$\smallskip
\item the \emph{preferential attachment probability} given by 
$p_{i,j}= \beta\,  (i\wedge j)^{-\gamma} (i \vee j)^{\gamma-1},$\smallskip
\item the \emph{strong probability} given by 
$p_{i,j}= \beta\, N^{\gamma-1} (i\wedge j)^{-\gamma},$\smallskip
\item the \emph{weak probability} given by 
$p_{i,j}= \beta\, N^{\gamma} (i\vee j)^{-\gamma-1}.$\smallskip
\end{itemize}
\pagebreak[3]

While the parameter $\beta>0$ is quantitative and regulates the edge density, the parameter $0<\gamma<1$ significantly influences the quality of the networks by determining the power-law exponent. Indeed, in all cases it is easy to check that the expected degree of vertex~$i$ is of order $k_i=(N/i)^\gamma$ and this leads to the resulting networks being scale-free with power-law exponent $\tau=1+\frac1\gamma >2$. 
The choice of inhomogeneous random graphs is natural for our problem, as they are the invariant distributions under independent updating of edges and hence the updating dynamics is stationary for these random graphs. The connection probabilities correspond to those of the classical types of scale-free networks. The factor probability corresponds to the Chung-Lu or configuration models, the preferential attachment probability to the preferential attachment networks, and the strong, resp.~weak, probabilities represent graphs where only the stronger, resp.~weaker, vertex adjacent to a potential edge determines the probability of a connection. Observe that the first two probabilities agree iff $\gamma=\frac12$.\medskip

We choose the rate of updating of an edge in dependence of a further parameter $\eta\in\R$, so that varying $\eta$ will allow us to interpolate between the static and the mean-field case. 
For this purpose, as the time-scale of the contact process evolution is fixed at order one, one has to couple the time-scale of the network evolution to the network size $N$ in such a way that the update rates of the relevant edges do not degenerate as $N\to\infty$. We achieve this by letting the update rates 
depend on the strength of the vertices adjacent to the edge. More precisely, if $k_i$ is the expected degree of vertex $i$, the update rate of the pair $\{i,j\}$ is 
\smash{$\varkappa (k_i^\eta+k_j^\eta)$}, where $\varkappa>0$ is a fixed constant.  In Remark~\ref{up} below we discuss possible alternative update rates. In this framework increasing $\eta$ speeds up, decreasing $\eta$ slows down the network evolution. At $\eta=0$ the network evolution happens at the same time-scale as the contact process, when $\eta\downarrow-\infty$  we approach the static case, and when~$\eta\uparrow\infty$  the mean-field case.
\medskip

Our main result, Theorem~1, describes, for all four types of connection probabilities, in dependence of the power-law exponent $\tau>2$ and $\eta\in\R$,
\begin{itemize}
\item phases of  \emph{fast extinction}, i.e. for sufficiently small $\lambda>0$ the expected extinction time is bounded by a multiple of a power of $\log N$;
\item in phases of slow extinction the \emph{metastable exponent}. Each metastable exponent characterises an optimal survival strategy for the infection.
\end{itemize}
Theorem~1 is an application of Theorems~\ref{teolower}, \ref{theoslow}, \ref{generalstatic}, and \ref{teoupper_edge}, 
which describe upper and lower bounds for the density of infected sites under general conditions on the connection probabilities, which are sharp for the cases of principal interest but applicable in great generality.%
\medskip

The resulting phase diagrams, see Figure~1 below, show  that there are three relevant survival strategies for  the infection. In all those strategies a small set of strong vertices, the so-called \emph{stars}, is responsible for keeping the infection alive for a long time. The metastable densities agree with the density of infected neighbours of the stars in the graph. In the quick strategies the stars pass the infection between each other either directly, as in the \emph{quick direct strategy}, or via a single common neighbour, that acts as a stepping stone, as in the \emph{quick indirect strategy}. Such an indirect mechanism can in particular be effective for networks that favour edges if just one of the adjacent vertices is powerful.\pagebreak[3]
\medskip

 In the \emph{local survival strategy} each star can keep the infection alive for a certain time by means of immediate reinfection by its neighbours, so that the effective recovery time of a star is much larger than order 
 one (alike vertices with degree $\gg \lambda^2$ in a static network). The cost of enabling such a local survival mechanism is considerable (paid in terms of a small set of stars) and once it is paid there is no additional cost in spreading the infection between stars, either directly or indirectly.
The local survival strategy stops working when the network evolution is too fast, more precisely when $\eta\geq \frac12$, in which case mean-field behaviour kicks in.  In the intermediate range, when $0<\eta<\frac12$, the updating of edges has an adverse effect on the reservoir of infected neighbours, affecting the metastable exponents by reducing the efficiency of local survival, but for 
$\eta\leq 0$ this effect becomes negligible and static behaviour kicks in. Therefore we see a phase transition at $\eta=0$ within the regime of local survival. 
As $\eta\downarrow-\infty$ only quick direct spreading and local survival remain viable strategies, but the preferred strategy still depends on the type of the connection probability.%
\medskip%

Metastable densities have been calculated for \emph{static networks} for a case of factor connection probabilities by Mountford et al~\cite{MVY13} and for preferential attachment probabilities by Van Hao Can~\cite{VHC17}. In \cite{JLM19} we started our project of studying metastability for \emph{evolving} networks by looking at networks evolving by fast updating of all edges adjacent to a vertex \emph{simultaneously}. This leads to a completely different phase diagram compared to Figure~1 and, in particular, the local survival strategy is not present in the strong form described above. The fast and simultaneous updating of edges in the setup of~\cite{JLM19}  enables the use of methods relying on the fast mixing of large parts of the network. These methods are unavailable for significant parts of the proof of Theorem~1. 
In the present paper we therefore focus on these parts and develop new techniques that succeed without such fast mixing assumptions. We omit or only give brief hints for the  parts of the proof of Theorem~1 that follow by straightforward extension of the methods developed in \cite{JLM19}.
\medskip

In an ongoing project~\cite{JLM22} we look at networks with slow simultaneous updating
of all edges adjacent to a vertex. This will give an even richer picture of the effects of  the network evolution on  the contact process  and will lead to exciting new effects for slow network evolutions. It will also require even more sophisticated techniques which however rely in part on the methods developed in the present paper. 
In view of this, some proofs in the present paper are given in slightly greater generality than necessary. We explain our proof techniques and in particular the new techniques developed here  after the precise statement of our results at the end of Section~2.
\pagebreak[3]

\section{Main results}

For $N\in\N$, we consider the inhomogeneous random graph~$\mathscr G^{\ssup N}$, with vertex set~$\{1,\ldots,N\}$, which contains every edge $\{i,j\}$ independently with probability
$$p_{i,j}:=\sfrac{1}{N} \, p\big(\sfrac{i}{N},\sfrac{j}{N}\big) \wedge 1,$$
where the connection probabilities are given in terms of a kernel 
$p\colon (0,1]\times(0,1]\rightarrow (0,\infty)$, which is symmetric, continuous and decreasing in both parameters. We further assume that there is some $\gamma\in(0,1)$ and constants $0<c_1<c_2$ such that for all $a\in(0,1)$,
\begin{equation}\label{condp}
p(a,1) \le\int_0^1 p(a,s) ds<c_2 a^{-\gamma}.
\end{equation}
and 
\begin{equation}\label{condp2a}
c_1 a^{-\gamma} \le p(a,1).
\end{equation}
These properties are satisfied, for any~$\beta>0$, by the four kernels we consider, namely:
\begin{itemize}
	\item the \emph{factor kernel} $p(x,y)=\beta x^{-\gamma} y^{-\gamma}$,\smallskip
	\item the \emph{preferential attachment kernel} $p(x,y)=\beta (x \wedge y)^{-\gamma} (x \vee y)^{\gamma-1},$  \smallskip
	\item the \emph{strong  kernel} $p(x,y)=\beta (x \wedge y)^{-\gamma},$\smallskip
	\item the \emph{weak kernel} $p(x,y)=\beta (x \vee y)^{-\gamma-1}$,\smallskip
\end{itemize}
except the weak kernel which does not satisfy \eqref{condp2a} but only the weaker condition
\begin{equation}\label{condp2b}
c_1 a^{-\gamma} \le \int_0^1 p(a,s) ds.
\end{equation}

Observe that for all the kernels above, if we have $\lim i_N/N=x$ for a sequence $(i_N)$ and $x\in (0,1],$ then for $N$ sufficiently large we have 
$$p_{i_N,j}=\sfrac{1}{N} p(\sfrac{i_N}{N},\sfrac{j}{N}), \mbox{ for any~$j\in\{1,\ldots,N\}$.}$$
As a result the degree of the vertex $i_N$ in~$\mathscr G^{\ssup N}$ converges to a Poisson distribution with parameter \smash{$\int_0^1 p(x,y) dy$}, so its typical degree is of order $x^{-\gamma}$. More globally, the empirical degree distribution of the network converges (in probability) to a limiting degree distribution~$\mu$, which is a mixed Poisson distribution obtained by taking $x$ uniform in $(0,1)$, then a Poisson distribution with parameter \smash{$\int_0^1 p(x,y) dy$}. It is easy to check that \smash{$\mu(k)=k^{-\tau+ o(1)}$} for $k\to\infty$, with $\tau=1+1/\gamma$, and hence the network \smash{$(\mathscr G^{\ssup N})$} is scale-free with
power-law exponent~$\tau>2$.
\medskip
\pagebreak[3]

We now construct evolving networks~$(\mathscr G^{\ssup N}_t \colon t\geq 0)_{N\in\N}$ by updating every unordered pair  $\{i,j\}$  of distinct vertices independently with rate 
\[
\kappa_{i,j}\;=\;\kappa_i+\kappa_j=\varkappa\left(\frac{N}{i}\right)^{\gamma\eta}+\varkappa\left(\frac{N}{j}\right)^{\gamma\eta},
\]
where $\eta\in\R$ and $\varkappa>0$ are fixed. Upon updating, independently of the previous state, an edge between vertices $i$ and $j$
is inserted with probability $p_{i,j}$. Note that, for 
each~$N$, the graph valued process~$(\mathscr G^{\ssup N}_t \colon t\geq 0)$ is stationary with stationary distribution given by~$\mathscr G^{\ssup N}$.%
\medskip%

The Markovian evolution of the graph~$(\mathscr G^{\ssup N}_t \colon t\geq 0)$ and of the contact process
can be superimposed to define our model of the contact process on the evolving  network. When $\eta=0$
network evolution and contact process operate on the same time-scale, if $\eta>0$ the network evolution is faster, if $\eta<0$ it is slower. We start this process with the stationary distribution of the graph and all vertices infected. Just like in the
static case there is a finite, random extinction time~$T_{\rm ext}$ and we say that
\begin{itemize}
\item there is \emph{fast extinction}, if there exists $\lambda_c>0$ such that for all infection rates $0< \lambda< \lambda_c$ the expected extinction time is bounded by some power of $\log N$;\smallskip
\item there is \emph{slow extinction} if, for all $\lambda>0$, there exists $\eps>0$ such that
$T_{\rm ext}\geq  e^{\eps N}$ with high probability.\smallskip
\end{itemize}
Our first interest is in characterising phases of fast or slow extinction. Slow extinction is indicative of metastable behaviour of the process, and in this case our interest focuses on the exponent of decay of the metastable density when $\lambda\downarrow 0$. More precisely, just like in~\cite{JLM19}, we set $X_t(i)=1$ if vertex $i$ is infected at time $t$, and $X_t(i)=0$ otherwise and let
$$I_N(t)=\frac1N \, \E\Big[ \sum_{i=1}^N X_t(i)\Big] 
= \frac1N \sum_{i=1}^N \P_i \big( T_{\rm ext}>t\big),$$
where $\P_i$ refers to the process started with only vertex~$i$ infected and the last equality holds by the
self-duali ty of the process. The contact process is called
\emph{metastable} if there exists $\eps>0$ such that for all sequences  $(t_N)$ going to infinity slower than $e^{\eps N}$, we have
$$ \liminf_{N\to\infty} I_N(t_N)>0,$$
and if $(s_N)$ and $(t_N)$ are both going to infinity slower than $e^{\eps N}$, we have 
$$I_N(s_N)-I_N(t_N) \underset{N\to \infty} \longrightarrow 0.$$
In that case, the \emph{lower metastable density} $\rho^-(\lambda)=\liminf I_N(t_N)>0$ and the \emph{upper metastable density} $\rho^+(\lambda)  =\limsup I_N(t_N)$ are well-defined and we say that $\xi$ is the \emph{metastable exponent}  if
$$\xi= \lim_{\lambda\downarrow 0}\frac{\log \rho^-(\lambda)}{\log \lambda} = \lim_{\lambda\downarrow 0}\frac{\log \rho^+(\lambda)}{\log \lambda}.$$
We are now ready to state our main result.
\pagebreak[3]

\begin{theorem}
\label{teofinal}\ \\[-4mm]
\begin{itemize}
\item[(a)] Consider the \textbf{factor kernel}. 
\begin{itemize}
\item[(i)] If $\eta\ge \frac 12$ and $\tau>3$, there is fast extinction.
\item[(ii)]
If $\eta<\frac 12$ or $\tau<3$, there is slow extinction and metastability. Moreover, the metastability exponent satisfies
\begin{equation}
\label{dens1}
\xi \:=\;\left\{\begin{array}{ccl}\frac1 {3-\tau} &\mbox{ if }&\;
\left\{\begin{array}{rl}
&\eta\le 0 \mbox{ and }\tau\le\frac 5 2,\\
\mbox{ or }&0\le \eta\le \frac12\mbox{ and }\tau\le\frac 5 2+\eta, \\
\mbox{ or }&\eta\ge \frac 12\mbox{ and }\tau<3,
\end{array}\right.
\\[-2mm]
\\
2\tau - 3&\mbox{ if }&\;\eta\le 0 \mbox{ and }\tau\ge\frac 5 2,\\
[-2mm]\\
\frac{2\tau - 3-2\eta}{1-2\eta}&\mbox{ if }&\;0\le \eta< \frac12\mbox{ and }\tau>\frac 5 2+\eta.
\end{array}\right.
\end{equation}\end{itemize}
\medskip

\item[(b)]
Consider the \textbf{preferential attachment kernel} or \textbf{strong kernel}.
\begin{itemize}
\item[(i)] 
If $\eta\ge \frac12$ and $\tau>3$, there is fast extinction.
\item[(ii)]
If $\eta<\frac12$, or if $\eta\ge \frac12$ and $\tau<3$, there is slow extinction and metastability, and the metastability exponent satisfies
\begin{equation}
\label{dens2}
\xi\:=\;\left\{\begin{array}{ccl}
2\tau-3&\mbox{ if }&\eta\le 0,\\
[-2mm]\\
\frac{2\tau-3-2\eta}{1-2\eta}&\mbox{ if }&0<\eta< \frac 12 \mbox{ and }\tau\ge 2+2\eta,
\\
[-2mm]\\ \frac {\tau-1} {3-\tau} &\mbox{ if }& 
\tau<3  \mbox{ and }\eta>\frac \tau 2 - 1.
 \end{array}\right.
\end{equation}
\end{itemize}
\medskip

\item[(d)] 
Consider the \textbf{weak kernel}. There is slow extinction and metastability, and the metastability exponent satisfies
$\xi=\tau-1$.
\end{itemize}
\end{theorem}

\begin{figure}[h!]
    \centering
    {{\includegraphics[width=7.2cm]{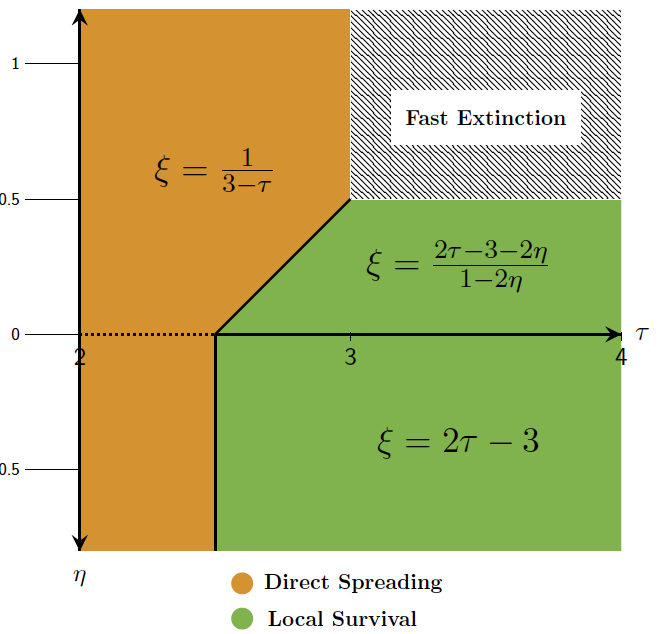} }}%
    \quad
    {{\includegraphics[width=7.4cm]{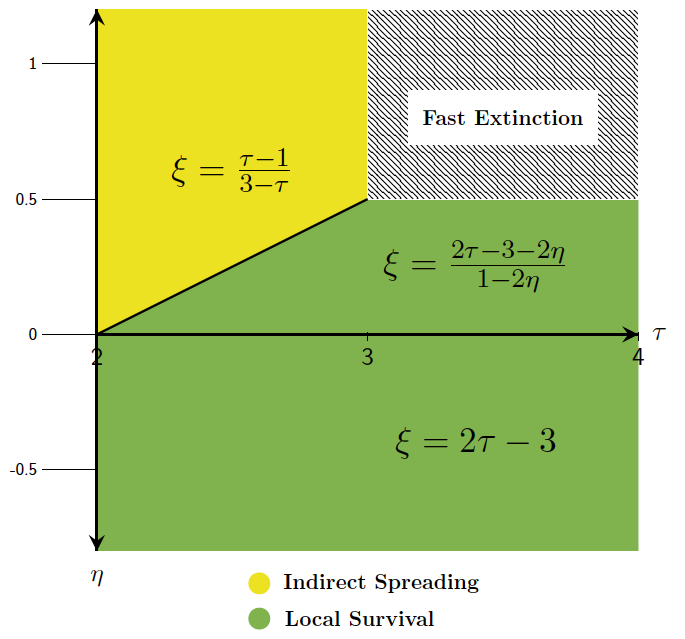} }}
\caption{Phase diagram summarising Theorem~\ref{teofinal}  for the factor kernel (left) and for the strong and preferential attachment kernel (right). Bold lines indicate first-order phase transitions. Note the phase transition at $\eta=0$ within the local survival spreading regime. For the weak kernel  quick direct spreading (brown phase) prevails for all parameters.\vspace{-1mm}}
\label{one}
\end{figure}

\begin{remark}
The different exponents in Theorem~\ref{teofinal} are actually indicative of different survival strategies for the infection, as indicated in Figure 1.%
\end{remark}
\ \\[-10mm]

\begin{remark}\label{up}
Theorem~\ref{teofinal} is robust under changes of the update rates. It essentially only requires that the update rate of $\{i,j\}$ 
 depends in the case $\eta>0$ on the more powerful of the adjacent vertices. In particular, if $k_i$ denotes the expected degree of vertex $i$, our results hold verbatim for the update rates \smash{$\varkappa (k_i^\eta \vee k_j^\eta)$} or \smash{$\varkappa (k_{i\wedge j}^\eta)$}.
\end{remark}%
\ \\[-11mm]

\begin{remark}
Theorem~\ref{teofinal} shows that our models interpolate between the static case and the mean-field case when taking $\eta$ to $-\infty$ or $+\infty$, respectively. We observe the same exponents as in the static case as soon as $\eta\le 0$ and as in the mean-field case as soon as $\eta\geq\frac12$. 
\end{remark}

We now comment on  the proofs, which split into lower and upper bounds.\pagebreak[3]
\medskip

\emph{Lower bounds} are required in the slow extinction case only. To verify them one needs to show that the conjectured strategies succeed in the given regimes. Even though the strategies themselves are in some cases similar to those in~\cite{JLM19}
this is much harder here, because we have to handle the dependencies arising from the relatively slow mixing of the network. A new argument is also needed to prove metastability, which is independent of the individual strategies and based on self-duality of the process, see Section~\ref{sec-metastab}.
\pagebreak[3]\medskip

In the case of \emph{quick direct spreading}, which relies on stars directly infecting each other, we 
define  a {connectivity condition}, that the network is likely to keep satisfying for an exponentially long time, see Proposition~\ref{catorbound}. The condition ensures that there is always a  lower bound on the number of healthy stars neighbouring infected stars.  On this condition we argue, similarly as Cator and Don in \cite{CD18}, that the number of infected stars can be bounded from below for an exponentially long time by a random walk with upward drift subject to an absorbing lower and reflecting upper bound. Roughly speaking, the extinction time can then be bounded from below by the absorption time, which is exponential in~$N$, and the lower bound on the metastable density follows (as in all other phases)  by considering the number of infected neighbours of the set of infected stars. We have marked this phase in brown in the phase diagram, see~Figure~1.\smallskip

For  \emph{quick indirect spreading} the method above has to be refined. We now introduce a discrete time scale chosen so that vertices are sufficiently likely to be stable, in the sense that they neither update nor recover during one epoch. We then ensure that at every time step the set of infected stable stars is connected to sufficiently many healthy stable stars via a path of length two, using a stable intermediate vertex, to retain the infection for an exponential amount of time, see Proposition~\ref{catorbound2}. This phase is marked in yellow in Figure~1.
\medskip

The  \emph{local survival} strategy uses an effective  time-scale  which 
represents the time until a powerful vertex has a sustained recovery, i.e.\ a recovery that is not outdone by immediate reinfection by the neighbours of the vertex. Proposition~\ref{survivaledge} is the key result that establishes the existence of this  time-scale. Heuristically, at the time of its recovery the neighbours of a powerful vertex are infected with probability $\frac{\lambda}{\lambda+1+\kappa}$, where $\kappa$ denotes the update rate, i.e. when the last event affecting the neighbour or the connecting edge was an infection. With the same probability this neighbour immediately reinfects the recovered powerful vertex. If $\eta\leq0$ this probability is of order $\lambda$ and the infection can survive if the set of stars is chosen so that their expected degree is of larger order than  $\lambda^{-2}$.
If $\eta>0$ the probability depends via the update rate $\kappa\approx k^{-\eta}$ on the degree $k$ of the powerful vertex so that for $\eta\geq\frac12$ the local survival strategy collapses. For $0<\eta<\frac12$ a local survival time exponential in $\lambda^2$ and $k^{1-2\eta}$ is possible. The effect of the large edge updating probability in this case reduces the efficiency of the strategy and leads to nondifferentiability of the corresponding metastability exponent at $\eta=0$.
In Propositions~\ref{prop:skeletonstable} and~\ref{keyspreading} we show that local survival at the powerful nodes for a time which is stretched exponential in the vertex degree insures slow extinction of the contact process. Note that this is different from the behaviour in Equation~(3) or Lemma~7 of~\cite{JLM19} where local survival at powerful nodes holds for a time polynomial in $k$ and slow extinction can only be established by combining with a suitable spreading strategy operating in the effective time-scale.
The two phases corresponding to the local survival strategy  are marked dark green in Figure~1. 
\medskip

\emph{Upper bounds} in \cite{JLM19} were proved in two steps, \emph{first} the full model was coupled to a simpler model, the wait-and-see model, which has more infected vertices and a simplified Markovian transition that only controls the presence of relevant edges. For this stochastic upper bound we were able, in the \emph{second} step, to associate a score to each configuration which eventually led to the construction of a supermartingale, which gave the required upper bounds by application of the optional stopping theorem. This argument, designed for simultaneous vertex updating with $\eta\geq 0$, can be extended to deal with edge updating schemes, see Theorem~\ref{teoupper_edge}. However, no extension is possible for  edge updating with $\eta<0$. 
\medskip

\pagebreak[3]

The new approach needed to deal with those cases is inspired by the methods of \cite{MVY13} for static networks and based on the self-duality of the process. Using self-duality, we can bound the upper metastable density by the probability that the process, starting from one infected vertex $x$, survives up to some large given time $t$. Until time $t$, the contact process then stays in the ``local dynamical neighbourhood'' of~$x$ in the graph, which is with high probability tree-like.
A well-known result~\cite{P92} is that the contact process on a \emph{static} tree with degrees bounded by $\frac1{8\lambda^2}$ is locally subcritical: it may survive globally but not locally.
We obtain a similar result for \emph{dynamical} graphs that can be applied to the local dynamical neighbourhood of~$x$. 
For precise statements, see Lemma~\ref{evolving_to_static} below and Inequalities~\eqref{boundA4} and~\eqref{boundA5}. Informally, the time evolution increases the number of neighbours 
an infected vertex can infect, but our study reveals that this cannot have a stronger effect than multiplying the edge connection probability by $1+4\varkappa$. We thus obtain upper bounds for the metastable density valid for all $\eta\le 0$, which lead to the same metastable exponent as in the static case.
\medskip 

The rest of this paper is structured as follows. We prove the lower bounds in Section~3, and the upper bounds in Section~4. Those parts of the argument, which are mere extensions of arguments  for the fast extinction case explained in~\cite{JLM19} are omitted in the main text and briefly sketched in the appendix, Section~5. 
\medskip

\pagebreak[3]

\section{Lower bounds}

\subsection{Graphical representation and lower bounds framework}

\label{sec:graphical}

For each $N\in \N$, the evolving network model $({\mathscr G}_t \colon t\geq 0)=({\mathscr G}^{\ssup N}_t \colon t\geq 0)$ is represented 
with the help of the following independent random processes;
\begin{enumerate}
	\item[(1)] For each $x,y\in\N$ with $x\not=y$  a Poisson point process \smash{$\mathcal U^{x,y}=(U^{x,y}_n)_{n\ge 1}$} 
	of intensity $\kappa_{x,y}$, describing the updating times 
	of the potential edge $\{x,y\}$. \smallskip
	\item[(2)] For each $\{x,y\}$ with $x\not=y$ and $x,y\leq N$, a sequence of independent random variables $(C^{x,y}_n)_{n\ge0}$ (which we denote ${\mathcal C}^{x,y}$), all Bernoulli with parameter $p_{x,y}$, 
	describing the presence/absence of the edge in the network after the successive updating times of the potential edge $\{x,y\}$. More precisely, if $t\ge 0$ then 
	$\{x,y\}$ is an edge in~\smash{${\mathscr G}_t$} if and only if  $C^{x,y}_n=1$ for $n= \vert[0,t]\cap \U^{x,y}\vert$.\smallskip
\end{enumerate}
Given the network we represent the infection by means of the following set of
independent random variables;
\begin{enumerate}
	\item[(3)] For each $x\in \N$, a Poisson point process $\mathcal R^x=(R^x_n)_{n\ge 1}$ of intensity one describing the recovery times of~$x$.\smallskip
	\item[(4)] For each $\{x,y\}$ with $x\not=y$, a Poisson point process $\mathcal{I}_0^{x,y}$ with intensity $\lambda$ describing the infection  times along the edge $\{x,y\}$. Only the trace $\mathcal I^{x,y}$ of this process on the~set
	$$
	\big\{t\in [0,\infty) \colon \{x,y\} \textrm{ is an edge of }\mathscr G_t\big\}=\bigcup_{n=0}^\infty \{[U^{x,y}_n, U^{x,y}_{n+1}) \colon  C^{x,y}_n=1\}
	$$
	can actually cause infections. Write $(I^{x,y}_n)_{n\ge1}$ for the ordered points of $\mathcal I^{x,y}$.
	If just before time $I^{x,y}_n$ vertex $x$ is infected and $y$ is healthy, then $x$ infects $y$ at time  $I^{x,y}_n$. If $y$ is infected and $x$
	healthy, then $y$ infects $x$.  Otherwise, nothing happens.\smallskip
\end{enumerate}
The infection is now described by a process $(X_t(x)\colon  x\in \{1,\ldots,N\},  t\geq 0)$ with values 
in~\smash{$\{0,1\}^N$}, such that $X_t(x)=1$ if $x$ is 
infected at time~$t$, and $X_t(x)=0$ if $x$ is healthy at time $t$.  
More formally, the infection process associated to this 
graphical representation and to a starting set $A_0$ of infected vertices, is the c\`adl\`ag process with $X_0(x)=\one_{A_0}(x)$ 
evolving only at times $t\in \mathcal R^x\cup \bigcup_{n=1}^\infty I_n^{x,y}$, according to the following rules:
\begin{itemize}
	\item If $t\in \mathcal R^x$, then $X_t(x)=0$ (whatever $X_{t-}(x)$). 
	\item If $t  \in \mathcal I^{x,y}$, then 
	$$
	(X_t(x),X_t(y))=
	\left\{
	\begin{array}{rl}
	(0,0) & \mbox{ if } (X_{t-}(x),X_{t-}(y))=(0,0). \\
	(1,1) & \mbox{ otherwise.} \\
	\end{array}
	\right.
	$$
\end{itemize}
The process $(({\mathscr G}_t, X_t) \colon t\geq 0)$ is a Markov process describing the simultaneous evolution of the network and of the infection. We 
denote by $(\F_t\colon t\ge 0)$ its canonical filtration.
Using the graphical representation we obtain monotonicity and duality properties of the contact
process on the evolving graph as in Proposition~1 of~\cite{JLM19}.
We detail briefly the construction of the dual process on a time interval $[0,t_0]$, which we use in Section~\ref{sec-metastab} to show metastability. The network dynamics being stationary and reversible, the reversed-time network 
$(\check {\mathscr G}_t^{t_0})_{0\le t \le t_0}:=
(\mathscr G_{t_0-t})_{0\le t \le t_0}$ has the same law as 
$(\mathscr G_t)_{0\le t \le t_0}$. 
We then construct the dual contact process 
$(\check X_t^{t_0})_{0\le t \le t_0}$ on 
$(\check {\mathscr G}_t^{t_0})_{0\le t \le t_0}$, by using the recovery times $\check {\mathcal R}^x$ and the infection times $\check {\mathcal I}^{x,y}$ given by
\begin{align*}
	\check {\mathcal R}^x&= \{t_0-t \colon t\in \mathcal R^x\}, \\
	\check {\mathcal I}^{x,y}&=\left\{t_0-t \colon  t\in \mathcal I^{x,y}\right\}
	=\left\{t_0-t \colon t\in \mathcal I_0^{x,y},  \{x,y\} \textrm{ is an edge of } \check {\mathscr G}_{t_0-t}^{t_0}\right\}.
\end{align*}
The dual processes $(\check X^{t_0}_t, \check{\mathscr G}^{t_0}_t)_{0\le t \le t_0}$ then have the same law as the original processes $(X_t,\mathscr G_t)_{0\le t\le t_0}$, 
and additionally, if we fix a starting set of infected vertices $A$ for the original process and $B$ for the dual process, then the event $\{\exists x\in B, X_{t_0}(x)=1\}$ coincides with the event $\{\exists x\in A, \check X^{t_0}_{t_0}(x)=1\}$.

\bigskip

Fix a parameter $a=a(\lambda)\in (0,1/2)$; throughout this section we will work on a subgraph $({\widetilde{\mathscr G}}_t \colon t\geq 0)$ with vertex set ${\!\St}\cup{\Co}^0\cup{\Co}^1\cup\Vo^{odd}$  where 
\begin{align*}
{\!\St}&:=\,\{x\in\{\lceil\sfrac{aN}{2}\rceil+1,\ldots, \lceil aN\rceil\}\colon\,x\text{ is even }\},\\[2pt]{\Co^0}&:=\,\{y\in\{\lceil \sfrac{N}{2}\rceil+1,\ldots, N\}\colon\,y=4k\text{ for some }k\in\N \},\\[2pt]{\Co^1}&:=\,\{y\in\{\{\lceil \sfrac{N}{2}\rceil+1,\ldots, N\}\colon\,y=4k+2\text{ for some }k\in\N \},
\\[2pt]{\Vo^{odd}}&:=\,\{z\in\{1,\ldots, N\}\colon\,z\text{ is odd }\}.
\end{align*}%
The vertices in $\St$ correspond to \emph{stars}, which are the key ingredients in the survival strategies. The vertices in $\Co=\Co^0\cup\Co^1$ correspond to \emph{connectors}, which are partitioned into $\Co^0$ and $\Co^1$ depending on how we use them. Connectors in $\Co^0$ will be used by stars to survive locally, while connectors in $\Co^1$ will be used by the infection to spread. Finally, vertices in $\Vo^{odd}$ will be used to provide lower bounds for the metastable density. In order to simplify computations the connection probabilities $p_{i,j}$ between vertices in ${\!\St}\cup{\Co}$ are replaced by the lower bounds
\[
\widetilde{p}_{i,j}\;=\left\{\begin{array}{cl}p_{\lfloor aN\rfloor,\lfloor aN\rfloor}&\text{ if }i,j\in{^a\!\St}\\p_{\lfloor aN\rfloor,N}&\text{ if }i\in{^a\!\St},j\in{^a\Co}\\0&\text{ if }i,j\in{^a\Co}.\end{array}\right.
\]
The idea behind simplifying the connection probabilities is to work on a model that cares only about vertex quantity and not identity. Even though it still remains that the updating  rates~$\kappa$ are different for each edge, this idea becomes heuristically correct since $\kappa_x=\Theta(a^{-\gamma\eta})$ uniformly over all stars~$x$, and $\kappa_x=\Theta(1)$ uniformly over all connectors~$x$. Observe that we can construct both $\widetilde{{\mathscr G}}_t$ and  ${\mathscr G}_t$ so that  $\widetilde{{\mathscr G}}_t\subseteq{\mathscr G}_t$ and hence the original process dominates the one running on the subgraph.\medskip

Our first auxilliary result extends Lemma 1 in \cite{JLM19}, giving lower bounds for the expected density whenever there is survival. Define  
\smash{$(\mathfrak{F}^{_{0,1}}_{t})$} as the filtration given by all the $\Rec^{x}$, $\I_0^{x,y}$, and $\U^{x,y}$ up to time $t$ where $x,y\notin\Vo^{odd}$, as well as all the connections between such vertices up to $t$. In words, \smash{$(\mathfrak{F}_t^{_{0,1}})$} is the natural filtration of the process running on the network that does not consider vertices in $\Vo^{odd}$.

\begin{lemma}\label{lemmalower}
	For any $r>0$ and $t>0$  there is $C>0$ (independent of $\lambda$, $a$, $N$) such that
	\begin{equation}
	\label{lemma1r}
	\E\Big[\sum_{z\in\Vo^{odd}} X_{t+1}(z) \,\Big|\,\mathfrak{F}_t^{0,1}\, \Big] \;\geq\; C \,\Big(\int_0^1(\lambda a p(a,x)\wedge1)dx\Big) N,
	\end{equation}
        on any event $A\in \mathfrak{F}_t^{0,1}$ implying $\big|\{x\in \St \colon X_t(x)=1\}\big|\,\geq\,raN$. 
        \end{lemma}

\begin{proof}
	The proof of the result is a straightforward adaptation of the argument given in \cite{JLM19}. Fix a realization of the variables 
	generating~$\mathfrak{F}_t^{_{0,1}}$, such that $A$ holds and define $\St_0=\{x\in \St \colon X_t(x)=1\}$ so that $|\St_0|\geq raN$. Let 
	\[\St':=\{x\in\St_0 \colon \mathcal{R}^x\cap[t,t+1]=\emptyset\}\]
	and observe that each $x\in \St_0$ belongs to $\St'$ independently with probability $e^{-1}$ and therefore the cardinality $S'$ of $\St'$ dominates a binomial random variable with parameters~$raN$ and $e^{-1}$, so that with probability $1-e^{-cN}$ we have \smash{$|\St'|\geq\frac{raN}{4}$}. The event above depends only on \smash{$\mathfrak{F}_{t+1}^{0,1}$}. For any $z\in \Vo^{odd}$ we say that $z$ \emph{succeeds} if, 
	\begin{itemize}
		\item $\Rec^{z}\cap[t,t+1]=\emptyset$, and
		\item there is $x\in\St'$ such that
		\begin{itemize}
		\item $\mathcal{I}_0^{x,z}\cap[t,t+1]\neq\emptyset$, and
		\item at the first infection in $\mathcal{I}_0^{x,z}\cap[t,t+1]$, the edge $\{x,z\}$ belongs to the graph.
		\end{itemize}
	\end{itemize}
	Observe that the events $\{z\text{ succeeds}\}_{z\in\Vo^{odd}}$ are independent of \smash{$\mathfrak{F}_{t+1}^{0,1}$}. Using stationarity of the network we deduce that for any fixed $z\in\Vo^{odd}$, conditional on $|\St'|\geq\frac{raN}{4}$ we have 
	\begin{align*}\P(z\text{ succeeds})&\geq e^{-1}\big(1-\prod_{x\in\St'}\big(1- (1-e^{-\lambda}) p_{x,z}\big)\big)\\[2pt]&\ge\, e^{-1}\big(1-\big(1- (1-e^{-\lambda}) p_{{\lceil aN\rceil},z}\big)^{|\St'|}\big)\\[2pt]&\ge\, e^{-1}\big(1- \exp\big(-  \sfrac {\lambda p_{\lceil aN\rceil,z} |\St'|}{2}\big)\big) \,\ge\, \sfrac {\lambda raN p_{\lceil aN\rceil,z}}{64} \wedge \sfrac 18.\end{align*}
	Since $X_{t+1}(z)=1$ for any $z\in \Vo^{odd}$ that succeeds we finally deduce
	\begin{align*}\E\Big[\sum_{z\in\Vo^{odd}} X_{t+1}(z) \,\Big|\,\mathfrak{F}_t^{0,1}\, \Big] &\geq\; \sum_{z\in\Vo^{odd}}\sfrac {\lambda raN p_{\lceil aN\rceil,z}}{64} \wedge \sfrac 18\\[2pt]&\geq\; C N\bigg(\frac{1}{N}\sum_{z\in\Vo^{odd}} \big(\lambda a p(a,\sfrac{z}{N})\wedge 1\big)\bigg), \end{align*}
	for fixed $C>0$, where we used that $p_{\lceil aN\rceil,z}=\sfrac{1}{N}p(\sfrac{\lceil aN\rceil}{N},\sfrac{z}{N})\geq\frac{1}{2N}p(a,\frac{z}{N})$. Approximating the term within parenthesis by a Riemann integral, we obtain the desired result.
\end{proof}
\medskip

In order to ease the notation, for the remainder of this section we assume that both $aN$ and $N$ are integers divisible by $8$, and that $N$ is very large. Under this assumptions we have $|\St|=\frac{aN}{4}$, $|\Co^0|=|\Co^1|=\frac{N}{8}$, and  $p_{\lfloor aN\rfloor, \lfloor aN\rfloor}=\frac{1}{N}p(a,a)$ as well as $p_{\lfloor aN\rfloor,N}=\frac{1}{N}p(a,1)$. As we restrict ourselves to the study of the contact process on~\smash{$\widetilde{{\mathscr G}}_t$} with connection probabilities~$\widetilde{p}_{i,j}$ we will abuse notation and drop the tildes.

\subsection{Quick Strategies}\label{quicks}

We  now identify conditions on the kernel~$p$ for the two quick survival strategies to succeed. 
\medskip

\begin{theorem}\label{teolower} 
	There exist positive $M_{(i)}$ and $M_{(ii)}$ depending on $\varkappa$, such that slow extinction holds for the contact process on the network if, for all $\lambda \in (0,1)$,
	there is $a=a(\lambda) \in (0,1/2)$ satisfying at least one of the following conditions:
	\medskip
	
	\begin{itemize}
		\item[(i)] {\bf (Quick Direct Spreading)}
		$\displaystyle \lambda ap(a,a)>M_{(i)}.$\\[-2mm]
		
		\item[(ii)] {\bf (Quick Indirect Spreading)}
		$\displaystyle \lambda^2a p^2(a,1)>M_{(ii)}$.\\[-2mm]
	\end{itemize}
	Moreover, in each of these cases we have 
	\begin{equation}
	\label{lowdensityetapos}
	\rho^-(\lambda)\;\geq\;c' \lambda a\int_0^1p(a,s)ds,
	\end{equation}
	where $c'>0$ is a constant independent of $a$ and $\lambda$.
\end{theorem}

\begin{remark}

Condition~(i) gives sharp lower bounds on the metastable exponent for 
\begin{itemize}
	\item the weak kernel with the choice of $a(\lambda)= r \lambda^{1/\gamma}$ for small enough constant $r$. We then get a lower bound for $\rho^-(\lambda)$ of the same order, namely $\lambda^{1/\gamma}$.
	\item the factor kernel in the regime marked brown in Figure~\ref{one}, where the choice of $a(\lambda)= r \lambda^{1/(2\gamma-1)}$ for a small enough constant $r$ gives a lower bound for $\rho^-(\lambda)$ of the order $\lambda^{\gamma/(2\gamma-1)}=\lambda^{1/(3-\tau)}$.
\end{itemize}
Condition~(ii) gives sharp lower bounds for the preferential attachment and strong  kernel in the regime marked in yellow in Figure~\ref{one}, where the choice of $a(\lambda)= r \lambda^{2/(2\gamma-1)}$ for a small enough constant $r$ gives a lower bound for $\rho^-(\lambda)$ of the order $\lambda^{1/(2\gamma-1)}=\lambda^{(\tau-1)/(3-\tau)}$.
\end{remark}

\begin{remark}
Under either condition~\eqref{condp2b} or \eqref{condp} and \eqref{condp2a}, we have $\int_0^1p(a,s)ds\geq ca^{-\gamma}$ so the lower bound in \eqref{lowdensityetapos} is always of the form $c'\lambda a^{1-\gamma}$. In the present form our result does not require these conditions.
\end{remark}

We only address the proofs for the case $\eta<0$, since the proofs for $\eta\geq0$ are almost identical to those for the vertex-updating scheme in~\cite{JLM19}. 

\subsubsection*{Quick direct spreading}
Quick direct spreading, as shown in \cite{JLM19}, is a mechanism in which stars directly infect other stars without the aid of connectors. Out of all the strategies, it is the one which is most affected by the slow mixing of the network, since stars in this case tend to form a network which stays locally unchanged for large periods of time. We begin our analysis by proving in a similar fashion to \cite{CD18}, that if $ap(a,a)$ is large, then with high probability, the set of stars is sufficiently connected for an exponentially long time.\smallskip

\begin{proposition}\label{catorbound}
	For any $A\subseteq\St$ and any $t>0$ let $L_{A}(t)$ be the amount of connections between $A$ and $A^c:=\St\backslash A$  at time $t$, and $\rho_A=|A|/(\frac{a}4 N)$ the relative size of $A$. Then, under the assumption $ap(a,a)>256$ there is some $\rho_0\in(\sfrac14, \sfrac12)$ such that
	\begin{equation}\label{cond1a}\P\big(\exists A\subseteq\St, \rho_A\in[\rho_0,1-\rho_0],\, \text{ with }L_{A}(t)<\sfrac{\rho_A(1-\rho_A)a^2p(a,a)N}{32}\big)\,\leq\,Ne^{-\rho_0 aN/4}.\end{equation}

\end{proposition}

\medskip

\begin{proof}
	From the stationarity of the dynamics it follows that for any set $A\subseteq\St$, the random variable $L_{A}(t)$ follows a binomial distribution with parameters $\rho_A(1-\rho_A)( \frac{a}4N)^2$ and $\frac{1}{N}p(a,a)$, so using a Chernoff bound we obtain
	\[\P\big(L_{A}(t)<\sfrac{\rho_A(1-\rho_A)a^2p(a,a)N}{32}\big)\,\leq\,e^{-\rho_A(1-\rho_A)a^2p(a,a)N/128}.\]
	Now, for any $\rho\in\{1/( \frac{a}4 N),\ldots,1\}$, the amount of sets $A\subseteq\St$ with $\rho_A=\rho$ is
	\[{{ \frac{a}4N}\choose{\rho \frac{a}4 N}}\,\leq\,\big(\frac{e}{\rho}\big)^{\rho\frac{a}4N}\,\leq\,e^{\rho \frac{a}4N(1-\log(\rho))},\]
	so from a union bound it follows that
	\[\P\big(\exists A\subseteq\St, \rho_A=\rho \text{ with }L_{A}(t)<\sfrac{\rho(1-\rho)a^2p(a,a)N}{32}\big)\,\leq\, \exp\Big(-\rho \tfrac{a}4 N\big( \sfrac{(1-\rho)ap(a,a)}{32}-1+\log\rho\big)\Big).\]
	From our assumption, we know that $ap(a,a)>256$ so that the expression between parenthesis is bounded from below by $8(1-\rho)-1+\log(\rho)$, which is strictly larger than $2$ for $\rho=1/2$, and hence also in some interval $[\rho_0,1-\rho_0]$ where $1/4<\rho_0<1/2$. Suppose that the condition on the left-hand side is satisfied, and observe that there are at most $ \frac{a}4N$ possible values of $\rho_A$, so that
	\[\P\big(\exists \rho\in[\rho_0,1-\rho_0],\,A\subseteq\St, \rho_A=\rho \text{ with }L_{A}(t)<\sfrac{\rho(1-\rho)a^2p(a,a)N}{32}\big)\,\leq\,Ne^{-\rho_0 aN/4}.\] \\[-11mm]
\end{proof}
\medskip

Using the proposition we now prove (i) in Theorem \ref{teolower}. Choose~$a$ such that $\lambda ap(a,a)>M_{(i)}:=256$, which  implies that $ap(a,a)>256$, and take $\rho_0$ as in Proposition \ref{catorbound}. Let \smash{$\{\tau_i\}_{i\in\N}$}  be the ordered sequence of updating times of edges between vertices in $\St$, i.e., $\tau_0=0$ and
\[\tau_{i+1}\,:=\;\inf\Big\{t>\tau_i \colon \bigcup_{u,v\in\St}\mathcal{U}^{u,v}\cap(\tau_i,t]\neq\emptyset\Big\}\]
and observe that the subgraph induced by $\St$ is stationary at these times as well. As a result we can apply the previous proposition to the random variables $L_A
(\tau_i)$ to obtain
\[\P\big(\exists A\subseteq\St, \rho_A\in[\rho_0,1-\rho_0] \text{ with }L_{A}(\tau_i)<\sfrac{\rho(1-\rho)a^2p(a,a)N}{32}\big)\,\leq\,Ne^{-\rho_0 aN/4}.\]
Now, since the time between updates $\tau_{i+1}-\tau_i$ is the minimum of $ \sfrac{a}4N(\sfrac{a}4N-1)$ exponential random variables with rate bounded from above by $\varkappa$, the amount of updating events on each interval $[0,T]$ is bounded from above by a Poisson random variable with mean $N^2\varkappa T$. Using a large deviation bound, we have
\[\P(\tau_{2\lceil N^2\varkappa T\rceil}<T)\,\leq\,e^{-\frac{N^2\varkappa  T}{4}}.\]
Since the network changes only at times of the form $\tau_i$, by taking $T=e^{\rho_0 aN/8}$ we can extend Proposition \ref{catorbound} to the whole interval $[0,T]$, that is, defining the event
\[\mathcal{E}_0\,:=\,\left\{\exists t\in[0,T], \,A\subseteq\St, \rho_A \in[\rho_0,1-\rho_0] \text{ with }L_{A}(t)<\sfrac{\rho_A(1-\rho_A)a^2p(a,a)N}{32}\right\},\]
its probability is bounded by
\[e^{-\frac{N^2\varkappa  T}{4}}\,+\,\P\big(\exists i\leq \lceil 2N^2\varkappa T\rceil, \,A\subseteq\St \colon \rho_A\in[\rho_0,1-\rho_0] \text{ with }L_{A}(\tau_i)<\sfrac{\rho(1-\rho)a^2p(a,a)N}{32}\big),\]
which is smaller than \smash{$e^{-\frac{N^2\varkappa  T}{4}}+2N^3\varkappa  Te^{-\rho_0 aN/4}\;\leq\;e^{-cN}$} for all $N$ large and $T=e^{cN}$ for some $c>0$ depending on $a$ and $\rho_0$ but not on $N$. The event $\mathcal{E}_0$ depends only on the processes $\{\mathcal{U}^{u,v},\mathcal{C}^{u,v}\}_{u,v\in\St}$ and by fixing a realization of these processes, we obtain
\[\P(T_{\rm ext}>T)\;\leq\;\E\big[\P\big(T_{\rm ext}>T\,\big|\,\{\mathcal{U}^{u,v},\mathcal{C}^{u,v}\}_{u,v\in\St}\big){\bf 1}_{\mathcal{E}_0^c}\big]\,+\,e^{-cN}.\]
It remains to prove that for any fixed realization of the environment satisfying the condition of $\mathcal{E}_0^c$, the process (now depending only on the $\mathcal{R}$ and $\mathcal{I}$) shows slow extinction. To do so, couple $X_t$ to a new process $X'_t$ based on the same input from the graphical construction  except that the initial set of infected vertices are the even integers in $\{\sfrac{aN}{2},\ldots, \sfrac{3aN}{4}\}$, and  that it ignores all infection events taking place at edges $\{x,y\}$ with $x\in\St$ and $y\in\Co$. By monotonicity, $X'_t\leq X_t$ for all $t\geq0$, and hence it is enough to show slow extinction for $X'_t$.\pagebreak[3]%
\bigskip%

Define $A_t$ as the set of infected stars at time $t$ for $X'$, that is, $A_t:=\{v\in\St \colon X'_t(v)=1\}$. It is easy to see that at every time $t$, $|A_t|$ decreases by one at rate $|A_t|$ (since every infected vertex recovers at rate 1), and it increases by one at rate $\lambda L_{A_t}(t)$ (since a healthy stars gets infected precisely when an infection traverses an edge between $A_t$ and $A_t^c$). From our choice of the environment we know that for $t\leq T$ there is no set $A\subseteq\St$ with $\rho_A\in[\rho_0,1-\rho_0]$ such that $L_A(t)<\sfrac{1}{32} \rho(1-\rho)a^2p(a,a)N$, so it follows that 
\[|A_t|\in[\rho_0 \sfrac{a}4 N,(1-\rho_0)\sfrac{a}4N]\;\Longrightarrow\;\lambda L_{A_t}(t)\geq \sfrac{1}{32} {\lambda \rho_{A_t}(1-\rho_{A_t})a^2p(a,a)N}.\]
Instead of proving slow extinction we go a little further and prove a stronger result instead, namely, that there is some $c>0$ such that
\[\P(\exists t<e^{cN},\,|A_t|< \rho_0 \sfrac{a}4 N)\leq e^{-cN}.\]
In order to control $|A_t|$ we define a continuous time random walk $\{B_t\}_{t\geq0}$ on the set $\{\lceil \rho_0 \sfrac{a}4N\rceil,\ldots,\lceil (1-\rho_0)\sfrac{a}4N\rceil\}\subseteq\N$ starting at $B_0=\sfrac{a}4 N$ with rates
\begin{align*}
B_t\longrightarrow B_t-1&\quad \mbox{ at rate }\quad B_t\\
B_t\longrightarrow B_t+1&\quad \mbox{ at rate }\quad \sfrac{\lambda B_t\rho_0 ap(a,a)}{8},  
\end{align*}
and where $\lceil \rho_0 \sfrac{a}2N\rceil$ and $\lceil (1-\rho_0)\sfrac{a}2N\rceil$ act as absorbing and reflecting barriers, respectively. From the lower bound on $\lambda L_{A_t}(t)$ and monotonicity it is easy to see that we can couple $X'$ and $B$ in such a way that $B_t\leq |A_t|$ for all $t\leq\tau_B$, where $\tau_B$ is the time at which $B_t$ hits its absorbing state.\smallskip\pagebreak[3]

From our assumption $\rho_0>1/4$ and $\lambda ap(a,a)>256$ we obtain \smash{$\sfrac{\lambda \rho_0ap(a,a)}{8}>8$}, and hence at each jump the process has probability at least $\sfrac{8}{9}$ to increase by one. It follows that~$B_t$ is a continuous time random walk with a drift towards $\lceil (1-\rho_0)\sfrac{a}4N\rceil$ whose rates are bounded by $N$. For such a process we have (see e.g.~\cite{MP16}) that there is $c=c(a,\rho_0)$ such that
\begin{equation}\label{finalqds}
\P(\exists t<e^{cN},\,|A_t|< \rho_0 aN)\,\leq\, \P(\tau_B< e^{cN})\,\leq\,e^{-cN},
\end{equation}
and hence there is slow extinction. Inequality 
\eqref{lowdensityetapos} follows  from Lemma~\ref{lemmalower} because our proof only involves processes and variables which are \smash{$\mathfrak{F}_\infty^{0,1}$}-measurable.
\medskip

\subsubsection*{Quick indirect spreading} 

Quick indirect spreading is a mechanism in which stars infect other stars with the aid of connectors. In \cite{JLM19} we showed that this mechanism gives slow extinction by dividing time into small periods and controlling the spread of the infection on each period, relying on the fact that the network mixes sufficiently fast. In the current model, edges adjacent to connectors update at constant rate so it is possible to adapt what was done in \cite{JLM19}, although relying on this feature of the network would raise reasonable doubt about whether a similar result could hold on static networks, represented in our model by taking $\eta=-\infty$. An alternative proof would be to use a method similar to what was done in Proposition \ref{catorbound}, thus using combinatorial arguments to show that the network exhibits good connection properties for an exponentially long time, and then running the process on top of a fixed realization of the environment. Even though possible, this proof would require addressing some subtle difficulties which arise in this mechanism such as dependencies and possible bottlenecks. Our approach is therefore based on a mixture of the two methods; in the proposition below we use combinatorial arguments to show that for an exponentially large time, from any set of stars $A$ not too small or large there are sufficiently many infection paths towards $\St\backslash A$ (observe that in Proposition \ref{catorbound} we showed a similar property but only regarding connections). Slow extinction will easily follow from this proposition.\medskip

Fix $t^*=-\log(99/100)(1+2\varkappa )^{-1}$, which will be used as a unit of time, and define $T_0=(0,t^*]$. A vertex $v$ is called $T_0$\emph{-stable} if $\mathcal{R}^v\cap T_0=\emptyset$, i.e., it does not recover in $(0,t^*]$. For any set $A\subseteq\St$ and any $T_0$-stable $z\in\St\backslash A$ we say that $z$ can be reached from $A$ if there is a path $P=(x,y,z)$ in ${\mathscr G}_0$  with
\begin{itemize}\setlength\itemsep{1ex}
	\item $x\in A$, $y\in \Co$,
	\item both $x$ and $y$ are $T_0$-stable,
	\item $(\mathcal{U}^{x,y}\cup\mathcal{U}^{y,z})\cap T_0=\emptyset$, that is, neither the edge $(x,y)$ nor $(y,z)$ update in $(0,t^*]$,
	\item $\mathcal{I}_0^{x,y}\cap[0,t^*/2]\neq \emptyset$, and $\mathcal{I}_0^{y,z}\cap[t^*/2,t^*]\neq \emptyset$.
\end{itemize}
 We define $S(A)$ as the set of all $T_0$-stable $z\in \St\backslash A$ that can be reached from $A$.  \medskip

\begin{proposition}\label{catorbound2}
	Set $t^*$ as above and assume that condition~(ii) in Theorem \ref{teolower} holds for $M_{(ii)}$ sufficiently large. Then for any $N$ large we have
\begin{equation}\label{cond2:quick}\P\big(\exists \,A\subseteq\St, \rho_A\in[0.05,0.95]\text{ with }|S(A)|<\sfrac{1-\rho_A}{20}aN\big)\,\leq\,e^{-hN},\end{equation}
for some $h=h(a,\lambda)$, and where $\rho_A=|A|/(\tfrac{a}{4}N)$ as in Proposition \ref{catorbound}.
	\end{proposition}
\pagebreak[3]

\begin{proof} 
	Define $\Co^s$ and $\St^s$ as the set of connectors and stars, respectively, that are $T_0$-stable. Observe that any vertex is $T_0$-stable independently with probability at least $e^{-t^*}\geq99/100$ from our choice of $t^*$. It follows from a Chernoff bound that there is some $h_1>0$ independent of $a$ and $N$ such that
	\[\P\big(|\Co^s|>\sfrac{98N}{400},\,|\St^s|>\sfrac{98aN}{400}\big)\;\geq\;1-e^{-h_1aN}.\]
	Observe that this event depends only on the $\mathcal{R}$ processes, and from the bound above it will be enough to show \eqref{cond2:quick} for any realisation of those processes satisfying this event. Fix now $A\subseteq\St$ and abbreviate $\rho=\rho_A$, which we assume satisfies $0.05\leq \rho\leq 0.95$.
	Define the sets $A^s$ and $(\St\backslash A)^{s}$ as the set of stars in $A$ and $\St\backslash A$, respectively, that are $T_0$-stable, which must satisfy $|A^s|>\rho aN/8$ and $|(\St\backslash A)^s|>(1-\rho)aN/8$ under the assumption $|\Co^s|>\sfrac{98N}{400},\,|\St^s|>\sfrac{98aN}{400}$ and the bounds on  $\rho$. With $A$ fixed let
	\[\Co^s(A):=\big\{y\in\Co^s\colon \exists x\in A^s\text{ with }\{x,y\}\in\mathscr G_0\text{, }\mathcal{U}^{x,y}\cap T_0=\emptyset\text{, and }\mathcal{I}_0^{x,y}\cap[0,t^*/2]\neq \emptyset\big\}\]
	be the set of connectors $y\in\Co^s$ initially connected to some $x\in A^s$, whose connection does not update during $T_0$, and such that an infection event occurs between $x$ and $y$ in $[0,t^*/2]$. For any given $y\in\Co^s$ the probability of satisfying these conditions for a given star $x\in A^s$ is given by
	\[\tfrac{p(a,1)}{N}e^{-2\varkappa  t^*}(1-e^{-\lambda t^*/2})\geq \sfrac{\lambda t^*p(a,1)}{3N}\]
	and hence
	\[\P(y\in \Co^s(A))\geq 1-\big(1-\sfrac{\lambda t^*p(a,1)}{3N}\big)^{\rho aN/8}\,\geq\,\sfrac{1}{30}\lambda t^*\rho ap(a,1)\]
	because $\lambda t^*ap(a,1)$ can be taken sufficiently small. Since the events $\{y\in\Co^s(A)\}_{y\in\Co^s}$ are independent, using a large deviation we conclude that there is some constant $h_2$ such that
	\begin{equation}\label{eq:auxprop2}\P\big(|\Co^s(A)|\geq \sfrac{1}{128}\lambda t^*\rho ap(a,1)N\big)>1-e^{-h_2\lambda t^*\rho ap(a,1)N},\end{equation}
	where from now on we call $\mathcal{E}_1$ the event on the left hand side above. Let $\sigma$ be the $\sigma$-algebra generated by the connections, updatings and infection events between $A^s$ and $\Co^s$ on $T_0$. We will work on realizations of these processes such that $|\Co^s(A)|\geq \sfrac{1}{128}\lambda t^*\rho ap(a,1)N$. For $z\in(\St\backslash A)^s$ and $y\in\Co^s(A)$ define $\mathcal{E}_2(z,y)$ as the event where
	\begin{itemize}\setlength\itemsep{0.5ex}
		\item $\{x,y\}\in\mathscr G_0$ and $\mathcal{U}^{y,z}\cap T_0=\emptyset$
		\item $\mathcal{I}_0^{y,z}\cap[t^*/2,t^*]\neq \emptyset$
	\end{itemize}
and call $\mathcal{E}_2(z)=\cup_{y\in\Co^s(A)}\mathcal{E}_2(z,y)$. Proceeding as before we conclude that on $\mathcal{E}_1$
 $$\P(\mathcal{E}(z,y)|\sigma)\geq \sfrac{t^*p(a,1)}{3N}$$ and hence for any $z\in(\St\backslash A)^s$, on $\mathcal{E}_1$
\begin{align*}\P(\mathcal{E}_2(z)\,|\,\sigma)&\geq 1-\big(1-\sfrac{\lambda t^*p(a,1)}{3N}\big)^{\lambda t^*\rho ap(a,1)N/128}\\[2pt]&\geq\,1-\exp\big(-\sfrac{(t^*)^2\rho}{400}\lambda^2 ap^2(a,1)\big)\;\geq\;1-\exp\big(-\sfrac{(t^*)^2}{400}\rho M_{(ii)}\big)=:q.\end{align*}
Now, let $\widetilde{S}$ be the set of all the $z\in(\St\backslash A)^s$ satisfying $\mathcal{E}_2(z)$, whose size is given by a binomial random variable with parameters $q$ and $|(\St\backslash A)^s|$ so Hoeffding's inequality yields
	\[\P\big(|\widetilde{S}|>\sfrac{1-\rho}{16}qaN\,|\,\sigma)\;\geq\;1-e^{-\frac{(1-\rho)}{8}aND(\frac{q}{2}||q)},\]
	on $\mathcal{E}_1$, where $D(b||c)=b\log(b/c)+(1-b)\log(\frac{1-b}{1-c})$ is the relative entropy between two coins with probabilities $a$, resp.~$b$, of heads.
	From the definition of $q$ we obtain
	\[D(\sfrac{q}{2}||q)\,\geq\,-\log(2)-\sfrac{1}{2}\log(1-q)\,=\,-\log(2)+\sfrac{(t^*)^2}{800}\rho M_{(ii)},\]
	and we conclude that
	\begin{equation}\label{eq:aux2prop2}\P\big(|\widetilde{S}|>\sfrac{1-\rho}{8}qaN\,|\,\mathcal{E}_1)\;\geq\;1-2e^{-(1-\rho)\rho (t^*)^2aM_{(ii)}N/1600}.\end{equation}
	By definition of $\tilde{S}$ and $\Co^s(A)$, for any $z\in\tilde{S}$ there is at least one $y\in\Co^s(A)$ and one $x\in A$ such that $(x,y,z)$ is a path throughout $T_0$, both $x$ and $y$ are $T_0$-stable, and $\mathcal{I}_0^{x,y}\cap[0,t^*/2]\neq \emptyset$, and $\mathcal{I}_0^{y,z}\cap[t^*/2,t^*]\neq \emptyset$. Therefore $\tilde{S}\subseteq S(A)$ and hence putting together \eqref{eq:auxprop2}
 and \eqref{eq:aux2prop2} we obtain
 	\begin{align*}\P\big(|S(A)|>\sfrac{1-\rho}{20}aN)&\geq\P\big(|S(A)|>\sfrac{1-\rho}{16}qaN)\\[2pt]&\geq\;1-2e^{-(1-\rho)\rho (t^*)^2aM_{(ii)}N/1600}-e^{-h_2\lambda t^*\rho ap(a,1)N}\\[2pt]&\geq\;1-3e^{-(1-\rho)\rho (t^*)^2aM_{(ii)}N/1600},\end{align*}
 	where in the first inequality we used that $M_{(ii)}$ is large so that $q>4/5$, and in the third inequality we have used  \eqref{condp} to deduce that $\lambda p(a,1)\geq \lambda^2 ap^2(a,1)\geq M_{(ii)}$. Now, for any $\rho\in\{{1}/{ \frac{a}4 N},\ldots,1\}$, the amount of sets $A\subseteq\St$ with $\rho_A=\rho$ is
	\[{{\frac{a}4 N}\choose{\rho \frac{a}4 N}}\,\leq\,\left(\frac{e}{\rho}\right)^{\rho \frac{a}4N}\,\leq\,e^{\rho  \frac{a}4N(1-\log(\rho))}.\]
	Hence, by a union bound,
	\[\P\big(\exists A\subseteq\St, \rho_A=\rho \text{ with }|S(A)|\leq\sfrac{1-\rho}{20}aN\big)\,\leq\,3e^{-\rho a\frac{N}4\big( \sfrac{(1-\rho)(t^*)^2M_{(ii)}}{400}-1+\log(\rho)\big)}.\]
	Observe that the expression in parenthesis is increasing in $M_{(ii)}$ and concave in~$\rho$, so by taking $M_{(ii)}$ sufficiently large (depending on $\varkappa $) we obtain
	that for all $\rho\in[0.05,0.95]$ we have $(1-\rho)(t^*)^2M_{(ii)}/400-1+\log(\rho)\geq 1$. As there are at most $ \frac{a}4N$ possible values of~$\rho$, we infer the result.
\end{proof}

\medskip

In order to prove the main statement we divide time into intervals of the form $T_n:=(nt^*,(n+1)t^*]$. For any $n\in\N$ and $A\subseteq\St$ construct the set $S_n(A)$ analogously to $S(A)$ but replacing all instances of $T_0$-stability by $T_n$-stability and the connections in $\mathscr G_{nt^*}$ instead of $\mathscr G_0$. Define events
\begin{align*}
\mathcal{D}^n_1&:=\;\{|\St^{u,n}|\leq 0.005aN\},\\
\mathcal{D}^n_2&:=\;\{\forall\ A\subseteq\St \mbox{ with } \rho_A\in[0.05,0.95] \text{ we have }|S_n(A)|\geq\sfrac{1-\rho_A}{20}aN\},
\end{align*}
where $\St^{u,n}$ is the set of stars that are not $T_n$-stable. From stationarity and Proposition \ref{catorbound2} we know that there is some $h_1>0$ depending on $a$ alone such that $\P(\mathcal{D}_2^n)>1-e^{-h_1 N}$. On the other hand, we know that each star is $T_n$-stable independently with probability $99/100$ so a concentration inequality gives $h_2>0$ also depending on $a$, such that $\P(\mathcal{D}_1^n)>1- e^{-h_2 N}$. Take $h=\min\{h_1,h_2\}$ and define
\[\mathcal{D}:=\bigcap_{n=0}^{e^{hN/2}}(\mathcal{D}_1^n\cap\mathcal{D}_2^n),\]
which from the argument above has probability at least $1-2e^{-hN/2}$. Define $\St_n:=\{x\in\St \colon X_{nt^*}(x)=1\}$ and $\rho_n=|\St_n|/(\sfrac{a}4 N)$, and observe that $\rho_0=1$ since we start from full occupancy. We will show inductively that on $\mathcal{D}$ we have $\rho_n>0.1$ for all $n\leq e^{hN/2}$ as follows: Suppose first that  $\rho_n>0.2$ and observe that since $\mathcal{D}_1^n$ is satisfied there are at most $0.02\sfrac{a}4N$ stars that can recover on $T_n$ and hence $\rho_{n+1}\geq\rho_n-0.02>0.1$. Suppose next that $0.1<\rho_n\leq 0.2$ so in particular by definition of $\mathcal{D}_2^n$ we have
\[|S_n(\St_n)|\;\geq\;\sfrac{1-\rho_n}{20}aN\;\geq\;0.16\sfrac{aN}{4}.\]
Since $\St_n$ are the stars that are infected at time $t^* n$, all the infections in the definition of $S_n(\St_n)$ are valid and hence any $z\in S_n(\St_n)$ satisfies $X_{(n+1)t^*}(z)=1$, so it follows that $\rho_{n+1}>0.1$. We conclude that as long as $\rho_n>0.1$ the same is true for $\rho_{n+1}$ and hence we have slow extinction.\smallskip

Again, inequality \eqref{lowdensityetapos} follows  from Lemma~\ref{lemmalower} because our proof only involves processes and variables which are $\mathfrak{F}_\infty^{0,1}$-measurable.
\smallskip

\subsection{Local Survival}\label{localsurvival}

In the previous section we proved that on a set of sufficiently powerful vertices, there are two survival strategies based on spreading the infection on this set either directly or with the help of intermediate connectors. An alternative strategy is for a sufficiently powerful vertex to retain the infection for a longer time with the help of its immediate neighbours and thereby effectively increase the time scale on which the infection can spread between stars. We shall see that, if the updating of edges is sufficiently slow,
such a local survival can hold for a time scale  which is stretched exponentially large in the degree of the powerful vertex, in contrast to strategies in \cite{JLM19}, which only allowed a polynomial local survival time. This qualitative difference implies that while in \cite{JLM19} further restrictions on the set of stars had to be imposed to enable  either direct or indirect spreading in the new timescale,
in our \textit{local survival} regime the delay is so effective that the infection can spread on the set of stars without further restrictions.
While for $\eta\leq0$  local survival is possible for vertices 
with degree $k\gg\lambda^{-2}$ (see~\cite{BB+} for results in the static case), in our dynamical network
with $0<\eta<\frac12$  the updating events of edges can be thought about as additional recovery events  so that local survival requires the more restrictive condition 
$k \gg \lambda^{-2}\kappa^{2}$, where $\kappa$ is the update rate associated with adjacent edges.\bigskip

We begin by defining $\ut$, which is the unit of time at which, when looking at a particular star and its neighbours, we typically see a constant amount of recovery and updating events, that is with $\kappa(x)=\kappa_{\lfloor xN\rfloor}$ we set 
\[\ut\,=\,\frac{1}{1+\kappa(\sfrac{1}{2})+\kappa(\sfrac{a}{2})}.\]
Define a sequence $(J_{k})_{k\ge 0}$ of time intervals with $J_0=[0,2\ut)$, $J_1=[0,3\ut)$, and for $k\geq 1$,
\[J_{k}\,:=\,[(k-2)\ut,(k+2)\ut),\]
where each $J_k$ has length $4\ut$ which will be of order $\Theta(a^{\gamma\eta})$ when $\eta>0$, and $\Theta(1)$ otherwise. We use these intervals as building blocks not only in our proof for local survival of a single star, but also when proving slow extinction, where many stars survive locally using the same set of connectors. With this in mind and observing that for different stars $x$ and $x'$ the processes $\I^{x,y}$, $\I^{x',y}$, $\Rec^x$, $\Rec^{x'}$, $\U^{x,y}$ and $\U^{x',y}$ with $y\in\Co^0$, as well as the connections $\mathcal{C}^{x,y}$, $\mathcal{C}^{x',y}$ are independent, we can make local survival events independent across different stars as soon as we work on a fixed realisation of the $\{\Rec^{y}\}_{y\in\Co^0}$ processes. Define
\[\Co_{k}:=\,\{y\in\Co^0\colon \Rec^y\cap J_k=\emptyset\}\]
where it follows from a large deviation argument and the definition of $\ut$ that there is a constant $\cc_0>0$ independent of $a$, $\lambda$ and $N$ such that
\begin{equation}\label{localstable}
\P(|\Co_{k}|> \cc_0N \; \mbox{ for all } k\in\{0,\ldots,\lfloor e^{\cc_0 N}\rfloor\})\,\geq\,1-e^{-\cc_0 N}.
\end{equation}
Since this bound is already of the form required for slow extinction, we will work from now on on a fixed realisation of the $\{\Rec^{y}\}_{y\in\Co^0}$ processes such that the event on the left holds. Going back to the specific case of local survival of a fixed star $x$ we can refine the previous definition to introduce the set of \textit{stable neighbours} as
\[\Co_{k,x}:=\,\{y\in\Co_{k} \colon \{x,y\}\in\mathscr G_{k\ut}\,\text{ and }\,\U^{x,y}\cap J_{k}=\emptyset\},\]
that is, on top of not recovering, connectors in $\Co_{k,x}$ are connected to $x$ and do not update their connection throughout the interval $[(k-2)\ut,(k+2)\ut)$. The next result shows that for a very large amount of such times, a fixed $x$ will have enough stable neighbours.
\medskip

\begin{proposition}\label{uboundCk}	There is $\cc_1>0$ independent of $a$, $\lambda$ and $N$ such that for any fixed $x\in\St$ and $N$ large,
	\[\P\big( |\Co_{k,x}|>\cc_1p(a,1)\text{ for all }k\leq e^{\cc_1p(a,1)}\big)\,\geq\,1-e^{-\cc_1p(a,1)}.\]
\end{proposition}

\smallskip

\begin{proof}
	Fix any $k\leq\lfloor e^{\cc_0 N}\rfloor$ where $\cc_0$ is as in \eqref{localstable}, and observe that for each $y\in\Co^0$, the random variable $|\U^{x,y}\cap J_k|$ is Poisson distributed with mean $4\ut(\kappa_x+\kappa_y)$, which from our definition of $\ut$ is bounded by $4$ in the case $\eta>0$, and is bounded by $4+8\varkappa $ if $\eta\leq 0$ since $\kappa\leq \varkappa $ in that regime. From stationarity we also have that $\{x,y\}\in\mathscr G_{k\ut}$ with probability $\frac{1}{N}p(a,1)$ independently from $|\U^{x,y}\cap J_k|$, and hence there is $\cc'_1$ independent of $a$, $\lambda$ and $N$, such that
	\[\P(y\in \Co_{k,x})\,\geq\,\sfrac{\cc'_1}{N}p(a,1).\]
	Notice that the events \smash{$\{y\in\Co_{k,x}\}_{y\in\Co^0}$} rely on disjoint processes and hence are independent, and since there are at least $\cc_0N$ connectors in $\Co_k$ for each $k\leq\lfloor e^{\cc_0 N}\rfloor$, a concentration argument gives, for $\cc_1=\cc_0\cc'_1/10$,
	\[\P\big(|\Co_{k,x}|>\cc_1p(a,1)\big)\,\geq\,1-e^{-2\cc_1p(a,1)}.\]
	 The result follows by taking a union bound and using that$e^{\cc_1p(a,1)}< e^{\cc_0 N}$ for $N$ large.
\end{proof}

\medskip

For any fixed star $x\in\St$ we use its stable neighbours to construct a modified version of the contact process as follows.
\begin{definition}\label{modx}
	Fix $x\in\St$. For any realisation of the graphical construction, the process $(\bar{X}^x_t)$, is defined analogous to $(X_t)$ with the only exception that:
	\begin{itemize}
		\item $\bar{X}^x_0\equiv X_0$ on $\Co_{0,x}\cup\{x\}$ and $\bar{X}^x_0\equiv0$ everywhere else.
		\item An infection event $t\in\I^{x,y}$ is only valid if $y\in\Co_{k,x}$ for some $k$ such that $t\in J_k$.
	\end{itemize}
\end{definition}
\smallskip

The processes $\bar{X}^x_t$ are introduced here as a way to study the local behaviour of $X_t$ around each star and its neighbouring connectors more easily, without the noise introduced by other stars and fortuitous events which are irrelevant in the long run. The final ingredient in our construction which will come up many times is the concept of \textit{infected stable neighbours}, which, as the name suggests corresponds to
\[\Co'_{k,x}\,:=\,\{y\in\Co_{k,x} \colon \bar{X}^x_{k\ut}(y)=1\},\]
which are the stable neighbours of $x$ that are infected at time $k\ut$. The following key lemma states in a precise way that using these constructions the total amount of time $x$ is healthy on an interval of length $\ut$, which is $\int_{k\ut}^{k\ut+\ut}(1-\bar{X}^x_s(x))ds$, must be small whenever $|\Co'_{k,x}|$ is sufficiently large.
\smallskip

\begin{lemma}\label{lemmakeylocal}
For any $0<s<t$ and $x\in\St$ define \smash{$\mathfrak{F}_{s,t,x}^0$} as the $\sigma$-algebra generated by the graphical construction up to time $s$, and the $\Rec^y$, $\U^{x,y}$ and $\mathcal{C}^{x,y}$ processes with $y\in\Co^0$ up to time $t$. There is a universal constant $C>0$ independent of $a$, $\lambda$ and $N$, such that for any $M\in\N$ with $W:=\lambda \ut M>C$ and any $k\in\N$, we have
	\begin{equation}\label{keylocal}\P\bigg(\int_{k\ut}^{k\ut+\ut}(1-\bar{X}^x_s(x))ds>\sfrac{\ut}{2}\,\bigg|\;\mathfrak{F}_{k\ut,(k+2)\ut,x}^0\;\bigg)\one_{\{|\Co'_{k,x}|\geq M\}}\,\leq\,Ce^{-W/4}.\end{equation}
The event on the left can be taken to depend only on the $\Rec^x$ and $\I^{x,y}$ processes on $[\ut,k\ut+\ut]$ with $y\in\Co'_{k,x}$.
\end{lemma}
%

\begin{proof}
	Fix first a realisation of $\Rec^x$ on the interval $[k\ut,k\ut+\ut]$ and call $r_1,r_2,\ldots,r_{n}$ its ordered sequence of recovery times (if there are none, then the result follows trivially), adding $r_0=k\ut$. To each $r_i$ we associate a waiting time~$s_i$ where
	\[s_i\,:=\,\left\{\begin{array}{cc}\inf\bigcup_{y\in\Co'_{t,x}}\I_0^{x,y}\cap[r_i,r_{i+1}]&\text{ if the set is non empty}\\[7pt](r_{i+1}-r_i)+E_i&\text{ otherwise, }\end{array}\right.\]
	where the $E_i$ are i.i.d. auxiliary random variables which follow an exponential distribution with rate $\lambda M$. In words, to each interval $[r_i,r_{i+1}]$ we associate the waiting time between the recovery $r_i$ and the first infection event coming from $\Co'_{k,x}$, if any, and $r_i+E_i$ otherwise. Since $x$ can only be healthy between a recovery event and the next infection event, we have
		\[\int_{k\ut}^{k\ut+\ut}(1-\bar{X}^x_s(x))ds\,\leq\,\sum_{i=1}^ns_i,\]
	where the quantity on the left will be strictly smaller in general, since we are neglecting infections between $x$ and connectors not in $\Co'_{k,x}$. The reason we add the auxiliary variables~$E_i$ is just to be able to handle this sum easily, since it can be checked that the resulting variables $s_i$ are i.i.d. following an exponential distribution with rate $\lambda M$ and hence $\sum_{i=1}^ns_i\sim\text{Gamma}(n,\lambda M)$. Now, as soon as $n<W$ we have as a result of a Chernoff bound for gamma random variables, that
	\[\P\Big(\sum_{i=1}^ns_i>\sfrac{\ut}{2}\;\Big|\mathfrak{F}_{k\ut,(k+2)\ut,x}^0,\,|\Rec^x\cap[k\ut,k\ut+\ut]|=n\Big)\one_{\{|\Co'_{k,x}|\geq M\}}\,\leq\,\left(\frac{eW}{n}\right)^ne^{-W/2}.\]
	Recalling that $|\Rec^x\cap[k\ut,k\ut+\ut]|$ is a Poisson random variable with parameter $\ut$, the left-hand side of \eqref{keylocal} is bounded from above by
	\begin{equation}\label{keylocaleq1}\sum_{n=0}^{\lfloor W\rfloor-1}e^{-W/2}\left(\frac{eW}{n}\right)^n\frac{\ut^n}{n!}e^{-\ut}\,+\,\P\big(\text{Poiss}(\ut)\geq W\big).\end{equation}
	For the first term, observe that from the Stirling bound $(n-1)!e^{n-1}\leq n^n$ each summand is smaller than \smash{exp$(-\frac{W}{2}-\ut+1)a_n$}, where $a_n:={(W\ut)^n}/{((n-1)!)^2}$. In particular we have $\frac{a_{n+1}}{a_n}=\frac{W\ut}{n^2}$
	and hence $a_n$ has a global maximum at $n_0=\lceil\sqrt{W\ut}\rceil$ for which, using again Stirling type bounds and the fact that $W\ut\leq n_0^2$, we have
	\[a_{n_0}\,\leq\,e^{2n_0-2}(n_0-1)^2\left(1-\tfrac{1}{n_0}\right)^{-2n_0}\,=\,O\left(e^{2\sqrt{W\ut}}\right).\]
	Since $\ut$ is bounded by $1$, and we are assuming that $W$ is large, it follows that the first term in \eqref{keylocaleq1} is of order $O(e^{-W/4})$. The second term in \eqref{keylocaleq1}, using $\ut\leq 1$ and a Chernoff type bound, is of order $O(e^{-W})$ and hence we obtain \eqref{keylocal}. The last statement of the proposition follows as the $s_i$ do not depend on infections between $x$ and connectors outside of $\Co'_{t,x}$.
\end{proof}

\smallskip

Lemma~\ref{lemmakeylocal} states that stars connected to sufficiently many infected neighbours will spend a significant portion of time infected. Intuitively, any such star acts as a beacon for the infection, thus infecting many of its neighbours and hence creating the \textit{survival loop} that will last for a long time. To formally state this idea we introduce the event $\mathcal{W}_k(x)$ as
\begin{equation}\label{defW}\mathcal{W}_k(x)\,=\,\Big\{|\Co'_{k,x}|\geq\delta \lambda p(a,1)\ut\,\text{ and  }\int_{k\ut}^{k\ut+\ut}(1-\bar{X}^x_s(x))ds\leq\ut/2\Big\},\end{equation}
with $\delta>0$ some constant independent of $a$, $\lambda$ and $N$, to be chosen later. In words, on $\mathcal{W}_k(x)$ the central node $x$ is infected at least half of the time throughout $[k\ut,k\ut+\ut]$ while holding a reservoir of infected neighbours of order at least $\lambda \delta p(a,1)\ut$. For simplicity we will simply write $\mathcal{W}_k$ throughout this section since the central star $x$ will be clear from context.%
\smallskip%

\begin{proposition}[\textbf{Local survival}]\label{survivaledge}
	Fix $x\in\St$ and let $\mathcal{S}$ be the event and $\cc_1$ be the constant in  Proposition \ref{uboundCk}. For any \smash{$k\leq e^{\cc_1p(a,1)/2}$} there is a large universal constant~$C$ such that if 
	\begin{equation}\lambda^2\ut^2p(a,1)>C,
	\label{timescaledef}
	\end{equation}
 then taking $\delta=\frac{\cc_1}{8}$, defining $\bar{k}=\lfloor e^{\delta\lambda^2\ut^2p(a,1)/4}\rfloor $  we have
	\[\P\bigg(\bigcap_{i=k}^{k+\bar{k}}\mathcal{W}_i\;\bigg|\;\mathfrak{F}_{k\ut,(k+\bar{k})\ut,x}^0\,\bigg)\;\geq\;1-Ce^{-\delta\lambda^2\ut^2p(a,1)/4},\]
on the event $\mathcal{S}\cap\{|\Co'_{k,x}|>\delta\lambda p(a,1)\ut\}$, where \smash{$\mathfrak{F}_{k\ut,(k+\bar{k})\ut,x}^0$} is as in Lemma \ref{lemmakeylocal}. Conditional on this $\sigma$-algebra the lower bound is taken to depend only on the $\I^{x,y}$ and $\Rec^{x}$ processes over $[k\ut,(k+\bar{k}\ut)]$.
\end{proposition}
\pagebreak[3]

\begin{proof}
	Since the probability is taken to be conditioned on the event $\{|\Co'_{k,x}|>\delta\lambda p(a,1)\ut\}$ we can use Lemma~\ref{lemmakeylocal} with $M=\delta \lambda\ut p(a,1)$ (with $W=\lambda\ut M$ as large as needed) so that on $\mathcal{S}\cap\{|\Co'_{k,x}|>\delta\lambda p(a,1)\ut\}$ we have
	\begin{equation}\label{survivaledgeaux1}\P\Big(\int_{k\ut}^{k\ut+\ut}(1-\bar{X}^x_s(x))ds\leq\ut/2\;\bigg|\mathfrak{F}_{k\ut,(k+\bar{k})\ut,x}^0\,\Big)\;\geq\;1-C'e^{-\delta\lambda^2\ut^2p(a,1)/4},\end{equation}
	for some universal constant $C'$, where the event on the left can be taken to be independent of the $\mathcal{I}^{x,y}_0$ processes on $[k\ut,k\ut+\ut]$ with $y\notin\Co'_{k,x}$. This already shows that $\mathcal{W}_k$ has a lower bound of the desired form. In order to find a similar bound for $\mathcal{W}_{k+1}$ assume that the event in \eqref{survivaledgeaux1} holds. We need to control the probability of the event
	\[\{|\Co'_{k+1,x}|\geq\delta \lambda p(a,1)\ut\}\] for which there are only two possible scenarios:
	\smallskip
	\begin{enumerate}\setlength\itemsep{7pt}
		\item $|\Co'_{k,x}\cap\Co_{k+1,x}|>\delta\lambda p(a,1)\ut$, that is, sufficiently many infected neighbours of $x$ that are stable on $J_k$ are also stable on $J_{k+1}$ and hence we obtain directly $|\Co'_{k+1,x}|>\delta\lambda p(a,1)\ut$ (since stable connectors do not recover).
		
		\item $|\Co'_{k,x}\cap\Co_{k+1,x}|\leq\delta\lambda p(a,1)\ut$, in which case we have $|\Co_{k+1,x}\backslash\Co'_{k,x}|\geq\frac{\cc_1}{2}p(a,1)$ because $\cc_1p(a,1)\geq 2\delta\lambda p(a,1)\ut$ for sufficiently small $\lambda$. Since the event in \eqref{survivaledgeaux1} is independent of infection events between $x$ and connectors not in $\Co'_{k,x}$, for every connector $y$ outside of this set we are free to check whether it gets infected by $x$ in $[k\ut,(k+1)\ut]$. Since vertices in $\Co_{k+1,x}$ are connected to $x$ and do not recover, it is enough for $y\in\Co_{k+1,x}$ to belong to $\Co'_{k+1,x}$ that
		\[\mathcal{I}^{x,y}_0\,\cap\,\{s\in[k\ut,(k+1)\ut),\,\bar{X}_s^x(x)=1\}\,\neq\,\emptyset\]
		but this occurs with probability at least $\frac{\lambda\ut}{4}$ if $\lambda$ is small enough, since the set on the right has Lebesgue measure at least $\ut/2$. Observing that all these events are independent we conclude that $|\Co'_{k+1,x}|$ dominates a binomial random variable with parameters $\cc_1p(a,1)$ and $\lambda\ut/4$, whose mean is equal to $2\delta\lambda p(a,1)\ut$ from our definition of $\delta$. For $\lambda$ small this quantity is much larger than $8\lambda^2\ut^2p(a,1)$ and hence a concentration argument gives on $\mathcal{S}$ that
		\begin{equation}\label{auxprop4}\P\big(\big|\Co'_{k+1,x}\big|>\delta\lambda p(a,1)\ut\,\big|\,\mathfrak{F}_{k\ut,(k+\bar{k})\ut,x}^0,\mathcal{W}_k\big)\,\geq\,1-e^{-\lambda^2\ut^2p(a,1)}.\end{equation}
	\end{enumerate}
	Repeating the same argument as the one used to bound $\mathcal{W}_k$ we conclude that on $\mathcal{S}$ 
	\[\P\big(\mathcal{W}_{k+1}\,\big|\,\mathfrak{F}_{k\ut,(k+\bar{k})\ut,x}^0,\mathcal{W}_k\big)\geq 1-C'e^{-\delta \lambda^2\ut^2p(a,1)}\]
	for some universal constant $C'$. We can repeat this procedure to obtain
	\[\P\big(\mathcal{W}_{k+s+1}\,\big|\,\mathfrak{F}_{k\ut,(k+\bar{k})\ut,x}^0,\mathcal{W}_{k+s}\big)\geq 1-C'e^{-\delta \lambda^2\ut^2p(a,1)}\]
	on $\mathcal{S}$ for any $0<s<\bar{k}$ and hence the result follows from a union bound.
\end{proof}

\medskip

 In the next section we show how Proposition \ref{survivaledge} allows the infection to survive locally around stars, thus giving enough time to the infection for it to spread throughout the network. Note however that the previous result requires for stars to already have a reservoir of infected neighbours, a scenario that is rarely seen whenever the infection reaches a new star (from another star, or a connector not in $\Co^0$).  It is intuitive though, that with some probability bounded from below a newly infected star can kickstart its reservoir from scratch, for instance, by not recovering for a whole unit of time. Even though such a bound is enough for our purposes, we present a more thorough result which holds under some mild assumptions on the function $a(\lambda)$.

\smallskip

\begin{proposition}\label{lemmakeylocal2}
	Fix $x\in\St$ and let $\mathcal{S}$ the event in  Proposition \ref{uboundCk}, with $\cc_1$ the constant appearing there. There is $C'>0$ such that for any $s\leq e^{\cc_1p(a,1)/2}$ and any $\lambda$ and $a(\lambda)$ satisfying $\lambda^2\ut^2 p(a,1)\geq C'$ we have
	\begin{equation}\label{keylocal2}\P\Big(|\Co'_{k+2,x}|\leq \cc_1\lambda p(a,1)\ut/16\;\Big|\;\mathfrak{F}_{s,(k+3)\ut,x}^0\;\Big)\,\leq\,\sfrac{C'(\lambda\ut\log(\frac{\cc_1}{16}\lambda p(a,1))+ 1)}{\cc_1\lambda^2\ut p(a,1)},\end{equation}
	on the event $\mathcal{S}\cap\{\bar{X}_{s}^x(x)=1\}$, where $k=\lceil s/\ut\rceil$ and \smash{$\mathfrak{F}_{s,(k+3)\ut,x}^0$} is as in Lemma \ref{lemmakeylocal}. Also, the event in \eqref{keylocal2} can be taken to depend only on the processes $\Rec^x$ and $\I^{x,y}$  with $y\in\Co_{k+1,x}$.
\end{proposition}

\begin{proof}
	In order to ease the notation denote $\P_s$ the conditional probability \smash{$\P(\cdot|\mathfrak{F}_{s,(k+3)\ut,x}^0)$} on the event $\mathcal{S}\cap\{\bar{X}_{s}^x(x)=1\}$. As a first step we claim that for any $$\tfrac{8\log(\frac{\cc_1}{16}\lambda p(a,1))}{\lambda p(a,1)}<\varphi\leq \ut$$ we have
	\begin{equation}\label{claimkeylocal2}\P_s\big(|\Co'_{k+1,x}|\leq \cc_1\lambda p(a,1)\varphi/8\big)\,\leq\,\varphi.\end{equation}
	We prove the claim by dividing the event within \eqref{claimkeylocal2} into the events \begin{align*}\mathcal{E}_1&:=\{|\Co'_{k+1,x}|\leq \cc_1\lambda p(a,1)\varphi/8,\,\Rec^x\cap[s,s+\varphi/2]\neq\emptyset\}\\[6pt]
	\mathcal{E}_2&:=\{|\Co'_{k+1,x}|\leq \cc_1\lambda p(a,1)\varphi/8,\,\Rec^x\cap[s,s+\varphi/2]=\emptyset\}\end{align*}
	and observe that $\P_s(\mathcal{E}_1)\leq\P(\Rec^x\cap[s,s+\varphi/2]\neq\emptyset)\leq 1-e^{-\varphi/2}\leq\varphi/2$. To bound $\P_s(\mathcal{E}_2)$ observe that on the event $\Rec^x\cap[s,s+\varphi/2]=\emptyset$, we have $\bar{X}_t^x(x)=1$ throughout $[s,s+\varphi/2]$, and hence a connector $y\in\Co_{k+1,x}$ is infected at time $(k+1)\ut$ as soon as $\I^{x,y}_0\cap[s,s+\varphi/2]$, which happens with probability $1-e^{-\lambda\varphi/2}\geq\frac{1}{4}\lambda\varphi$. Since all these events are independent and on $\mathcal{A}$ we have $|\Co_{k+1,x}|\geq\cc_1p(a,1) $, $|\Co'_{k+1,x}|$ dominates a binomial random variable with parameters \smash{$\frac{1}{4}\lambda\varphi$} and $\cc_1p(a,1)$ so we can use a Chernoff bound to deduce
	\[\P_s(\mathcal{E}_2)\leq e^{-\frac{1}{8}\cc_1p(a,1)\lambda\varphi}\leq\varphi/2,\]
	where the last inequality follows from the lower bound on $\varphi$. We conclude the claim by taking a union bound for $\mathcal{E}_1$ and $\mathcal{E}_2$. Hence we can stochastically bound ${8|\Co'_{k+1,x}|}/(\cc_1\lambda p(a,1))$ from below by a random variable $Y$ with $\P(Y=0)=q$, $\P(Y=\ut)=1-\ut$ and density $1$ over the interval $[q,\ut]$, where $q:={8\log(\frac{\cc_1}{16}\lambda p(a,1))}/(\lambda p(a,1))$.
	\smallskip
	
	\noindent Next, observe from Lemma \ref{lemmakeylocal} that there is some $C>0$ such that
	\[\P_s\bigg(\int_{k\ut}^{k\ut+\ut}(1-\bar{X}^x_s(x))ds>\sfrac{\ut}{2}\;\bigg|\;|\Co'_{k+1,x}|\bigg)\,\leq\,Ce^{-\lambda\ut |\Co'_{k+1,x}|/4}.\]
	as soon as  $\lambda\ut |\Co'_{k+1,x}|>C$. Taking expectation and using the previous claim we obtain
	\begin{align*}
	\P_s\bigg(\int_{k\ut}^{k\ut+\ut}(1-\bar{X}^x_s(x))ds>\sfrac{\ut}{2}\bigg)&\leq C\E_s\bigg(e^{-\frac{\cc_1\lambda^2\ut p(a,1)}{32} \cdot \frac{8|\Co'_{k+1,x}|}{\cc_1\lambda p(a,1)}}\bigg)\\[2pt]&\leq C\Big(q+\int_{q}^{\ut}e^{-\frac{\cc_1\lambda^2\ut p(a,1)}{32} \cdot s}ds+(1-\ut)e^{-\frac{\cc_1\lambda^2\ut^2 p(a,1)}{32}}\Big)\\[2pt]&\leq C\Big(q+\tfrac{32}{\cc_1\lambda^2\ut p(a,1)}+e^{-\frac{\cc_1\lambda^2\ut^2 p(a,1)}{32}}\Big)\\[2pt]
	&\leq\;\frac{C'(\lambda\ut\log(\frac{\cc_1}{16}\lambda p(a,1))+ 1)}{\cc_1\lambda^2\ut p(a,1)},
	\end{align*}
	where in the last inequality we have used that $\lambda^2\ut^2 p(a,1)$ is large and hence we can neglect the exponential term. Finally, the result follows by repeating the argument used in Proposition \ref{survivaledge} to deduce \eqref{auxprop4} under the assumption \smash{$\int_{k\ut}^{k\ut+\ut}(1-\bar{X}^x_s(x))ds>\sfrac{\ut}{2}$}.
\end{proof}

\subsection{Local survival} 

The following theorem establishes that under certain conditions, as soon as the set of stars is chosen such that there is local survival, the infection can spread on this set.
\begin{theorem}\label{theoslow}
	Let $\ut$ be as in Section \ref{localsurvival}. There is $M_{(iii)}>0$ depending on $\varkappa $ such that slow extinction and metastability for all $\lambda \in (0,1)$ hold if there is $a=a(\lambda) \in (0,1/2)$ satisfying
	\begin{equation}\label{localsurv}
	\lambda^2\ut^2p(a,1)>-M_{(iii)}\log(\lambda a).
	\end{equation}
	Moreover, under this assumption we have
	\begin{equation}
	\label{lowdensityetaposdelayed}
	\rho^-(\lambda)\;\geq\;c'\lambda a p(a,1),
	\end{equation}
	where $c'>0$ is a constant independent of $a$ and $\lambda$.
\end{theorem}

\begin{remark}
Theorem~\ref{theoslow} gives sharp lower bounds on the metastable exponent for the factor, preferential attachment and strong  kernel in the regime marked dark green in Figure~\ref{one}. More precisely, the choice $a(\lambda)=r \lambda^{2/\gamma(1-2\eta)} |\log \lambda|^{-(1-2\eta)/\gamma}$ for $\eta>0$, or $a(\lambda)=r \lambda^{2/\gamma} |\log \lambda|^{-1/\gamma}$ for $\eta\le 0$, with a small enough constant $r$, yields lower bounds of the order
\begin{equation}\label{LBMetastableDS}
\rho^-(\lambda)\gtrsim 
\left\{
\begin{array}{ll}
\lambda^{\frac {2\tau- 3-2\eta}{1-2\eta}}
\vert\log \lambda \vert^{-(\tau-2)(1-2\eta)} &\text{ if } \eta > 0, \\
\lambda^{2\tau- 3} \vert\log \lambda \vert^{-(\tau-2)} &\text{ if } \eta \le 0.
\end{array}
\right.
\end{equation}
The phase transition at $\eta=0$ 
originates from the change in behaviour of~$\ut$.
\end{remark}

To prove this result recall the definitions of $J_k$, $\Co_{k}$, $\Co_{k,x}$ and $\Co'_{k,x}$ given in Section \ref{localsurvival}, and define time periods $L_m$ of the form \smash{$L_m:=[\bar{k}\ut m,\bar{k}\ut(m+1))$} where \smash{$\bar{k}=\lfloor e^{\delta\lambda^2\ut^2p(a,1)/4}\rfloor$} is the amount of time units a star can survive locally according to Proposition~\ref{survivaledge}. Our proof relies on showing that if we find some fraction of the stars surviving locally within a period $L_m$, then with very large probability these stars will have enough time to repeatedly propagate the infection to new stars, which will in turn survive locally throughout $L_{m+1}$, thus creating a loop that gives slow extinction.
Throughout this proof we need the following definitions for all $m\geq 1$,
\begin{align*}\St_m&\,:=\,\left\{x\in\St \colon|\Co_{k,x}|>\cc_1p(a,1),\,\forall k \in \N\text{ with }k\ut\in L_{m-1}\cup L_m\right\}, \\[6pt]\tilde{\St}_m&\,:=\,\left\{x\in\St_m\colon |\Co'_{\bar{k}m,x}|\geq\delta \lambda p(a,1)\right\}, \\[6pt]\displaystyle\St_m'&\displaystyle\,:=\,\left\{x\in\tilde{\St}_m\colon \forall k \in \N\text{ with }k\ut\in L_m,\,\text{ we have }\mathcal{W}_k(x)\right\},\end{align*}
\noindent where $\cc_1$ is as in Proposition \ref{localstable} and $\delta$ is as in Proposition~\ref{survivaledge}. If $m=0$ then the definition of $\St_m$ changes only in that the $k\ut$ belong to $L_m$. In words, $\St_m$ are the stars which maintain a large amount of stable neighbours throughout $L_m\cup L_{m+1}$, while $\tilde{\St}_m$ are the stars among~$\St_m$ which also begin this time period with a large amount of stable neighbours infected. $\St'_m$ are the stars 
among~$\tilde{\St}_m$ that survive locally throughout $L_m$.\smallskip

Observe that a requirement that is common to all the results in Section \ref{localsurvival} is that stars have sufficiently many stable neighbours, so we need to guarantee existence of sufficiently many such stars at all times. This is the content of the following proposition.

\begin{proposition}\label{prop:skeletonstable}
Define $\mathfrak{F}$ the $\sigma$-algebra generated by the $\Rec^y$, $\U^{x,y}$ and $\mathcal{C}^{x,y}$ processes with $y\in\Co^0$ and $x\in\St$. There is some $\cc_2>0$(depending on $\lambda$ and $a$) such that defining \[\mathcal{A}\,:=\,\{|\St_m|>\tfrac{aN}{3},\,\,\forall m\leq e^{\cc_2N}\}\] then $\mathcal{A}$ is $\mathfrak{F}$-measurable and we have
$\P(\mathcal{A})>1-e^{-\cc_2N}$, for all sufficiently large $N$.\end{proposition}

\begin{proof}
	Recall from \eqref{localstable} that there is $\cc_0>0$ 
	fixed such that \[\P(|\Co_{k}|> \cc_0N \, \forall k\in\{0,\ldots,\lfloor e^{\cc_0 N}\rfloor\})\,\geq\,1-e^{-\cc_0 N}\]
	Fix a realisation of the $\Rec^{y}$ processes with $y\in\Co^0$ such that this event holds and fix $m$ such that $(m+1)\bar{k}\leq e^{\cc_0N}$. Using that $p(a,1)\ll \lambda^2\ut^2p(a,1)$ for small $\lambda$ we can adapt Proposition \ref{uboundCk} so that for any $x\in\St$
	\[\P\big( |\Co_{k,x}|>\cc_1p(a,1)\text{ for all }k\ut\in L_{m-1}\cup L_m\big)\,\geq\,1-e^{-\cc_1p(a,1)}>\tfrac{3}{4}\]
	and these events are independent for different $x$. It follows from a large deviation argument that there is $0<\cc'_2<\cc_0$ such that
	\[\P\left(|\St_m|>\tfrac{aN}{3}\right)\,\geq\,1-e^{-\cc'_2 N}\]
	so taking a union bound we obtain
	\[\P\big(|\St_m|>\tfrac{aN}{3},\,\,\forall m\leq \tfrac{1}{\bar{k}}e^{\frac{\cc'_2N}{2}}\big)\,\geq\,1-e^{-\frac{\cc'_2 N}{2}}\]
	and as $\bar{k}$ does not depend on $N$ it is enough to take $\cc_2<\cc_2'/2$ and $N$ sufficiently large.
\end{proof}

\medskip

Since at time $0$ all vertices are infected it follows that, for each star $x$, we have $\Co'_{0,x}=\Co_{0,x}$. In particular \smash{$\St_0=\tilde{\St}_0$} because for each $x\in\St_0$ we have
\[|\Co'_{0,x}|=|\Co_{0,x}|\geq \cc_1p(a,1)\geq\delta\lambda p(a,1).\]
On the event $\mathcal{A}$ this gives $|\tilde{\St}_0|>\frac{aN}{3}$ which will be sufficient to kickstart the slow extinction loop. Fix now $m\in\N$ and define the $\sigma$-algebra $\mathfrak{F}_m$ 
generated by $\mathfrak{F}$ and by the $\Rec^x$ and $\mathcal{I}^{x,y}$ processes up to time $\bar{k}\ut m$, with $y\in\Co^0$ and $x\in\St$. Finally, for $0<\cc'_1<\frac{1}{3}$ some value to be fixed later, consider the event
\smash{$\mathcal{B}_m\,:=\,\{|\tilde{\St}_m|>\cc'_1aN\},$}
which is $\mathfrak{F}_m$-measurable. Our main goal will be to show that there exists $\cc_2'>0$ independent of $N$ (but which might depend on $a$ and $\lambda$) such that on $\mathcal{A}\cap\mathcal{B}_m$,
\begin{equation}\label{recursiveedge}
\P\big(\mathcal{B}_{m+1}\,\big|\,\mathfrak{F}_m\big)\,\geq\,1-e^{-\cc_2' N}.
\end{equation}
To do so we split our proof depending on whether the event \[\mathcal{C}_{m}:=\{|\tilde{\St}_{m}\cap\St_{m+1}|>2\cc_1'aN\}\]
(which is $\mathfrak{F}_m$-measurable) is satisfied or not.\smallskip

For the case in which $\mathcal{C}_m$ is satisfied observe first that applying an analogous version of Proposition \ref{survivaledge} on $L_m$ and our hypothesis on $\lambda^2\ut^2p(a,1)$ we obtain that for any fixed star $x$, on $\{x\in\tilde{\St}_m\}$ we have
\[\P\big(x\in\St'_{m}\,\big|\,\mathfrak{F}_{m}\big)\,\geq1-Ce^{-\delta\lambda^2\ut^2p(a,1)/4}\geq\sfrac{2}{3}\]
where the bound follows from the study the restricted process $\bar{X}^x$ so that the events \smash{$\{x\in\St'_{m}\}_{x\in\tilde{\St}_{m}}$} can be taken to be independent for the purposes of our computations. Now, on the event $\mathcal{C}_m$ we have 
$$|\tilde{\St}_{m}\cap\St_{m+1}|>2\cc_1'aN$$ and using a large deviation bound we deduce that there is some constant $\cc_2'>0$ independent of $N$ such that on $\mathcal{C}_m$,
\begin{equation}\label{eqauxedgespreading}\P\big(|\St_{m}'\cap\St_{m+1}|>\tfrac{6}{5}\cc_1'aN\,\big|\,\mathfrak{F}_{ m}\big)\,\geq\,1-e^{-\cc_2' N}.\end{equation}
In order for $x\in\St_m'\cap\St_{m+1}$ to belong to $\tilde{\St}_{m+1}$ we only require that \smash{$|\Co'_{\bar{k}(m+1),x}|\geq\delta \lambda p(a,1)$}, but as in the proof of Proposition~\ref{survivaledge}, we can show that this occurs with a large probability (say, larger than $\frac{5}{6}$) and hence it follows from a large deviation argument that on $\mathcal{C}_m$
\begin{equation}\label{eq:result1slow}\P\left(\mathcal{B}_{m+1}\,\big|\,\mathfrak{F}_m\right)\,\geq\,1-e^{-\cc_2' N}.\end{equation}
For the case in which $\mathcal{C}_m$ is not satisfied observe that we are still assuming that $\mathcal{B}_m$ holds so we can obtain an analogous version of \eqref{eqauxedgespreading}, that is,
\begin{equation}\label{eqauxedgespreading2}\P\big(|\St_{m}'|>\tfrac{3}{5}\cc_1'aN\,\big|\,\mathfrak{F}_{ m}\big)\,\geq\,1-e^{-\cc_2' N}\end{equation}
on $\mathcal{B}_m$. {Differently from the previous case, the role of the stars in $\St_{m}'$ is to propagate the infection to stars in $\tilde{\St}_{m}\cap\St_{m+1}$, where the choice of this set is natural when trying to avoid dependencies with the event $\mathcal{D}_m:=\{|\St_{m}'|>\frac{3}{5}\cc_1'aN\}$.} The following proposition states that if $\mathcal{D}_m$ is satisfied, then with a large probability the infection propagates to sufficiently many stars where it survives locally.
\begin{proposition}\label{keyspreading}
	Fix $m\in\N$ and let $\mathcal{A}$, $\mathcal{C}_m$ and $\mathcal{D}_m$ as before. Define $\mathfrak{F}'_m$ as the $\sigma$-algebra generated by $\mathfrak{F}_m$ and the $\Rec^x$ and $\mathcal{I}^{x,y}$ processes on $L_m$, with $y\in\Co^0$ and $x\in\tilde{\St}_m$. Then there is $\cc_2'>0$ independent of $N$ and $m$ such that on $\mathcal{A}\cap\overline{\mathcal{C}}_m\cap\mathcal{D}_m$
	\[\P\left(\mathcal{B}_{m+1}\,\big|\,\mathfrak{F}'_m\right)\,\geq\,1-e^{-\cc_2' N}.\]
	Moreover, this event depends only on the $\Rec^x$, $\Rec^y$, $\U^{x,y}$ and $\mathcal{I}_0^{x,y}$ processes on $L_m$, with $y\in\Co^1$ and \smash{$x\notin\tilde{\St}_m$.}
\end{proposition}

\begin{proof}
Fix a realisation of the processes which generate $\mathfrak{F}'_m$ such that $\mathcal{A}$ and $\mathcal{D}_m$ are satisfied but $\mathcal{C}_m$ is not. Define for any given $k\in\N$ with $J_k\subseteq L_m$ the set
\[\Co^1_{k}\,:=\,\{y\in\Co^1\colon \Rec^y\cap[k\ut,(k+1)\ut)=\emptyset\},\]
where, by choice of $\ut$, each $y\in\Co^1$ belongs to $\Co^1_{k}$ independently with probability at least~$e^{-1}$, so using a large deviation argument we conclude that 
\smash{$\P(|\Co_k^1|>\tfrac{N}{18})\geq 1-e^{-N/18}$} and hence for $N$ large,
\begin{equation}\label{eq:stable1}\P\left(|\Co_k^1|>\tfrac{N}{18},\,\text{ for all }k\in\N, J_k\subseteq L_m\big|\,\mathfrak{F}'_m\right)\geq 1-e^{-N/18}.\end{equation}
Enlarge $\mathfrak{F}'_m$ so as to consider the processes $\Rec^y$ on $L_m$ with $y\in\Co^1$ and call $\mathcal{E}_m^1$ the event within the probability above. For $k$ as before define now
\[\Co^{2}_{k}\,:=\,\{y\in\Co^1_k \colon \text{ there is }x\in \St'_m\text{ with  }\U^{x,y}\cap[k\ut,(k+1)\ut)=\emptyset
\text{ and }\{x,y\}\in{\mathscr G}_{k\ut}\},\]
which will play the role of \textit{stable connectors} used by the infection to spread from $\St'_m$ towards new stars. To show that with a large probability the amount of stable connectors is large observe that from our choice of $\ut$ and the event $\mathcal{D}_m$,
\[\P(y\in\Co^2_{k}\,|\,\mathfrak{F}'_m)\geq 1-(1-\tfrac{e^{-1}}{N}p(a,1))^{\frac{3}{5}\cc_1'aN}\geq1-e^{-\frac{3\cc_1'}{5e}ap(a,1)}\geq\frac{3\cc_1'}{10e}ap(a,1),\]
on the event $\mathcal{D}_n\cap\{y\in\Co^1_{k}\}$, where we have used that $ap(a,1)<\frac{1}{2}$ for small $a$. Since the events \smash{$\{y\in\Co^2_k\}_{y\in\Co_k^1}$} are independent, it follows that there is $\cc_2'>0$ independent of $N$ such that on $\mathcal{D}_n\cap\mathcal{E}_m^1$
\[\P\left(|\Co^2_{k}|>\tfrac{\cc'_1}{180}ap(a,1)N\,\big|\,\mathfrak{F}'_m\right)\geq 1-e^{-\cc_2'N}\]
and using a union bound we conclude that on the same event
\begin{equation}\label{eq:stable2}\P(|\Co^2_{k}|>\tfrac{\cc'_1}{180}ap(a,1)N,\text{ for all }k\in\N, J_k\subseteq L_m\,|\,\mathfrak{F}'_m)\geq 1-e^{-\cc_2'N/2},\end{equation}
thus proving that there is a bounded fraction of stable connectors. Enlarge once again $\mathfrak{F}'_m$ so as to consider the processes $\U^{x,y}$ and $\mathcal{C}^{x,y}$ on $L_m$ with $y\in\Co^1$ and $x\in\St'_m$, and call $\mathcal{E}_m^2$ the event within the probability above; since the bounds in \eqref{eq:stable1} and \eqref{eq:stable2} are of the required form, all we need to show is that on $\mathcal{A}\cap\overline{\mathcal{C}}_m\cap\mathcal{E}^2_m$ we have
\[\P\left(\mathcal{B}_{m+1}\,\big|\,\mathfrak{F}'_m\right)\,\geq\,1-e^{-\cc_2' N}.\]
To do so fix any given $z\in\St_{m+1}\backslash\tilde{\St}_m$ and $k\in\N$ with $J_k\subseteq L_m$ and say that the event $\mathcal{E}^3_m(z,k)$ is satisfied if there is $y\in\Co^2_k$ such that
\begin{enumerate}
	\item $\U^{y,z}\cap[k\ut,(k+\tfrac{3}{4})\ut]\neq\emptyset$;
	\item $\I^{y,z}\cap [(k+\tfrac{3}{4})\ut,(k+\tfrac{7}{8})\ut]\neq\emptyset$;
	\item There is some $x\in\St'_m$ satisfying the conditions given in the definition of $\Co^2_k$, such that $\I^{x,y}\cap[k\ut,(k+\tfrac{3}{4})\ut]\neq\emptyset$ and there is infection event $t$ such that $X_t(x)=1$.
\end{enumerate}
For any $y\in\Co_k^2$ condition $(1)$ is satisfied with probability $1-e^{-\frac{3\ut}{4}(\kappa(y)+\kappa(z))}\geq\frac{\ut}{4}(\kappa(y)+\kappa(z))$. Assume that the first condition is satisfied and observe that in order for condition $(2)$ to occur it is enough that $\I_0^{y,z}\cap [(k+\tfrac{3}{4})\ut,(k+\tfrac{7}{8})\ut]\neq\emptyset$ and that at the first infection event in this interval the edge $\{y,z\}$ belongs to the network; since the edge was updated in the first half of the interval this has probability $(1-e^{-\frac{\lambda\ut}{8}})\frac{1}{N}p(a,1)\geq \frac{\lambda\ut p(a,1)}{16N}$. Finally, since $y\in\Co_k^2$ there is at least one star $x\in\St'_m$ satisfying the conditions given in the definition of $\Co^2_k$, and by definition of $\St'_m$, the set
\[\left\{t\in [k\ut,(k+1)\ut] \colon X_t(x)=1\right\}\]
has Lebesgue measure at least $\frac{1}{2}$, so its intersection with \smash{$[k\ut,(k+\tfrac{3}{4})\ut]$} has Lebesgue measure at least $\frac{1}{4}$ and hence $(3)$ is satisfied with probability at least \smash{$1-e^{-\lambda\ut/4}\geq\frac{\lambda\ut}{8}$}. On the event~$\mathcal{E}^2_m$ we have \smash{$|\Co^2_{k}|>\tfrac{\cc'_1}{180}ap(a,1)N$} so
\[\P\left(\mathcal{E}^3_m(z,k)\,\big|\,\mathfrak{F}'_m\right)\,\geq\,1-\left(1-\tfrac{\lambda^2\ut^3(\kappa(y)+\kappa(z))p(a,1)}{512N}\right)^{\tfrac{\cc'_1}{180}ap(a,1)N}\geq \tfrac{\lambda^2\ut^3(\kappa(y)+\kappa(z))\cc_1'ap^2(a,1)}{20000}.\]
Observe that condition $(1)$ in the definition of $\mathcal{E}^3_m(z,k)$ makes these events independent across different $k$ and hence by defining 
$$\mathcal{E}^3_m(z):=\bigcup_{k=m\bar{k}}^{(m+1)\bar{k}-1}\mathcal{E}^3_m(z,k),$$ we have
\begin{align*}\P\big(\mathcal{E}^3_m(z)\,\big|\,\mathfrak{F}'_m\big)&\geq\,1-\left(1-\sfrac{\lambda^2\ut^3(\kappa(y)+\kappa(z))\cc_1'ap^2(a,1)}{20000}\right)^{\bar{k}-1}\\[6pt]&\geq\,1-\exp\left(\tfrac{\lambda^2\ut^3(\kappa(y)+\kappa(z))\cc_1'ap^2(a,1)}{20000}\cdot\bar{k}\right)\end{align*}
but \smash{$\bar{k}=\lfloor e^{\delta\lambda^2\ut^2p(a,1)/4}\rfloor\geq(\lambda a)^{-M_{(iii)}\delta/4}$} so if $M_{(iii)}$ is large (depending on $\gamma$, $\varkappa $ and $\eta$ alone) then for all $\lambda$ and $a$ small this probability is bounded from below by $\frac{2}{3}$ regardless of $\cc_1'$.\smallskip

It follows from condition $(3)$ in the definition of $\mathcal{E}^3_m(z,k)$ that on this event the infection passes from some $x\in\St'_m$ to a connector $y\in\Co^2_k$ which then infects $z$ by condition~$(2)$. In particular, for any star \smash{$z\in\St_{m+1}\backslash\tilde{\St}_m$} the probability of getting infected at some given time in $L_m$ is bounded from below by $\frac{2}{3}$. Now, on $\mathcal{A}\cap\overline{\mathcal{C}}_m$ we have that both \smash{$|\tilde{\St}_{m}\cap\St_{m+1}|\leq2\cc_1'aN$} and \smash{$|\St_{m+1}|>\frac{aN}{3}$} so if we assume $\cc'_1<\frac{1}{12}$ then we necessarily have \smash{$|\St_{m+1}\backslash\tilde{\St}_m|>\frac{aN}{6}$}. As the events $\mathcal{E}_m^3(z)$ are independent across different $z$ a large deviation argument reveals that there is $\cc_3'>0$ independent of $N$ such that on $\mathcal{A}\cap\overline{\mathcal{C}}_n\cap\mathcal{E}_m^2$ we have
\begin{equation}\label{eq:stable3}
\P\big(|\{z\in\St_{m+1}\backslash\tilde{\St}_m,\,\mathcal{E}_m^3(z)\}|>\tfrac{aN}{12}\,\big|\,\mathfrak{F}'_m\big)\geq 1-e^{-\cc_3'N}.
\end{equation}
Finally, take $z\in\St_{m+1}\backslash\tilde{\St}_m$ satisfying $\mathcal{E}_m^3(z)$ so it gets infected at some time $s$ with $m\bar{k}\leq s\leq(m+1)\bar{k}-2$. It follows from our hypothesis $\lambda^2\ut^2p(a,1)>-M_{(iii)}\log(\lambda a)$ that $\lambda^2\ut^2p(a,1)\geq C'$ where $C'$ is as in Proposition \ref{lemmakeylocal2}. Even further, since $\lambda \ut\leq 1$ we also have that $\lambda\ut p(a,1)$ is very large so that
\[\sfrac{C'(\lambda\ut\log(\frac{\cc_1}{16}\lambda p(a,1))+ 1)}{\cc_1\lambda^2\ut p(a,1)}\leq\tfrac{1}{2}\]
and hence with probability at least $\frac{1}{2}$ we have \smash{$|\Co'_{k+1,z}|> \frac{\cc_1\lambda p(a,1)}{8}$} where $\cc_1$ is the constant appearing in Proposition \ref{uboundCk}. If this condition is satisfied we can use Proposition~\ref{survivaledge} to deduce that with probability at least $\frac{1}{2}$ the infection survives locally around $z$ throughout the rest of $L_m$ and then belongs to $\St_{m+1}$ (the details are analogous to the case where $\mathcal{C}_m$ was satisfied). By a large deviation argument, on the event \smash{$\mathcal{A}\cap\overline{\mathcal{C}}_n\cap\mathcal{E}_m^2$},
\[\P\big(|\tilde{\St}_{m+1}|>\tfrac{aN}{60}\,\big|\,\mathfrak{F}'_m\big)\geq 1-e^{-\cc_3'N},\]
for some $c_3'>0$ independent of $N$, and hence the result follows by taking $\cc_1'<\frac{1}{60}$.
\end{proof}

Using Proposition \ref{keyspreading} together with \eqref{eq:result1slow} and \eqref{eqauxedgespreading2} we conclude \eqref{recursiveedge}, and hence
\[\P\bigg(\bigcap_{m=0}^{e^{c_2'N/2}}\mathcal{B}_{m}\,\bigg|\,\mathfrak{F}\bigg)\,\geq\,1-e^{-c_2' N/2}.\]
on the event $\mathcal{A}$. Using Proposition \ref{prop:skeletonstable} we finally deduce that the event on the left (now unconditioned) has a probability which is bounded from below by $1-e^{-\cc_3'N}$ for some~$\cc_3'>0$ independent of~$N$ and we conclude that there is slow extinction.\medskip

To obtain the lower bound on the metastable density given by \eqref{lowdensityetaposdelayed} fix $t>0$ and observe that from the last proof we have deduced in particular that
\[\P\big(|\St'_{m}|>\cc_1'a N/2,\;\forall m\leq e^{c_2'N/2}\big)\,\geq\,1-e^{-\cc_2' N/2},\]
where from our construction, the event on the left can be taken independent from the processes involving vertices in $\Vo^{odd}$. Let $k=\lfloor t/\ut\rfloor-1$ and fix a realization of the graphical construction (not involving vertices in $\Vo^{odd}$) up to time $k\ut$ and of the $\U^{x,y}$, $\Rec^y$ and $\mathcal{C}^{x,y}$ processes with $y\in\Co^0$ up to time $k\ut+\ut$ such that $|\{x\in\St \colon \mathcal{W}_k(x)\}|>c_1'a N/2$. For any such $x$ we have $|\Co'_{k,x}|\geq \delta\lambda\ut p(a,1)$, and it is enough for $X_t(x)=1$ to have
\[\Rec^{x}\cap[k\ut,t]=\emptyset,\qquad\text{and}\qquad\exists y \in\Co'_{k,x},\;\I^{x,y}\cap[k\ut, t]\neq\emptyset.\]
The probability of this event is bounded from below by some constant $c>0$ by our assumption $\lambda^2 \ut^2p(a,1)\geq M_{(iii)}$. It follows that
$\P(|\{x\in\St \colon X_t(x)=1\}|>raN)\,\geq\,c$
for constant~$r$ and~$c$, so the result is a direct consequence of Lemma \ref{lemmalower}.
\pagebreak[3]

\subsection{Metastability}\label{sec-metastab}

In our previous work~\cite{JLM19}, which we recall was considering a different network dynamics, we could derive metastability of the process in the regimes of slow extinction from the following result:
		\begin{equation} \label{metastability}
		\limsup_{N\to \infty} \sup_{x\in \{1,\ldots,N\}} \P_x(t<T_{ext}<t_N) \underset{t\to \infty}{\to} 0,
		\end{equation}
		where $t_N=e^{\eps N}$, with $\eps>0$ sufficiently small. We obtained~\eqref{metastability} by infecting sufficiently many stars by time $t$, and then starting a survival strategy available for the infection in this regime.
In the current work, metastability can still be derived from~\eqref {metastability} but we cannot easily adapt its proof in the regimes where the quick strategies prevail. The problem is that our study of the survival strategy in these regimes relies heavily on the fact that the process is started from a constant proportion of stars infected, and cannot be adapted when we start with a sublinear number of infected stars. This is problematic as one cannot hope to first infect a constant proportion of the stars by a time $t$, which does not depend~on~$N$.%
\medskip

We therefore provide a separate argument, based on the duality of the process, to prove~\eqref{metastability}, where in this expression we have to run the contact process on the whole network and not only on the subset of vertices $\St\cup \Co$.
%
Fix a vertex~$x$ and write $(X_t^x)_{0\le t \le t_N}$ for the contact process starting from only $x$ infected, and write $(\check X_t)_{0\le t \le t_N}$ for the dual contact process, starting from everyone infected.  By construction (see Section \ref{sec:graphical}) the nonextinction event coincides with the event
	\[\{X_T^x\cap \check X_{t_N-T}\ne \emptyset\},\]
where $T$ is some time (which will be taken random), and where by a slight abuse of notation we identify $X_T^x$ and $ \check X_{t_N-T}$ with their respective sets of infected vertices. The advantage of writing nonextinction in this fashion is that $(\check X_t)_{0\le t \le t_N}$ has the same distribution as the original process started from full occupancy, and for this initial condition we have proved in all the regimes of slow extinction that with high probability the proportion of infected stars does not go below some level $\rho=\rho(\lambda)>0$ before time $t_N$. Such an event for the dual process roughly translates into the existence of valid infection paths starting from at least $\rho N$ different stars at time $T$, all the way to time $t_N$. 
Since this happens with high probability, it is enough for non extinction to occur to have at least one of those stars infected at time $T$. This will be achieved by choosing $T$ carefully.\bigskip

Fix $s\leq t_N$ and let $v\in\St$ be a star; we say that $v$ is \textit{$s$-good} if there is a valid path starting from $v$ at time $s$, and ending at time $t_N$. Define $E_N(s,t_N)$ as the event
	\[E_N(s,t_N):=\{\text{there are at least }\rho N\text{ $s$-good stars}\}\]
	Observe that this event relies only on ${\mathscr G}_s$ and on the graphical construction between times $s$ and $t_N$. In particular,
	\[\limsup_{N\to\infty}\P_x(E^c_N(s,t_N))=\limsup_{N\to\infty}\P_x(|\check X_{t_N-s}\cap\St|< \rho N)=0\]
	where the first equality follows from duality and stationarity of the environment at 
	time~$0$, and the second equality follows from the results on previous sections. Observing that the probability is independent of $x$ we can strengthen this statement to obtain 
	\begin{equation}\label{eq1meta}\limsup_{N\to \infty} \sup_{x\in \{1,\ldots,N\}} \P_x(E^c_N(s,t_N))=0.\end{equation}
 Fix now some arbitrary $r\in\N$, and define $T$ as the smallest integer $k$ such that:
 \begin{itemize}
 	\item There is a connector $u\in\Co$ which is infected at time $k-1/2$ and does not recover on the time interval $[k-1/2,k]$. We write $A^k_u$ for this event. 	
 	\item events $A^k_{u,y_1},\ldots, A^k_{u,y_r} $ described below hold for at least $r$
 	stars $y_1,\ldots, y_r \in \St$.
 	\end{itemize}
 	For each star $y\in\St$ we denote by $A^k_{u,y}$ the event that $u$ infects $y$ during 
 	$[k-1/2,k]$ and
 	\begin{itemize}
 		\item $y$ does not recover on $[k-1/2,k]$.
 		\item the edge $\{u,y\}$ is updated exactly twice on $[k-1/2,k]$; once on $[k-1/2,k-3/8]$ and once on $[k-2/8,k-1/8]$.
 		\item If the edge $\{u,y\}$ is present at time $k-1/2$, then either after the first update it becomes present, in which case it transmits the infection to $y$ during the interval $[k-3/8,k-2/8]$, or it is present after the second update, in which case the infection is transmitted during the interval $[k-1/8,k]$.
 		\item If the edge is absent at time $k-1/2$ then after the first update it is present, and transmits the infection to $y$ during the interval $[k-3/8,k-2/8]$.
 	\end{itemize}
We now provide a few simple observations regarding these events:
\begin{enumerate}
	\item If we condition on $T_{ext}>k-1$ and on the configuration at time $k-1$, then the probability of succeeding in finding the infected connector $u$, namely the probability of $\bigcup_{u\in\Co} A^k_u,$ is uniformly bounded below by some positive constant.
	\item There exists $0<\cc_1<\cc_2$ independent of $N$, $u$, $y$ and $r$ such that conditionally on~$A^k_u$, and on the state of the edge $\{u,y\}$ at times $k-1/2$ and $k$, the probability of~$A^k_{u,y}$ lies within $[\frac{\cc_1}{N},\frac{\cc_2}{N}]$.
	\item Conditionally on $A^k_u$, on the configuration at time $k-1/2$, \emph{and additionally on the graph configuration at time $k$}, the events $\{A^k_{u,y}\}_{y\in \St}$ are independent.
\end{enumerate}

Checking these observations is not difficult, and essentially relies on the independence of the different Poisson processes, and the fact that an edge between a vertex and a connector is updated at a rate lying between $\varkappa a^{-\gamma\eta}$ and $\varkappa$, and after each update it is present with probability between $\frac{1}{N}p(a/2,1/2)$ and $\frac{1}{N}p(a,1)$. We only provide some details for the bounds on the probability of $A_{u,y}^k$ conditional on $\{u,y\}$ being present at time $k$, which is the most delicate case in the second observation.
\smallskip

Assume first that $\{u,y\}$ is absent at time $k-1/2$. In order to satisfy \smash{$A_{u,y}^k$} and have $\{u,y\}$ present at time $k$ we require for $\{u,y\}$ to be present after both of its updatings, an event with probability \smash{$\Theta(\frac{1}{N^2})$} where the implied constants depend only on $p$, $a$, $\lambda$ and $\eta$. In order for $\{u,y\}$ to be present at time~$k$ it needs to update at least once and be present after its last update in $[k-1/2,k]$, which occurs with probability $\Theta(\frac{1}{N})$. We conclude that the conditional probability of having \smash{$A_{u,y}^k$} is of order 
\smash{$\Theta(\frac{1}{N})$}.\pagebreak[3]\medskip

Assume next that $\{u,y\}$ is present at time $k-1/2$ and notice that in order to satisfy \smash{$A_{u,y}^k$} and have $\{u,y\}$  present at time $k$ the best possible scenario is that $\{u,y\}$  becomes present after its second update and transmits the infection during $[k-1/8,k]$, which occurs with probability $\Theta(\frac{1}{N})$. On the other hand, since $\{u,y\}$  is present at time $k-1/2$, the most likely scenario in which $\{u,y\}$  is still present at time $k$ is the one in which it does not update throughout $[k-1/2,k]$, an event which has probability $\Theta(1)$. We conclude again that the conditional probability of having $A_{u,y}^k$ is of order $\Theta(\frac{1}{N})$.\medskip

Using the previous observations we can now prove \eqref{metastability}. Indeed, using the stopping time~$T$ and the event $E_N$ we can bound $\P_x(t<T_{ext}<t_N)$ from above by
\begin{equation}\label{meta2}
\P_x(t\le T\wedge T_{ext})+ \P_x(T<t, E_N(T,t_N)^{\rm c})+ \sum_{k\leq t}\P_x(T=k,  E_N(k,t_N),t<T_{ext}<t_N).
\end{equation}
For the first expression in \eqref{meta2} it follows from observation (1) that for all $k\leq t$, with probability uniformly bounded from below at least one of the events $A_u^k$ is satisfied, so conditioning on $A_u^k$ for some $u$ and on the configuration at time $k-1/2$, all we need is to show that at least $r$ stars $y$ satisfy $A_{u,y}^k$. Since there are $\Theta(N)$ stars, and the events $A^k_{u,y}$ are (conditionally) independent with probability $\Theta(1/N)$, the probability that we obtain at least $r$ of these events satisfied is indeed bounded below by some constant $\cc$ depending on $r$ but not on $x$, $k$ or $N$. In particular, in order for $\{t\leq T\wedge T_{ext}\}$ to occur the previous events have to fail for all $k\leq t$, and hence 
\[\sup_{x\in\{1,\ldots,N\}}\P_x(t<{T\wedge T_{ext}})\,\leq\,(1-\cc)^{\lfloor t\rfloor},\]
for all $N$ large. As a result, $\lim_{t\to\infty}\limsup_{N\to\infty}\sup_{x\in\{1,\ldots,N\}}\P_x(t<{T\wedge T_{ext}})=0$. For the second expression in \eqref{meta2} notice that
\[\P_x(T<t, E_N(T,t_N)^{\rm c})\,\leq\,\sum_{k\leq t}\P_x(E_N(k,t_N)^{\rm c}),\]
 and hence for any fixed $t$ we can use \eqref{eq1meta} $\lceil t\rceil$ times to conclude \[\limsup_{N\to\infty}\sup_{x\in\{1,\ldots,N\}}\P_x(T<t, E_N(T,t_N)^{\rm c})=0.\]
 Finally, for the terms $\P_x(T=k,  E_N(k,t_N),t<T_{ext}<t_N)$ appearing in \eqref{meta2}, fix any realization of the process until time $k-1/2$, any graph configuration $\mathscr G_k$, and any realization of the graphical construction between times $k$ and $t_N$ such that we have both $T>k-1$ and $E_N(k,t_N)$ and additionally $A_u^k$ is satisfied for some connector $u$, which we assume without loss of generality that $u$ is the smallest such connector. Calling 
 \smash{$\tilde{\P}$} the probability conditional on these realizations we have
 \[\tilde{\P}(T=k,  E_N(k,t_N),t<T_{ext}<t_N)\,\leq\,\tilde{\P}(S^t_k\geq r,\,S_k^i=0)\]
 where $S^t_k$ stands for the total amount of stars $y$ satisfying $A_{u,y}^k$, while $S^i_k$ stands for the ones satisfying \smash{$A_{u,y}^{{k}}$} but for which there is also a valid infection path starting at $(y,k)$ and ending at time $t_N$. Indeed, observe that as soon as $A_{u,y}^k$ occurs then $y$ is infected at time~$k$, and hence if there is a valid infection path starting at $(y,k)$ and ending at time $t_N$, then necessarily $T_{ext}\geq t_N$. Now, it follows from observations (2) and (3), and a straightforward computation that
 \[\tilde{\P}(S^t_k\geq r,\,S_k^i=0)\,=\,\tilde{\P}(S_k^i=0\,\big|\,S^t_k\geq r)\tilde{\P}(S^t_k\geq r)\,\leq\,\left(\tfrac{\cc_2(1-\rho)}{\cc_1\rho+\cc_2(1-\rho)}\right)^r\tilde{\P}(S^t_k\geq r)\]
 and since $\tilde{\P}(S^t_k\geq r)=\tilde{P}(T=k)$ we deduce
 \[\P_x(T=k,  E_N(k,t_N),t<T_{ext}<t_N)\,\leq\,\left(\tfrac{\cc_2(1-\rho)}{\cc_1\rho+\cc_2(1-\rho)}\right)^r\P_x(T=k,\,T_{ext}>k-1),\]
 so that
 \[\sum_{k\leq t}\P_x(T=k,  E_N(k,t_N),t<T_{ext}<t_N)\,\leq\,\left(\tfrac{\cc_2(1-\rho)}{\cc_1\rho+\cc_2(1-\rho)}\right)^r.\]
 Putting together the bounds for all terms appearing in \eqref{meta2} we conclude
 \[\limsup_{t\to\infty}\limsup_{N\to \infty} \sup_{x\in \{1,\ldots,N\}} \P_x(t<T_{ext}<t_N)\leq \left(\tfrac{\cc_2(1-\rho)}{\cc_1\rho+\cc_2(1-\rho)}\right)^r.\]
 Since $r$ is arbitrary we finally conclude \eqref{metastability}.

\section{Upper bounds}

\subsection{Upper bound by duality}

In this section, we provide an upper bound  for the upper metastable density
available for all kernel types
in the case $\eta\leq 0$. This upper bound is essentially the same as in the static case, which is an indication that  the temporal variability of the network does not significantly improve the spread of the infection.

\begin{proposition}\label{static_upperbound}
Consider the evolving network model with $\eta\leq0$. There is a constant $c>0$ depending on the various parameters but not on $\lambda$ such that the upper metastable density $\rho^+(\lambda)$ satisfies
\begin{equation}\label{static_upperdens}
\rho^+(\lambda)\le \left\{
\begin{array}{ll}
c \lambda^{\frac 1 {3-\tau}}& \mbox{ for the factor kernel if }\tau\le \frac 5 2,\\
c \lambda^{2\tau-3}& \mbox{ for the strong or preferential attachment kernel,}\\
c \lambda^{\tau-1}& \mbox{ for the weak kernel.} 
\end{array}
\right.
\end{equation}
\end{proposition}

\begin{remark}
In the quick direct spreading  phase, our upper bound actually matches the lower bound for $\rho^-(\lambda)$ up to a constant multiplicative factor. In the local survival phase, our upper bound matches the lower bound up to a multiplicative factor of order $\vert\log\lambda\vert^{\tau-2}$. Note that the precise order of the metastable density has been obtained on the (static) configuration model (for which we recall the connection probabilities follow the factor kernel) in \cite{MVY13} and for a preferential attachment model in \cite{VHC17}. This precise order often matches our lower bounds, but on the configuration model with $\tau>3$ it involves a logarithmic factor with a different exponent. In all cases we believe the precise order on our slowly evolving networks (namely $\eta\le 0$) matches the one on the corresponding static network, and the current section is a strong indication that is is indeed the case. 
	However, obtaining this precise order (up to a multiplicative constant) would involve substantial additional difficulties and technicality (as should already be clear from the mentioned works \cite{MVY13,VHC17}), and we did not pursue this direction further.
\end{remark}

In order to prove Proposition \ref{static_upperbound} we  introduce first a general theorem, which gives upper bounds in terms of the kernel and some additional parameters:
	\begin{theorem}\label{generalstatic} Consider the evolving network model with $\eta\leq0$, satisfying \eqref{condp}. Let $R\in\N$, and $a\in [0,1]$ be functions of $\lambda$ satisfying
\begin{equation}\label{condgeneralstatic}\lim_{\lambda\downarrow0}a(\lambda)=0,\;\text{ and }\qquad\lim_{\lambda\downarrow0}\frac{R\ln\big(R a^{-\gamma}\ln(1/\lambda)\big)}{a^{-\gamma}}=0,\end{equation}
and assume further that $R\geq c_R$ and $a\geq c_a\lambda^{\frac{2}{\gamma}}$ for some sufficiently large $c_a$ and $c_R$ (independent of $\lambda$ and $N$). Then there is $c'$ depending on the various parameters but not on $\lambda$ such that the upper metastable density $\rho^+(\lambda)$ satisfies
\begin{equation}\label{eqgeneralupper}
\rho^+(\lambda)\le a+2\lambda^{c'R}+\sum_{l=1}^{R-1}\big(2\lambda(1+8\varkappa )\big)^l\mathscr{F}_1(a,p,l)+\big(2\lambda(1+8\varkappa )\big)^R\mathscr{F}_2(a,p,R)
\end{equation}
where
\[\mathscr{F}_2(a,p,R):= 
\int_{[a,1]^R \times[0,1]} p(x_0,x_1) \ldots p(x_{R-1},x_R) \, \mathrm d x_0\ldots\mathrm d x_R,
\]
and for $l\in\N$,
\[
\mathscr{F}_1(a,p,l):=
\int_{[a,1]^l \times[0,a]} p(x_0,x_1) \ldots p(x_{l-1},x_l) \, \mathrm d x_0\ldots\mathrm d x_l.\]
\end{theorem}

The proof of Theorem \ref{generalstatic} relies on a careful definition of the local-neighbourhood of the uniformly chosen starting vertex $x_0\in \{1,\ldots, N\}$ of the infection. In the following, we drop $N$ from the notation, and write $V$ for the vertex set $\{1,\ldots, N\}$, and $E_t$ for the edge set of \smash{$\mathscr G_t$}.
Let $R\in\N$ and $a\in [0,1]$ be as in the statement and define $T=c_T R\ln(1/\lambda)$ for some constant $c_T$ to be chosen later. We define the local neighbourhood of $x_0$ up to time $T$, 
distance~$R$, and pruning at level $a$, as the evolving graph \smash{$(V_{[0,T]}^{R,a},E^{R,a}_t)_{t\le T}$}, which is the subgraph of $(V,E_t)_{t\le T}$ obtained by keeping only vertices and edges on a path with at most~$R$ edges starting at $x_0$, passing no vertex in $\{1,\dots, \lceil aN\rceil-1\}$. 
More precisely, we define, for~$t\ge 0$,
\[
E^{R,a}_t:=
\left\{ \{x,y\} \in E_t \colon
\begin{array}{l}
 \exists\, l\le R, \exists\, 0\le t_1 \le \ldots \le t_l=t, \\
 \exists\, x_1, \ldots, x_{l-1}\in \{\lceil aN\rceil, \ldots, N\}, x_l:=y\\
 \forall i\in \{1, \ldots, l\} \colon \{x_{i-1},x_i\} \in E_{t_i} 
\end{array}
\right\}.
\]
and $$E^{R,a}_{[0,T]}= \bigcup_{s\le T} E^{R,a}_s.$$
Denote by  $V^{R,a}_0$ the set of vertices in the connected component of $x_0$ in 
$(V,E^{R,a}_0)$, and by \smash{$V^{R,a}_{[0,T]}$} the set of vertices in the connected component of $x_0$ in \smash{$(V,E^{R,a}_{[0,T]})$}.
%
The contact process on the evolving graph $(V,E_t)_{t\le T}$ together with its graphical representation induces a contact process on the local neighbourhood \smash{$(V^{R,a}_{[0,T]},E^{R,a}_t)_{t\le T}$} by restricting the infection events to edges in \smash{$E^{R,a}_t$}. Actually, the former  dominates the latter and it is easy to see that  both contact processes coincide up to time $T$, or up to the infection time of some vertex in $\{1,\dots, \lceil aN\rceil-1\}$, or up to the infection time of some vertex $x$ at distance $R$ from $x_0$ (here the distance is understood in \smash{$(V^{R,a}_{[0,T]},E^{R,a}_{[0,T]})$}).
As neither $a$, $R$ nor $T$ depend on $N$, as $N\to \infty$ the graph \smash{$(V^{R,a}_{[0,T]},E^{R,a}_{[0,T]})$} is a tree with high probability so we only consider this case. 
 Recall that for any vertex $x\geq\lceil aN\rceil$ its expected degree is bounded from above by 
$$\int_0^1 p(x/N,y) dy\le c_2 (x/N)^{-\gamma}\leq c_2a^{-\gamma}.$$ 
 We define $M= 2 c_2 a^{-\gamma}$, which satisfies $M\le 1/(8\lambda^2)$ by our choice of $a$ when taking $c_a^\gamma\ge 16 c_2$. We define events $A_1$ to $A_5$ as follows:\smallskip
\begin{itemize}
\item $A_1=\big\{x_0\in\{1,\dots, \lceil aN\rceil-1\}\big\}$.
\smallskip
\item $A_2$ is the event that  
\begin{itemize}
\item either for some $t\le T$, a vertex in $(V^{R,a}_{[0,T]},E^{R,a}_t)$ has degree larger than $M$, 
\item or a vertex in the graph $(V^{R,a}_{[0,T]},E^{R,a}_{[0,T]})$ has degree larger than $(1+2\varkappa )TM$.
\smallskip
\end{itemize}
\item $A_3$ is the event that none of $A_1$ or $A_2$ occurs, but the infection on the evolving graph $(V^{R,a}_{[0,T]},E^{R,a}_t)_{t\le T}$ is still alive at time $T$.\smallskip

\item $A_4$ is the event that none of $A_1$ or $A_2$ occurs, but the infection on the evolving graph $(V^{R,a}_{[0,T]},E^{R,a}_t)_{t\le T}$ reaches some vertex in $\{1,\dots, \lceil aN\rceil-1\}$.\smallskip

\item $A_5$ is the event that none of $A_1$ or $A_2$ occurs, but the infection on the evolving graph $(V^{R,a}_{[0,T]},E^{R,a}_t)_{t\le T}$ reaches some vertex $x$ at distance $R$ from $x_0$. 
\end{itemize}
If none of these events occur, then the infection is extinct by time $T$.
So the upper metastable density $\rho^+(\lambda)$ satisfies
\begin{equation}
\label{SumA1A5}
\rho^+(\lambda) \le \sum_{i=1}^5 \P(A_i),
\end{equation}
and we now want to bound the probability of these events happening.
\pagebreak[3]
\medskip

To begin with, observe that the probability of $A_1$ is bounded by $a$. In order to bound the probability of $A_2$, we say a vertex $x\in \{ \lceil aN \rceil,\ldots, N\}$ is \emph{good} if its degree at times $t\le T$ stays bounded by $M$, and its aggregated degree (namely its degree in $(V, \bigcup_{t\le T} E_t)$) is bounded by $(1+2\varkappa  T) M$.

\begin{lemma} \label{GoodDegreeCondition}
There exists a constant $c>0$ independent of $\lambda$, such that every vertex $x\in \{ \lceil aN \rceil,\ldots, N\}$ has probability at least $1-e^{-cM}$ of being good. 
\end{lemma}
\begin{proof}
Observe that the degree of vertex $x \in \{ \lceil aN \rceil,\ldots, N\}$ at a given time is a {Poisson binomial random variable} with expectation bounded by $M/2$. As $1/2<e^{-1/2}$, using a classical Chernoff bound, the probability that it exceeds $e^{-1/2}M$ is bounded by $e^{-cM}$ for some constant $c>0$.\smallskip%

Suppose now that the degree of $x$ is larger than $M$ at some time $t<T$. Then, using that the updating rate of every edge adjacent to $x$ is at most 
$2\varkappa $, we obtain that with probability bounded from zero, at least a proportion $e^{-1/2}$ of these edges remain intact during $(t,t+\varkappa /2]$.
 Writing $\deg(x,t)$ for  the degree of $x$ at time $t$, for some constant $C>0$,
\begin{align*}
\P(\exists t<T, \deg(x,t)\ge M) &\le C \P\big(\exists k\in\{0,\ldots, \left\lfloor 2T/{\varkappa }\right\rfloor+1\},\deg(x,\varkappa  k /2)\ge e^{-1/2}M\big) \\
 &\le C \left(\tfrac {2T} {\varkappa } +2\right)e^{-cM}.
\end{align*}
Restating the third assumption in \eqref{condgeneralstatic} as  
\smash{$\lim_{\lambda\to0}\frac{R\ln(MT)}{M}=0$} we deduce in particular that $\lim_{\lambda\to0}{\ln(T)}/{M}=0$ so the upper bound above can be written in the form $e^{-cM}$ by decreasing the value of the constant $c>0$. Finally, the total aggregated degree of $x$ during the time interval $[0,T]$ is a Poisson binomial random variable with expectation bounded by $(1+2\varkappa  T) M/2$, so the probability that it exceeds $(1+2\varkappa  T) M$ is also exponentially small in $M$, proving Lemma~\ref{GoodDegreeCondition}.
\end{proof}

We may now apply Lemma~\ref{GoodDegreeCondition} repeatedly and obtain that the probability of $A_2$, namely the probability that there exists a vertex in $V^{R,a}_{[0,T]}$ failing to be good, is bounded above by
\[
e^{-c M} \sum_{i=0}^R \left((1+2\varkappa )TM\right)^i.
\]
Using the hypothesis \smash{$\lim_{\lambda\to0}\frac{R\ln(MT)}{M}=0$} we obtain that \smash{$\P(A_2)\le e^{-c' M}\le e^{-c'T}\leq\lambda^{c'R}$}, for some $c'>0$ independent of $\lambda$, where we have used that $M\geq T$ for small $\lambda$.
\medskip

Suppose now that we condition on a realisation of the evolving graph \smash{$(V^{R,a}_{[0,T]},E^{R,a}_t)_{t\le T}$} that is not in $A_1\cup A_2$. In particular, no degree in this evolving graph ever exceeds the value $1/(8\lambda^2)$. This allows to define supermartingales on it as follows. Write $d(x,y)$ the distance of two vertices $x$ and $y$ in \smash{$(V^{R,a}_{[0,T]},E^{R,a}_{[0,T]})$}. For every vertex \smash{$x\in V^{R,a}_{[0,T]}$}, we let $l(x)= d(x,x_0)$ and we denote by $b(x)$ the birth time of $x$, defined as smallest $t$ with \smash{$x\in V^{R,a}_{[0,t]}$}. 
Defining  
\[
M_t^{(x)}:=e^{t/4} \sum_{y\in V^{R,a}_{[0,T]}} \one_{X_t(y)=1} (2\lambda)^{d(x,y)},
\]
we get that $(M^{_{(x)}}_t \colon t\le T)$ is a supermartingale (see Lemma 5.1 in \cite{MVY13}, where these martingales were already introduced), starting from \smash{$M_0^{_{(x)}}=(2\lambda)^{l(x)}$} as initially only $x_0$ is infected. Moreover, if $x$ is ever infected after its birth time, then this martingale has to reach value at least $e^{b(x)/4}$. We easily deduce the following,\smallskip
\begin{enumerate}
\item[(a)] for $x\in V ^{R,a}_{[0,T]}$, the probability that $x$ gets infected is at most $(2\lambda)^{l(x)} e^{-b(x)/4}$.\smallskip

\noindent
\item[(b)] the probability that the infection (on the evolving graph $(V^{R,a}_{[0,T]},E^{R,a}_t)_{t\le T}$) is not extinct by time $T$ is bounded from above by \smash{$(2\lambda)^{-R} e^{-T/4}$}.
\end{enumerate}
Using (b) and the definition of $T$ we conclude that
\[\P(A_3)\leq \lambda^{-R} e^{-T/4}=\lambda^{R(c_T/4-1)}\]
where the exponent is positive when taking $c_T$ large. 
\medskip

It remains to bound $\P(A_4)$ and $\P(A_5)$, which we treat together.
Using (a), we have,
\begin{eqnarray}
\label{boundA4}
\P(A_4) &\le& \E\bigg[\sum_{\heap{x\in V^{R,a}_{[0,T]}}{ x<aN}} (2\lambda)^{l(x)} e^{-\frac {b(x)} 4}\bigg],\\
\label{boundA5}
\P(A_5) &\le& (2\lambda)^R \; \E\bigg[\sum_{\heap{x\in V^{R,a}_{[0,T]}}{l(x)=R}} e^{-\frac {b(x)} 4}\bigg].
\end{eqnarray}
We now have to bound the weighted sums appearing on the right-hand side. The following lemma shows that this weighted sum on the evolving graph can be replaced by a simpler expression concerning only the graph at time zero.

\begin{lemma}\label{evolving_to_static}
We have, for $l\in \{1, 2, \ldots, R-1\},$
\[
\E\bigg[\sum_{\heap{x\in V^{R,a}_{[0,T]}}{l(x)=l, x<aN}} e^{-\frac {b(x)} 4}\bigg] \le (1+8\varkappa )^l \;\E\Big[\big|\{x \in V^{R,a}_0 \colon  l(x)=l, x<aN\}\big|\Big]
\]
as well as
\[
\E\bigg[\sum_{\heap{x\in V^{R,a}_{[0,T]}}{l(x)=R}} e^{-\frac {b(x)} 4}\bigg] \le (1+8\varkappa )^R \;\E\Big[\big|\{x \in V^{R,a}_0\colon l(x)=R\}\big|\Big].
\]
\end{lemma}
\begin{proof}
First, observe that $8\varkappa $ is the exact value of $\E[\sum_{k\ge 1} e^{-\frac {Z_k} 4}]$ if $(Z_k)_{k\ge 1}$ are the ordered points of a Poisson point process of intensity $2\varkappa $ on the positive half-axis. 
We extend the definition of the birth time $b(x)$ to vertices $x\notin V^{R,a}_{[0,T]}$ by letting  $b(x)=\infty$ in this case. Then,
\begin{eqnarray*}
\E\bigg[\sum_{\heap{x\in V^{R,a}_{[0,T]}}{l(x)=l, x<aN}} e^{-\frac {b(x)} 4}\bigg]
&\le& \sum \E\left[ \one_{b(x_1)\le \ldots \le b(x_l)\le T, \forall i\le l \colon \{x_{i-1},x_i\} \in E_{b(x_i)}} e^{-b(x_l)}\right] \\[-5mm]
&\le& \sum \E\bigg[ \prod_{i=1}^l e^{-\frac{\delta(x_{i-1},x_i,b_{i-1})} 4}\bigg],
\end{eqnarray*}
where the sum is over different vertices $x_1,\ldots, x_l$ such that $x_l<a$ and $x_1,\ldots, x_{l-1}$ are larger than $a$, and $\delta(x,y,t):=\inf\{s\ge t \colon \{x,y\} \in E_s\} - t$.\pagebreak[3] \medskip

Observe that, conditionally on $\{b_{i-1}\le T\}$, the random variable $\delta(x_{i-1},x_i,b_{i-1})$ is independent of $\delta(x_{j-1},x_j,b_{j-1})$ for $j<i$, and has the same law as $\delta(x_{i-1},x_i,0)$. We can therefore bound this term further, by
\begin{eqnarray*}
\sum \prod_{i=1}^l\E\left[e^{-\frac{\delta(x_{i-1},x_i,0)} 4}\right]
&=& \sum \prod_{i=1}^l (1+4(\kappa_{x_{i-1}}+\kappa_{x_i})) \P\big(\{x_{i-1},x_i\} \in E_0\big) 
\end{eqnarray*}
\begin{eqnarray*}
&\le& (1+8\varkappa )^l \sum \prod_{i=1}^l \P\big(\{x_{i-1},x_i\} \in E_0\big), \\[-4mm]
\end{eqnarray*}
Here we get the equality by first conditioning on the updating times of the edges $\{x_{i-1},x_i\}$, which are Poisson point processes of intensity $\kappa_{x_{i-1}}+\kappa_{x_i}\le 2\varkappa $. The first inequality of Lemma~\ref{evolving_to_static} follows. 
We do not detail the proof of the second inequality, which is similar.
\end{proof}

Finally, it remains to estimate the right-hand side of the inequalities appearing in Lemma~\ref{evolving_to_static}.
Simple sum-integral comparisons give
\[
\E\Big[\big|\{x \in V^{R,a}_0 \colon  l(x)=l, x<a\}\big|\Big]\le 
\int_{[a,1]^l \times[0,a]} p(x_0,x_1) \ldots p(x_{l-1},x_l) \, \mathrm d x_0\ldots\mathrm d x_l.
\]
and\\[-5mm]
\[
\E\Big[\big|\{x \in V^{R,a}_0 \colon  l(x)=R\}\big|\Big]\le 
\int_{[a,1]^R \times[0,1]} p(x_0,x_1) \ldots p(x_{R-1},x_R) \, \mathrm d x_0\ldots\mathrm d x_R.
\]
Adding up the bounds for $\P(A_1)$, $\P(A_2)$, $\P(A_3)$, $\P(A_4)$, and $\P(A_5)$, we conclude the bound appearing in Theorem \ref{generalstatic}. In the following subsections we deduce the upper bounds appearing in Proposition \ref{static_upperbound} by applying the previous theorem for a particular choice of $a$ and $R$ for each kernel and by computing $\mathscr{F}_1(a,p,l)$ and $\mathscr{F}_2(a,p,R)$.

\subsection*{Application of Theorem~\ref{generalstatic} to the factor kernel.}

In the case of the factor kernel $p(x,y)= \beta x^{-\gamma} y^{-\gamma}$, we easily obtain
\[
\mathscr{F}_1(a,p,l)\le
\left\{ 
\begin{array}{ll}
 \frac \beta {(1-\gamma)^2} \left(\frac \beta {1-2\gamma}\right)^{l-1}a^{1-\gamma} &\text{if }\gamma<1/2, \\
\frac \beta {(1-\gamma)^2} \left(\beta\log(1/a)\right)^{l-1}a^{1-\gamma}&\text{if }\gamma=1/2 ,\\
\frac \beta {(1-\gamma)^2} \left(\frac {\beta a^{1-2\gamma}} {2\gamma-1}\right)^{l-1} a^{1-\gamma}&\text{if }\gamma>1/2,
\end{array}
\right.
\]
as well as
\[
\mathscr{F}_2(a,p,R)\le
\left\{ 
\begin{array}{ll}
 \frac \beta {(1-\gamma)^2} \left(\frac \beta {1-2\gamma}\right)^{R-1} &\text{if }\gamma<1/2, \\
\frac \beta {(1-\gamma)^2} \left(\beta\log(1/a)\right)^{R-1}&\text{if }\gamma=1/2 ,\\
\frac \beta {(1-\gamma)^2} \left(\frac {\beta a^{1-2\gamma}} {2\gamma-1}\right)^{R-1}&\text{if }\gamma>1/2.
\end{array}
\right.
\]
Observe that by choosing $a=c_a\lambda^{\frac{2}{\gamma}}$ the expression \smash{$\big(2\lambda(1+8\varkappa )\big)^l\mathscr{F}_1(a,p,l)$} decreases exponentially fast with $l$ when taking sufficiently small $\lambda$ and $\gamma<\frac{2}{3}$ so in that case
\[\sum_{l=1}^{R-1}\big(2\lambda(1+8\varkappa )\big)^l\mathscr{F}_1(a,p,l)\,\lesssim\, \big(2\lambda(1+8\varkappa )\big)\mathscr{F}_1(a,p,1)\,\lesssim\, \lambda a^{1-\gamma}.\]
Similarly,
$(2\lambda(1+8\varkappa ))^R\mathscr{F}_2(a,p,R)\,\lesssim\,\lambda^{(\frac{2}{\gamma}-3)R-1}$
which is of a smaller order if we assume $c_R$ to be sufficiently large. Since we can choose $c_R$ as large as needed, the dominating term on the right hand side of  \eqref{eqgeneralupper} is of order $\lambda a^{1-\gamma}$.\pagebreak[3]

In the case $\gamma\ge 2/3$, we can choose a larger value of $a$, namely $a=c_a \lambda^{\sfrac 1 {2\gamma-1}}\geq c_a\lambda^{\frac{2}{\gamma}}$ with
\[
c_a\ge \left(\tfrac {4(1+8\varkappa ) \beta}{2\gamma-1}\right)^{\frac 1 {2\gamma-1}},
\]
so that
\[
\sum_{l=1}^{R-1}\big(2\lambda(1+8\varkappa )\big)^l\mathscr{F}_1(a,p,l) \lesssim \lambda a^{1-\gamma} \sum_{l\ge 1} \frac 1 {2^{l-1}} \lesssim \lambda a^{1-\gamma}.
\]
We also choose $R=c_R \log(1/\lambda)$ with
\smash{$c_R\ge \sfrac \gamma{ (2\gamma-1)\log 2}$} so that $\lambda^R\mathscr{F}_2(a,p,R)\lesssim \lambda a^{1-\gamma}$. By choosing $c_R$ large we obtain from Theorem  
~\ref{generalstatic} that \smash{$\rho^+(\lambda)\lesssim \lambda a^{1-\gamma}\lesssim \lambda^{\frac{\gamma}{2\gamma-1}}=\lambda^{\frac{1}{3-\tau}}$} so we deduce~\eqref{static_upperdens} in the case of the factor kernel.

\subsection*{Application of Theorem~\ref{generalstatic} to the strong and preferential attachment kernel.}

As the strong kernel is dominated by the preferential attachment kernel $p(x,y)=\beta (x\wedge y)^{-\gamma} (x\vee y)^{\gamma-1}$, it suffices to look at the latter. As the integral estimates are more involved in this case, they are best treated 
by introducing the relevant operator. We suppose \smash{$\gamma\ne \frac 1 2$}, as in the case 
\smash{$\gamma=\frac 12$} the preferential attachment kernel agrees with the factor kernel and hence in that case 
\smash{$\rho^+(\lambda)\lesssim \lambda a^{1-\gamma}\lesssim \lambda^{{2}/{\gamma}-1}=\lambda^{2\tau-3}$}. For the case $\gamma\neq\frac{1}{2}$ we rely on the following lemma.
\begin{lemma} \label{operator_bound}
We have, for some constant $c>0$ that may depend on the parameters of the model but not on $a$ or $\lambda$,
\[
\mathscr{F}_1(a,p,l)\le
\left\{
\begin{array}{ll}
a^{1-\gamma} c^l a^{(\frac 12-\gamma)(l-1)}
 &\text{if } \gamma>1/2, \\
a^{1-\gamma} c^l &\text{if } \gamma<1/2.
\end{array} 
\right.
\]
Moreover,
\[
\mathscr{F}_2(a,p,R)\le
\left\{
\begin{array}{ll}
\big(c a^{\frac 1 2 - \gamma}\big)^{R}
 &\text{if } \gamma>1/2, \\
c^R &\text{if } \gamma<1/2.
\end{array} 
\right.
\]
\end{lemma}
\begin{proof}
We proceed as in~\cite{EM14}, which contains a similar calculation, 
and we suppose for convenience $l\ge 2$ and $R\ge 2$. On the space $L^2(a,1)$, we introduce the operator $T$ by
\[
Tg(x)=\int_a^1 p(x,y) g(y) \de y
\]
Letting $f(x)=\int_a^1 p(x,y) \de y$ and $g(x)=\int_0^a p(x,y) \de y$, we may now write, for $l\ge 2$,
\[
\mathscr{F}_1(a,p,l)
\;=\;\int_a^1 f(x_1) (T^{l-2}g)(x_1) \, \de x_1\;\le\; \|f\|_2\ \|g\|_2\ \vertiii{T}^{l-2},\]
where $T$ is the operator norm. We do not treat the easier case $\gamma<\frac 12$, where the operator norm of $T$ is bounded independently of $a$. If $\gamma>\frac 12$, we get, as in \cite{EM14},
\[
\vertiii{T}\le \frac {\beta\sqrt 2} {2\gamma -1} a^{\frac 12-\gamma}.
\]
Moreover, easy computations give
\[
\|f\|_2\le \frac {c_2}{\sqrt{2\gamma-1}} a^{\frac 12-\gamma}, \qquad 
\|g\|_2\le \frac {\beta (1-a^{2\gamma-1})^{1/2}}{(1-\gamma)\sqrt{2\gamma-1}} a^{1-\gamma}.
\]
The first part of Lemma~\ref{operator_bound} follows. The second part is similar, starting with the observation
\[
\mathscr{F}_2(a,p,R)
=\int_a^1 f(x_1) \big(T^{R-2}(f+g)\big)(x_1) \, \de x_1.
\]
\ \\[-10mm]\end{proof}
Fix then $a=c_a\lambda^{\frac{2}{\gamma}}$ and $R=c_R$ for sufficiently large $c_R$ and $c_a$. Observing that multiplying $\lambda^l$ to the bound for $\mathscr{F}_1(a,p,l)$ in Lemma \ref{operator_bound} we get a geometric sum whose main contribution is the first term, that is,
\[\sum_{l=1}^{R-1}\big(2\lambda(1+8\varkappa )\big)^l\mathscr{F}_1(a,p,l) \lesssim \big(2\lambda(1+8\varkappa )\big)\mathscr{F}_1(a,p,1) \lesssim \lambda a^{1-\gamma}=\lambda^{2\tau-3}\]
The result then follows from Theorem \ref{generalstatic} and Lemma \ref{operator_bound} by observing that for sufficiently large $c_R$ we have
\smash{$a+2\lambda^{c'R}+\big(2\lambda(1+8\varkappa )\big)^R\mathscr{F}_2(a,p,R)\lesssim a=\lambda^{\frac{2}{\gamma}}=\lambda^{2\tau-2}\le \lambda^{2\tau-3}$}.

\subsection*{Application of Theorem~\ref{generalstatic} to the weak kernel.}
{For the weak kernel we can perform an easier calculation and get the following result.
\begin{lemma} \label{operator_bound_2}
For some constant $c>0$ that may depend on the parameters of the model but not 
on $a$ or $\lambda$, we have
\[
\mathscr{F}_1(a,p,l)\le
c^l a^{1-\gamma l}\]
and
\[
\mathscr{F}_2(a,p,R)\le c^R a^{1-\gamma R}.
\]
\end{lemma}
Choosing $a=c_a\lambda^{\frac{1}{\gamma}}$ and $R=c_R$ for sufficiently large $c_a$ and $c_R$ and perform similar computations as in the previous cases to deduce \smash{$\rho^{+}(\lambda)\lesssim \lambda a^{1-\gamma}=a=\lambda^{\frac{1}{\gamma}}=\lambda^{\tau-1}$}, thus we conclude the bound in Proposition~\ref{static_upperbound} for the weak kernel.
}

\subsection{Upper bound by the supermartingale technique}
In this section we prove fast extinction (in the relevant phases) or provide an upper bound for the upper metastable density, based on adaptations of the supermartingale technique developed in \cite{JLM19}. More precisely, we prove the following result:

\begin{proposition}
\label{supermartingale_upperbound}
\ \\[-5mm]
\begin{itemize}
\item[(a)]
Consider the factor kernel and $\eta\ge 0$. 
\begin{itemize}
\item[(i)] If $\eta> \frac12$ and $\tau>3$, there is fast extinction.
\item[(ii)] If $\eta<\frac 12$ or $\tau<3$, the upper metastable density satisfies,
\[
\rho^+(\lambda)\le 
\left\{
\begin{array}{ll}
c \lambda^{\frac 1{3-\tau}} &\mbox{if }\tau\le \frac 5 2+\eta \mbox{ and }\tau<3,\\
c \lambda^{\frac {2\tau-3-2\eta}{1-2\eta}} &\mbox{if }0\le \eta\le \frac 1 2 \mbox{ and }\tau\ge \frac 5 2+\eta.
\end{array}
\right.
\]
\end{itemize}

\item[(b)]
Consider the {strong or} preferential attachment kernel and $\eta\ge 0$. 
\begin{itemize}
\item[(i)] If $\eta> \frac12$ and $\tau>3$, there is fast extinction.
\item[(ii)] If $\eta<\frac 12$ or $\tau<3$, the upper metastable density satisfies,
\[
\rho^+(\lambda)\le 
\left\{
\begin{array}{ll}
c \lambda^{\frac {\tau-1}{3-\tau}} &\mbox{if }\tau\le 2+2\eta \mbox{ and }\tau<3,\\
c \lambda^{\frac{2\tau-3-2\eta}{1-2\eta}} &\mbox{if }0\le \eta\le \frac 1 2 \mbox{ and }\tau\ge 2+2\eta.
\end{array}
\right.
\]
\end{itemize}
\item[(c)]
Consider the weak kernel and $\eta\ge 0$. Then the upper metastable density satisfies
$$\rho^+(\lambda)\le c \lambda^{\tau-1} \log(1/\lambda)^{\tau}.$$
\end{itemize}
\end{proposition}

\begin{remark}
	In the local survival phases, as well as in the phase $\eta>0$ for the weak kernel, our upper bound for $\rho^+(\lambda)$ differ from our lower bound on $\rho^-(\lambda)$ by a multiplicative factor of logarithmic order. In those phases the precise order for the metastable density is unclear.
\end{remark}

Similarly as in \cite{JLM19}, we first prove {a general theorem obtained by the introduction of an appropriate supermartingale, and then apply it to different kernels.}

\begin{theorem}
\label{teoupper_edge}
Let $\eta\ge 0$ and define {$\kappa(x)=\kappa_{\lfloor xN\rfloor}$} and $\psi \colon (0,1)\to(0,\infty)$ as
\[
\psi(x)= \int_0^1 \frac {p(x,y)} {(\kappa(x)+\kappa(y))^2} \, \mathrm d y.
\]
For $\lambda>0$, suppose there exists some $a=a(\lambda)>0$ and some non-increasing and integrable function $s:[a,1]\rightarrow [1,\infty)$ (or $s:(0,1]\rightarrow [1,\infty)$ in the case $a=0$) such that
\begin{eqnarray}\label{cond1}
6 \lambda^2 \psi(x)&\le& 1 \quad \forall x\in (a,1],\\
\label{cond2} 3\lambda \left(1+\frac {\lambda}{2\varkappa ^2}\right)\int_0^1 p(x,y) s(y\vee a) \mathrm d y&\;\leq\;&s\left(x\right) \quad \forall x\in (a,1].
\end{eqnarray}
Then, 
\begin{enumerate}
\item if  $a=0$, then there is some $\omega=\omega(\lambda)$ such that, for large~$N$,
\begin{equation}\label{resextime_edge}
\E\big[T_{\rm ext}\big]\;\leq\; \omega \log N. \notag
\end{equation}
In particular there is fast extinction.
\end{enumerate}\pagebreak[3]
\begin{enumerate}
\item[(2)] if $a>0$, then there exists $\omega=\omega(\lambda)>0$ such that, for all $N$ and all $t\ge0$, we have
\begin{equation}
\label{resdens}
I_N(t)\;\leq\;a+\frac{1}{s(a)}\int_a^1 s(y)\, dy + \frac \omega t+\frac 1 N.
\end{equation}
In particular, if there is metastability, then the upper metastable density satisfies
\begin{equation}
\label{upperdens_edge}
 \rho^+(\lambda) \le a(\lambda)+\frac{1}{s(a(\lambda))}\int_{a(\lambda)}^1 s(y)\, dy.
\end{equation}
\end{enumerate}
\end{theorem}

\begin{remark}
Heuristically, Condition~\eqref{cond1} can be opposed to Condition~\eqref{timescaledef} and interpreted as the request that there is no local survival effect available for vertices in $[aN,N]$. The fact that Condition~\eqref{cond2} then suffices to get upper bounds, is an indication that the process does not spread the infection more than an infection process where each edge $xy$ of the complete graph would transmit the infection at rate $3\lambda (1+\lambda/2\varkappa ^2)p_{x,y}$, namely a multiplicative factor $3(1+\lambda/\varkappa ^2)$ more than the average rate $\lambda p_{x,y}$.
\end{remark}

{We postpone the proof of the theorem to the appendix, and check here how it can be used to deduce Proposition~\ref{supermartingale_upperbound}. In order to use Theorem~\ref{teoupper_edge} we have to determine a level~$a$ and a function $s$ satisfying the hypotheses of the theorem and providing the required upper bounds.} The way we do this is led by two complementary principles. First, a purely analytic approach, when the conditions required by Theorem~\ref{teoupper_edge} lead to optimal or natural choices. Second, the comparison of the approach with the upper bounds. In particular, the contact process should show a subcritical behaviour as long as none of the vertices $1,\ldots, \lfloor aN\rfloor$ is infected, while the stars $\mathscr S=\{1,\ldots, \lfloor a N\rfloor\}$ introduced in the lower bounds 
are typically infected in the metastable state. We thus expect that $a$ in the upper bounds is larger than $a$ in the lower bounds
and aim to make them of the same order. \medskip
\pagebreak[3]

Recalling the definition of the kernels, to avoid cluttered notation, we henceforth assume that $\beta=1$.
To simplify the computations, we also use the notation $f\lesssim g$ if the positive functions $f$ and $g$ of $\lambda$ satisfy that $f/g$ is bounded when $\lambda\to 0$, and similarly $f\gtrsim g$ or $f\asymp g$. We further recall that $\tau\in (2,\infty)$ and $\gamma\in (0,1)$ are linked by the 
relation~$\tau=1+1/\gamma$.
\medskip

\subsection*{Application of Theorem~\ref{teoupper_edge} to the factor kernel.}

We now apply Theorem~\ref{teoupper_edge} to the factor kernel to deduce part~(a) of Proposition~\ref{supermartingale_upperbound}. In these settings, we have 
\[
\psi(x)\le x^{-\gamma+2\gamma \eta} \int_0^1 y^{-\gamma} \de y \le \sfrac {1} {1-\gamma} \,
x^{-\gamma+2\gamma \eta}
\]
and thus in order to satisfy~\eqref{cond1}, it suffices to satisfy
\begin{equation}\label{cond1bis}
\sfrac 6{1-\gamma} \, \lambda^2 \,\left(1 \vee a^{-\gamma+2\gamma \eta}\right) \le 1.
\end{equation}
This is automatically satisfied for small $\lambda$ when $\eta\ge 1/2$, while for $\eta<1/2$, it requires \smash{$a\gtrsim \lambda^{\frac 2{\gamma(1-2\eta)}}.$} For small $\lambda$, the expression $1+\frac \lambda{2\varkappa ^2}$ is bounded by $4/3$, thus~\eqref{cond2} is implied~by
\begin{equation}\label{cond2bis}
\left(4\lambda \int_0^1 p(x,y) s(y\vee a) \de y\right) \le s(x) \quad \forall x\ge a.
\end{equation}
For the factor kernel the left-hand side factorises, and if we consider the natural choice $s(x)=x^{-\gamma}$, then~\eqref{cond2bis} is equivalent to $\Delta_a\le 1$, where 
\smash{$\Delta_a:=4\lambda \int_0^1 y^{-\gamma} (y\vee a)^{-\gamma} \de y.$}
\smallskip

We now consider the different phases separately.
\smallskip

\noindent
{\bf (1) The case $\eta\ge 1/2$ and $\tau>3$.}
In that case \eqref{cond1} is automatically satisfied for small~$\lambda$, and $\Delta_0=\lambda c$ for some finite $c$, so, with $s(x)=x^{-\gamma}$, \eqref{cond2bis} is also satisfied for $a=0$ and $\lambda\le 1/c$. Thus, for small values of $\lambda>0$, the expectation $\E[T_{ext}]$ is logarithmic in $N$ and there is fast extinction.
\medskip

\noindent
{\bf (2) The case $\eta\ge 1/2$ and $\tau=3$.}
Again, \eqref{cond1} is satisfied for small $\lambda$. Considering $s(x)=x^{-\gamma}$, we now have \smash{$\Delta_a\sim \lambda \log\frac 1 a$}. It follows that~\eqref{cond2bis} is satisfied with \smash{$a(\lambda)=e^{-r \lambda^{-1}}$} with some well-chosen $r>0$, and we deduce \smash{$\rho^+(\lambda)\le e^{-r' \lambda^{-1}}$} for some $r'>0$.
\medskip

\noindent
{\bf (3) The case $\tau<3$ and $\eta\ge \tau-\frac 52$.}
Choosing again $s(x)=x^{-\gamma}$, we have $\Delta_a\equiv \lambda a^{1-2\gamma}$, so~\eqref{cond2bis} is satisfied if we take \smash{$a(\lambda)=r a^{\frac 1 {2\gamma-1}}$} for some well-chosen $r>0$. In that phase this makes~\eqref{cond1} automatically satisfied as well for small $\lambda$. We deduce
\smash{$\rho^+(\lambda)\lesssim \frac 1 {s(a(\lambda))}=\lambda^{\frac 1 {3-\tau}}.$}
\medskip

\pagebreak[3]

\noindent
{\bf (4) The case $\eta<\frac 1 2$ and $\tau> \frac 5 2+\eta$.}
In this case the most restrictive constraint for the choice of $a(\lambda)$ comes from Inequality~\eqref{cond1}, which leads 
naturally 
to the choice
\smash{$a(\lambda) = r \lambda^{_{\frac 2 {\gamma(1-2\eta)}}}$,}
for some large enough constant $r>0$. 
Inequality~\eqref{cond2bis} is then automatically satisfied (for small $\lambda$) if we choose $s(x)= x^{-\gamma}$, but other choices of scoring functions are still allowed and can give better upper bounds. 
We choose $s(x)=x^{-\mathfrak{c}}$ where $\mathfrak{c}=1-\gamma/2-\gamma\eta$. Observe that, for this choice
$$\gamma<\frac{2}{3+2\eta}\Longrightarrow\gamma<\mathfrak{c},\quad\text{ and }\quad \eta<\frac{1}{2}\Longrightarrow\gamma+\mathfrak{c}>1$$
so it suffices to check \eqref{cond2bis} for $x=1$, but $4\lambda \int_{0}^1y^{-\gamma}s(y\vee a)\, \de y\;\asymp\; \lambda a^{1-\gamma-\mathfrak c} \asymp 1,$
and one can readily check that the left-hand side is indeed bounded by 1 if $r$ is chosen large enough. We then obtain as an upper bound for the upper metastable density
\[
\rho^+(\lambda)\lesssim \sfrac 1 {s(a(\lambda))}=a^\mathfrak{c}=r^{\mathfrak c} \lambda^{\frac {2-\gamma-2\gamma\eta} {\gamma(1-2\eta)}}= r^{\mathfrak c} \lambda^{\frac{2\tau-3-2\eta}{1-2\eta}}.
\]

\subsection*{Application of Theorem~\ref{teoupper_edge} to the strong or preferential attachment kernel.}

Similarly to what was done with the factor kernel, we choose the function $s$ depending on the parameters $\gamma$ and $\eta$. In this case, however, our choice of $s$ needs to be a little more subtle:%
\medskip

\noindent
{\bf (1)\ The case $\eta\ge \frac 12$ and $\tau>3$.}
As with the factor kernel, \eqref{cond1} is satisfied for small~$\lambda$ with $a=0$. We choose the scoring function \smash{$s(x)=x^{-\gamma'}$} for some $\gamma'\in(\gamma,1-\gamma)$ whence~\eqref{cond2bis} and~\eqref{cond2} are satisfied for small $\lambda$.
Using the first part of the theorem we deduce $\E[T_{ext}]\lesssim \log N$, and fast extinction.
\smallskip

\noindent
{\bf (2)\ The case $\tau<3$ and $\eta>\frac \tau 2 - 1$.}
This phase corresponds to the quick indirect spreading phase in the lower bounds. Looking at the definition of stars in that case, it is reasonable to fix \smash{$a= r \lambda^{\frac{2}{2\gamma-1}}$} for some (large) fixed $r$. Then we can readily check that~\eqref{cond1} is satisfied for small $\lambda$. We further choose the scoring function $s$ differently as
\smash{$s(x)\:=\;x^{\gamma-1}+\sfrac{5\lambda}{2\gamma-1} x^{-\gamma}$}, 
which is larger than $1$ and decreasing, because in this case we necessarily have $\gamma>\frac{1}{2}$. To argue why this may be reasonable scoring function, one may think that the first term corresponds to a direct infection of vertex $x$ to a stronger vertex $y<x$, and the second to an indirect infection, that passes through some vertex $z>x$ first.
We now check \eqref{cond2} and, for simplicity, use 
$p(x,y)\; \leq \;
x^{-\gamma}y^{\gamma-1}+x^{\gamma-1}y^{-\gamma}$
so that
\begin{align*}
4\lambda \int_0^1 & \, p(x,y)  s(y)\, dy\\
& \leq\;4\lambda x^{\gamma-1}\int_a^1\left(y^{-1}+\sfrac{5\lambda}{2\gamma-1}y^{-2\gamma}\right)dy\;+\;4\lambda x^{-\gamma}\int_a^1\left(y^{2\gamma-2}+\sfrac{5\lambda}{2\gamma-1}y^{-1}\right)\, dy
\end{align*}
\begin{align*}
\phantom{4\lambda \int_0^1} &\leq\;\sfrac{50\lambda^2a^{1-2\gamma}}{(2\gamma-1)^2}x^{\gamma-1}\;+\;\sfrac{5\lambda}{2\gamma-1}x^{-\gamma} \;=\;\sfrac{50r^{1-2\gamma}}{(2\gamma-1)^2}x^{\gamma-1}\;+\;\sfrac{5\lambda}{2\gamma-1}x^{-\gamma},
\end{align*}
where the last inequality is actually only valid for small $\lambda$, using that $\lambda\log(1/a)$ can then be taken as small as wanted. Taking $r$ large (depending on $\gamma$ alone) we have that the latter expression is smaller than $s(x)$, giving that our choice of $s$ satisfies \eqref{cond2}. Applying the second part of the theorem we finally obtain the upper bound,
\[
\rho^+(\lambda) \le a +\frac 1 {S(a)} \int_{a}^1S(y) \de y \lesssim \lambda^{\frac{\tau-1}{3-\tau}}.
\]
\noindent
{\bf (3)\ The case $\eta<\frac12$ and $\tau> 2+2\eta$. } This case is similar to the case of the factor kernel. As before let \smash{$a(\lambda) = r \lambda^{\frac 2 {\gamma(1-2\eta)}}$},
choosing $r$ large enough so that~\eqref{cond1} is satisfied. We further define the scoring function as 
the monomial $s(x)=x^{-\mathfrak{c}}$ where $\mathfrak{c}=\frac{2-\gamma-2\gamma\eta}{2}$ as before. Observe that in this case we have
\begin{equation}
\label{eq11}4\lambda\int_{a}^1p(x,y)S(y)dy\; \leq \;4\lambda x^{\gamma-1}\int_{a}^xy^{-\gamma-\mathfrak{c}}dy\;+\;4\lambda x^{-\gamma}\int_x^1 y^{\gamma-\mathfrak{c}-1}dy
\end{equation}
where $\eta<\frac{1}{2}$ shows that $1-\gamma-\mathfrak{c}=\frac{\gamma}{2}(2\eta-1)<0$ so that the first integral is bounded by
$$\tfrac{4\lambda}{\gamma+\mathfrak{c}-1}a^{1-\gamma-\mathfrak{c}}x^{\gamma-1}\;=\;\tfrac{4r^{1-\gamma-\mathfrak{c}}}{\gamma+\mathfrak{c}-1}x^{\gamma-1}\;\leq\;\tfrac12x^{-\mathfrak{c}}$$

\noindent as soon as $r$ is taken large enough. For the second integral, it is no longer true that $\gamma-\mathfrak{c}<0$ always as with the factor kernel, so we need to consider three possibilities:\smallskip
\begin{itemize}
	\item $\mathbf{ \gamma-\mathfrak{c}<0:}$ The integral is bounded by $\frac{4\lambda}{\mathfrak{c}-\gamma}x^{-\mathfrak{c}}\leq\frac12x^{-\mathfrak{c}}$ if $r$ is large enough.
	\smallskip
	
	\item $\mathbf{ \gamma-\mathfrak{c}=0:}$ The integral is
	$-4\lambda x^{-\gamma}\ln(x)\leq -4\lambda \ln(a)x^{-\mathfrak{c}}\leq\frac12x^{-\mathfrak{c}}$
	when $\lambda$ is small.\smallskip
	
	\item $\mathbf{ \gamma-\mathfrak{c}>0:}$ The integral is bounded by $\frac{4\lambda}{\gamma-\mathfrak{c}}x^{-\gamma}$. We have $\lambda x^{-\gamma}\leq r^{\mathfrak{c}-\gamma}x^{-\mathfrak{c}}$ by choice of \smash{$a(\lambda)=r \lambda^{\frac{2}{\gamma(1-2\eta)}}$.} Hence, if $r$ is large, 
	$$4\lambda x^{-\gamma}\int_x^1 y^{\gamma-\mathfrak{c}-1}dy\;\leq\;\tfrac{4\lambda}{\gamma-\mathfrak{c}}x^{-\gamma}\;\leq\;\tfrac{4r^{\mathfrak{c}-\gamma}}{\gamma-\mathfrak{c}}x^{-\mathfrak{c}}\;\leq\;\tfrac{1}{2}x^{-\mathfrak{c}}.$$
\end{itemize}
Using the bounds for both integrals in \eqref{eq11} we conclude that $s(x)=x^{-\mathfrak{c}}$ satisfies Condition~\eqref{cond2}. Applying the second part of the theorem, we obtain the bound
$$ \rho^+(\lambda) \le a +\tfrac 1 {S(a)} \int_{a}^1S(y) dy\;\leq\;\tfrac{2r^\mathfrak{c}}{1-\mathfrak{c}}\lambda^{\frac{2-\gamma-2\gamma\eta}{\gamma(1-2\eta)}}= \tfrac{2r^\mathfrak{c}}{1-\mathfrak{c}}\lambda^{\frac{2\tau-3-2\eta}{1-2\eta}}$$
where the computations are the same as in the case of the factor kernel. 

\subsection*{Application of Theorem~\ref{teoupper_edge} to the weak kernel.} {We choose $s(x)=x^{-1}$ and note that~\eqref{cond1} holds automatically. Then~\eqref{cond2} can be verified with
\smash{$a(\lambda)\geq c (\lambda \log(1/\lambda))^\frac1\gamma$} and this implies the given upper bound for $\rho^+(\lambda)$.}
\pagebreak[3]

\section{Appendix: Proofs of Theorem~\ref{teoupper_edge}}

As in \cite{JLM19}, we work with the \emph{wait-and-see process}~$(Y_t)$, which can take values 0 or 1 on the vertices $x\in V$ but also on the edges $\{x,y\}$ of the complete graph, with the understanding $Y_t(x)=1$ if $x$ is \emph{infected} at time $t$, and $Y_t(x,y)=1$ if the edge $\{x,y\}$ is \emph{revealed} at time $t$. The process evolves as follows:
\begin{itemize}
\item Updates occur as in the original network and have the effect of turning the status of the potential edges updating to unrevealed.
\item Each revealed edge incident to an infected vertex  and an uninfected vertex, transmits the infection to the uninfected vertex at rate $\lambda$.
\item Each unrevealed edge $\{x,y\}$ incident to at least one infected vertex, gets revealed at rate $\lambda p(x,y)$. Simultaneously, it transmits the infection to its potentially uninfected incident vertex.
\item Vertices recover at rate 1.
\end{itemize}
The wait-and-see process can be coupled with the original process, with initially all edges unrevealed and the same set of infected vertices, so that at any time $t$ an infected vertex for $X_t$ is also infected for $Y_t$, and each revealed edge in $Y_t$ indeed belongs to the edge set of~$\mathscr G_t$. This coupling was provided in \cite{JLM19} for the vertex updating model, and the adaptation to the edge updating model is straightforward. As a consequence, it suffices to prove Theorem~\ref{teoupper_edge}
for the vertex component of $Y$, when initially all edges are unrevealed.%
\bigskip%

We continue with the proof of Theorem~\ref{teoupper_edge} and  suppose $a$, $s$ are given satisfying the hypotheses of this theorem. If $a$ is nonzero, then we extend $s$ to $[0,1]$ by defining $s(x)=s(a)$ for $x\le a$. To each vertex $x\in\{1,\ldots, N\}$, we associate $s_x:=s(x/N)$, which serves as a base score if $x$ is infected but not surrounded by any revealed edge. We further introduce
\[
Q_t(x):=\sum_{y: Y_t(x,y)=1}\frac \lambda {\kappa_{x,y}^2}, \qquad R_t(x):=\sum_{y: Y_t(x,y)=1}\frac \lambda {\kappa_{x,y}}.
\]
As we work here with $\eta\ge 0$, the updating rates are lower bounded by $2\varkappa $, and in particular we have
\begin{equation}\label{RboundingQ}
R_t(x)\ge 2\varkappa  Q_t(x).
\end{equation}
We now define the score of the configuration as $M_t:=\sum_{x=1}^N s_x \nu_t(x)$ with
\[
\nu_t(x):=\left\{
\begin{array}{ll}
1+2 Q_t(x) &\mbox{if } Y_t(x)=1,\\
R_t(x)+2 Q_t(x) &\mbox{if } Y_t(x)=0.
\end{array}
\right.
\]
We show that it is a supermartingale up to the extinction time $T_{ext}$, or up to the hitting time $T_{hit}$ of a vertex $x\in\{1,\ldots, \lceil aN \rceil -1\}$. More precisely, considering
\[
t<T_{hit}:=\inf\{s\ge 0 \colon Y_s(x) = 1 \text{ for some }x<\lceil aN\rceil\},
\] we show
$\frac 1 {dt}  \E[M_{t+dt} -M_t\vert \F_t] \le - \rho M_t,$
for some $\rho>0$. We choose $t<T_{hit}$ and observe that in that case we have $Y_t(x)=0$ for all $x<aN$. Now we provide upper bounds on the infinitesimal change of $\nu_t(x)$ for $x\in\{1,\ldots, N\}$, depending on the value of $Y_t(x)$.\\

$\mathbf{ (i)\ Y_t(x)=1 \mbox{ and } x>aN}.$ In this case $\nu_t(x)$ may change due to
\begin{itemize}
	\item {\bf a recovery;} recoveries contribute to the infinitesimal change adding an expression of the form $R_t(x)-1$
	\item {\bf infections;} since $Y_t(x)=1$, the vertex is able to infect vertices whose edges are not revealed. This occurs at rate $\lambda p_{x,y}$ so it contributes to the infinitesimal change the expression
	\[\sum_{y: Y_t(x,y)=0} \lambda p_{x,y} \sfrac{2\lambda}{\kappa_{x,y}^2}\leq 2\lambda^2 \psi(x)\le \sfrac 1 3,
	\]
	where the last inequality follows from~\eqref{cond1} and $x>aN$.
	\item {\bf updates;} finally, edges between $x$ and its revealed neighbours can update at 
rate~$\kappa_{x,y}$ giving an expression of the form
	\[\sum_{y: Y_t(x,y)=1} \kappa_{x,y} \sfrac{-2\lambda}{\kappa_{x,y}^2} = -2 R_t(x).\]
\end{itemize}

When adding all of these terms we finally obtain, using~\eqref{RboundingQ},
\[
\frac 1 {dt} \E[\nu_{t+dt}(x)-\nu_t(x) \vert \F_t]\le 
-\frac 2 3 - R_t(x)\le -\frac 2 3 - 2 \varkappa  Q_t(x).
\]
\smallskip

$\mathbf{(ii)\ Y_t(x)=0}$ In this case the score of $x$ may change due to
\begin{itemize}
	\item {\bf infections coming from a revealed neighbour;} these induce the expression
	\[
	\lambda N_t(x) (1-R_t(x))  \le \lambda N_t(x),
	\]
	where $N_t$ denotes the number of revealed infected neighbours.
	\item {\bf infections coming from unrevealed infected neighbours;} similarly as above, each unrevealed infected neighbour infects $x$ with rate $\lambda p_{x,y}$ so in particular the expression added to the infinitesimal change is 
	\[\sum_{\substack{y:Y_t(x,y)=0\\Y_t(y)=1}}\lambda p_{x,y} \left(1- R_t(x)+\sfrac {2\lambda}{\kappa_{x,y}^2}\right) \le \sum_{\substack{y:Y_t(x,y)=0\\Y_t(y)=1}}\lambda p_{x,y} \left(1+\sfrac {\lambda}{2\varkappa ^2}\right).
	\]
	
	\item {\bf updates;} as before, edges between $x$ and its revealed neighbours can update at rate $\kappa_{x,y}$, adding the main term used to compensate the infections
	\[\sum_{y: Y_t(x,y)=1} -\kappa_{x,y}\left( \sfrac {2\lambda}{\kappa_{x,y}^2}+\sfrac \lambda {\kappa_{x,y}}\right)= -2R_t(x)-\lambda N_t(x).
	\]
\end{itemize}
When adding all of these terms we obtain
\[
\frac 1 {dt} \E[\nu_{t+dt}(x)-\nu_t(x) \vert \F_t]\le 
- 2 R_t(x)+\sum_{\substack{y:Y_t(x,y)=0\\Y_t(y)=1}}\lambda p_{x,y} \left(1+\tfrac {\lambda}{2\varkappa ^2}\right).
\]
\pagebreak[3]
We can consider the whole score, and use the hypothesis $Y_t(x)=0$ for $x<aN$ to write
\begin{align*}
\frac 1 {dt} &  \E[M_{t+dt} -M_t\vert \F_t]  = \sum_x  s_x \frac 1 {dt} \E[\nu_{t+dt}(x)-\nu_t(x) \vert \F_t] \\
& \le - \sum_{x\colon Y_t(x)=1} \left(\sfrac 2 3+ 2\varkappa  Q_t(x)\right) s_x - 2\!\!\!\sum_{x\colon Y_t(x)=0} R_t(x) s_x 
+ \!\!\!\!\sum_{\heap{x,y\colon Y_t(x)=0,}{ Y_t(y)=1}}   \lambda \left(1+\sfrac {\lambda}{2\varkappa ^2}\right) p_{x,y} s_x
\end{align*}
For the last term, we can reverse the role of $x$ and $y$ and obtain
\[
\sum_{x\colon Y_t(x)=1} \sum_{y\colon Y_t(y)=0}   \lambda \left(1+\sfrac {\lambda}{2\varkappa ^2}\right) p_{x,y} s_y\le \sum_{x\colon Y_t(x)=1} \sfrac {s_x} 3,
\]
where we have used a simple sum-integral comparison (recall $p$ and $s$ are nonincreasing in their parameters) and~\eqref{cond2}. Recalling~\eqref{RboundingQ}, we get, for
 $\rho= \min(\frac 13, \varkappa ,\frac {2\varkappa }{\varkappa +1})>0$, that
\[
\frac 1 {dt}  \E[M_{t+dt} -M_t\vert \F_t] \le
- \sum_{x\colon Y_t(x)=1} \left(\sfrac 1 3+ 2\varkappa  Q_t(x)\right) s_x - 2\sum_{x\colon Y_t(x)=0} R_t(x) s_x \le -\rho M_t.
\]
As a first consequence, defining $T=T_{hit}\wedge T_{ext}$, we obtain that $(M_{t\wedge T})_{t\ge 0}$ is a supermartingale. Furthermore, defining $Z_t=\log(1+M_t)+\frac \rho 2 t$, we obtain by Jensen's inequality and the concavity of the logarithmic function on $[1,\infty)$ that, on the event $\{t<T\}$,
\[
\frac 1 {dt}  \E[Z_{t+dt} -Z_t\vert \F_t] \le \rho \left(\sfrac 1 2 - \sfrac {M_t}{1+M_t}\right).
\]
Observing that the score of a configuration with at least one vertex infected is at least $s(1)\ge 1$, we get that the expression above is nonpositive, and thus $(Z_{t\wedge T})_{t\ge 0}$ is a supermartingale.%
\smallskip%

In the case $a=0$, we deduce the first part of Theorem~\ref{teoupper_edge} by observing that in that case the nonnegative supermartingale $(Z_{t\wedge T})$ converges almost surely to $Z_{T_{ext}}=\frac \rho 2 T_{ext}$, so applying the optional stopping theorem gives
\[
\E[T_{ext}]\le \frac 2 \rho Z_0\le \frac 2 \rho \log\bigg(1+ \sum_{x=1}^N s_x\bigg) \le \frac 2 \rho \log\Big(1+ N \int s(x)\de x\Big)= O(\log N),
\]
as needed. Note that $s$ is integrable by on $(0,1]$ by~\eqref{cond2}.
\smallskip

In the case $a>0$, we use the duality of the process and the coupling to deduce
\[
I_N(t)\le \frac1 N \sum_{x=1}^N \P_x(t<T_{ext})\le \frac {\lceil a N\rceil} N+\frac 1 N \sum_{x=\lceil a N\rceil+1}^N \big(\P_x(T_{hit}<T_{ext})+\P_x(t<T)\big),
\]
where $\P_x$ corresponds to the law of the process with initial condition $Y_0=\delta_0$. Using that $(M_{t\wedge T})$ is a supermartingale, we can bound the first term in the sum by
\[
\P_x(T_{hit}<T_{ext})\le \frac {s_x}{s(a)}.
\]
Using the supermartingale $(Z_{t\wedge T})$ and Markov's inequality for the second term we get
\[
\P_x(t<T)\le \frac 1 t \E_x[T]\le \frac 2 {\rho t} \E_x[Z_T] \le \frac 2 {\rho t} \E_x[Z_0]\le \frac 2 {\rho t} \log(1+s_x).
\]
Adding over $x$ and by a simple sum-integral comparison, we get:
\[
I_N(t)\le a+\frac 1 N+ \frac 1{s(a)} \int_a^1 s(y)\de y + \frac 2 {\rho t} \int_a^1 \log(1+s(y))\de y.
\]
Thus~\eqref{resdens} is proved with $\omega=\frac 2 \rho \int_a^1\log(1+s)\, ds$.

\bigskip

\noindent {\bf Acknowledgements:} AL was partially supported by the CONICYT-PCHA/Doctorado nacional/2014-21141160 scholarship, grant GrHyDy ANR-20-CE40-0002, and by the 
\linebreak FONDECYT grant 11221346.


\bigskip


{\scriptsize\renewcommand{\baselinestretch}{0.5}
\noindent
{\bf Emmanuel Jacob}, Ecole Normale Sup\'erieure de Lyon, Unit\'e de Math\'ematiques Pures et Appliqu\'ees,  \linebreak UMR CNRS 5669 , 46, All\'ee d'Italie, 69364 Lyon Cedex 07, France.
\medskip

\noindent
{\bf Amitai Linker}, Departamento de Matem\'aticas, Facultad de Ciencias Exactas, Universidad Andr\'es Bello, \linebreak Sazi\'e 2212, Santiago, Chile.
\medskip

\noindent
{\bf Peter M\"orters}, Universit\"at zu K\"oln,  Mathematisches Institut, Weyertal 86--90, 50931~K\"oln, Germany.\par}

\begin{thebibliography}{10}
\scriptsize{


\bibitem{BB+}
Berger N, Borgs C, Chayes JT, Saberi A (2005) 
On the spread of viruses on the internet.
{\it Proc. 16th Symposium on Discrete Algorithms,}  
(Association for Computing Machinery, New York), pp 301--310.


\bibitem{VHC17}
Can, VH (2017)
Metastability for the contact process on the preferential attachment graph.
{\it Internet Mathematics}  https://doi.org/10.24166/im.08.2017, 45 pp.
                 
\bibitem{CD18}
Cator E and Don H (2018) 
Metastability of the contact process on Erd\H{o}s-R\'enyi and configuration model graphs. arXiv:1804.03753.



\bibitem{CD09}
Chatterjee S and Durrett R (2009)
Contact processes on random graphs with power law degree distributions have critical value zero.
{\it Ann. Probab.\ \bf 37}, 2332--2356.


\bibitem{DL88} 
Durrett R and Liu X (1988) 
The contact process on a finite set. 
{\it Ann. Probab.\ \bf  16}, 1158--1173.

\bibitem{DS88} 
Durrett R, Schonmann R (1988)
The contact process on a finite set II. 
{\it Ann. Probab.\ \bf  16}, 1570--1583.

\bibitem{EM14} 
Eckhoff M and M\"orters, P (2014)
Vulnerability of robust preferential attachment networks.
{\it Electron. J. Probab.}
{\bf 19}, paper no. 57, 47 pp.

\bibitem{GMT05}
Ganesh A, Massoulie L and Towsley D (2005) 
The effect of network topology on the spread of epidemics. 
In: {\it Proceedings IEEE Infocom, Vol. 2, New York, NY,} 1455--1466. 

\bibitem{Hilario}
Hil\'ario M, Ungaretti D, Valesin D, Vares M.E. (2021)
Results on the contact process with dynamic edges or under renewals.
Preprint {\it arXiv:2108.03219}. 

\bibitem{JLM19}
Jacob E, Linker A, and  M\"orters P (2019) 
Metastability of the contact process on fast evolving scale-free networks. 
{\it Ann. Appl. Probab.} {\bf 29}, 2654--2699.

\bibitem{JLM22}
Jacob E, Linker A, and  M\"orters P (2022) 
Metastability of the contact process on slowly evolving scale-free networks. 
{\it In preparation.}  

\bibitem{L85}
Liggett T (1985) 
{\it Interacting Particle Systems}. Springer,  New York.

\bibitem{L99} 
Liggett T (1999) 
{\it Stochastic interacting systems: contact, voter and exclusion processes.}
 Springer,  New York.

\bibitem{MP16}
Menshikov M, Popov S, and Wade A (2016). 
\emph{Non-homogeneous Random Walks: Lyapunov Function Methods for Near-Critical Stochastic Systems}, 
Cambridge University Press. doi:10.1017/9781139208468

\bibitem{MVY13}
Mountford T, Valesin D, and Yao Q (2013).
Metastable densities for the contact process on power law random graphs.
{\it  Electron. J. Probab.\ \bf 13}, 1--36. 

\bibitem{MV16}
Mourrat, JC, and Valesin D (2016).
Phase transition of the contact process on random regular graphs.
{\it  Electron. J. Probab.\ \bf 21}, paper no. 31, 17 pp.



\bibitem{PV01}
Pastor-Satorras R and Vespignani A (2001). Epidemic spreading in scale-free networks. 
{\it Phys. Rev. Letters \ \bf 86}, 3200-3203.


\bibitem{P92}
Pemantle R (1992). The contact process on trees.
{\it Ann. Probab.\ \bf 20}, 2089-2116.



\bibitem{Silva}
da Silva, G.L.L., Oliveira, R.I. and Valesin, D. (2021)
The contact process over a dynamical $d$-regular graph.
Preprint {\it arXiv:2111.11757 }.  


\bibitem{Steif}
Steif J (2009)
A survey of dynamical percolation.
In: {\it Fractal Geometry and Stochastics IV.}
Birkh\"auser Progress in Probability {\bf 61}, 145--174.

}
\end{thebibliography}
\end{document}